\documentclass[a4paper]{amsart}
\usepackage{latexsym,bm,stmaryrd}
\usepackage{amsmath,amsthm,amsfonts,amssymb,mathrsfs,pb-diagram}

\usepackage{microtype}
\usepackage{mathtools}
\usepackage{mathbbol,wasysym}

\usepackage{tikz}
\usetikzlibrary{arrows,decorations.markings}

\advance\textheight by 1.5mm
\usepackage{cite}
\usepackage{graphics}
\usepackage[all]{xy}

\usepackage[enableskew,vcentermath]{youngtab}
\def\tab(#1){\mbox{\small$\young(#1)$}\,}

\let\<=\langle
\let\>=\rangle
\def\({\big(}
\def\){\big)}

\let\Sym=\BS

\newcommand{\N}{\mathbb N}
\newcommand{\Z}{\mathbb Z}

\newcommand{\HH}{\mathscr{H}}
\renewcommand{\H}{\HH}

\DeclareMathAlphabet{\mathpzc}{OT1}{pzc}{m}{it}

\newcommand{\R}[1][n]{\mathscr R^\Lambda_{#1}}

\newcommand{\bi}{\mathbf{i}}
\newcommand{\bj}{\mathbf{j}}

\def\t{{\mathfrak t}}

\def\he{\hat{e}}
\def\hT{\hat{T}}
\def\hX{\hat{X}}
\def\hx{\hat{x}}
\def\hs{\hat{s}}

\newcommand{\lam}{\lambda}
\newcommand{\Lam}{\Lambda}

\newcommand\blam{{\boldsymbol\lambda}}

\DeclareMathOperator\Std{Std}

\def\SStd(#1,#2){\Std^\Lambda_{#2}(#1)}

\renewcommand\t{\mathfrak{t}}

\DeclareMathOperator{\pol}{Pol_\beta}
\DeclareMathOperator{\ppol}{Pol_n}
\def\wpol{\widetilde{\pol}}
\DeclareMathOperator{\height}{ht}

\DeclareMathOperator{\res}{res}
\DeclareMathOperator{\Ker}{Ker}

\DeclareMathOperator{\cha}{char}
\DeclareMathOperator{\ind}{Ind}
\DeclareMathOperator{\image}{Im}

\def\mod{\text{-mod}}

\newcounter{main}
\theoremstyle{plain}

\swapnumbers \numberwithin{equation}{section}
\newtheorem{prop}[equation]{Proposition}
\newtheorem{thm}[equation]{Theorem}
\newtheorem{cor}[equation]{Corollary}
\newtheorem{lem}[equation]{Lemma}
\newtheorem{conj}[equation]{Conjecture}
\theoremstyle{definition}
\newtheorem{dfn}[equation]{Definition}
\newtheorem{dfn2}[equation]{Definition and Theorem}
\theoremstyle{remark}
\newtheorem{rem}[equation]{Remark}

  {\refstepcounter{equation}\trivlist
   \item[\hskip\labelsep\theequation.~\textbf{Example#1}\space]
   \ignorespaces
  }{\unskip\nobreak\hfil%
    \penalty50\hskip2em\hbox{}\nobreak\hfil$\Diamond$%
    \parfillskip=0pt\finalhyphendemerits=0\penalty-100\endtrivlist}

\def\map#1#2{\,{:}\,#1\!\longrightarrow\!#2}

{\catcode`\|=\active
  \gdef\set#1{\mathinner{\lbrace\,{\mathcode`\|"8000%
                                   \let|\midvert #1}\,\rbrace}}
}
\def\midvert{\egroup\mid\bgroup}

\usepackage{tikz}
\usetikzlibrary{matrix}

\newcount\tableauRow\newcount\tableauCol
\newcommand\Tableau[2][-4]{
  \begin{tikzpicture}[scale=0.4,draw/.append style={thick,black},baseline=#1mm]
    \tableauRow=0
    \foreach \Row in {#2} {
       \tableauCol=1
       \foreach\k in \Row {
          \draw(\the\tableauCol,\the\tableauRow)+(-.5,-.5)rectangle++(.5,.5);
          \draw(\the\tableauCol,\the\tableauRow)node{\k};
          \global\advance\tableauCol by 1
       }
       \global\advance\tableauRow by -1
    }
  \end{tikzpicture}
}
\def\Bitab(#1|#2){\Bigg(%
\hspace*{1mm}\Tableau{#1}\hspace*{2mm}\Bigg|\hspace*{2mm}\Tableau{#2}\hspace*{1mm}\Bigg)
}
\def\Tritab(#1|#2|#3){\Bigg(%
  \hspace*{1mm}\Tableau{#1}\hspace*{2mm},\hspace*{2mm}\Tableau{#2}\hspace*{2mm}%
       ,\hspace*{2mm}\Tableau{#3}\hspace*{1mm}\Bigg)
}

\newcommand\ShadedTableau[2][\relax]{
  \begin{tikzpicture}[scale=0.4,draw/.append style={thick,black},baseline=-4mm]
    \ifx\relax#1\relax%
    \else 
      \foreach\box in {#1} { \filldraw[blue!30]\box+(-.5,-.5)rectangle++(.5,.5); }
    \fi
    \newcount\TabRow\newcount\TabCol
    \TabRow=0
    \foreach \Row in {#2} {
       \TabCol=1
       \foreach\k in \Row {
          \draw(\the\TabCol,\the\TabRow)+(-.5,-.5)rectangle++(.5,.5);
          \draw(\the\TabCol,\the\TabRow)node{\k};
          \global\advance\TabCol by 1
       }
       \global\advance\TabRow by -1
    }
  \end{tikzpicture}
}
\newcommand\YoungDiagram[2][\relax]{
  \begin{tikzpicture}[scale=0.5,draw/.append style={thick,black},baseline=-2mm]
    \ifx\relax#1\relax%
    \else 
    \foreach\box in {#1} {
      \filldraw[blue!30]\box rectangle ++(1,1);
    }
    \fi
    \newcount\tableauRow
    \tableauRow=0
    \foreach \diagramCol in {#2} {
       \draw(1,\the\tableauRow)grid ++(\diagramCol,1);
       \global\advance\tableauRow by -1
    }
  \end{tikzpicture}
}
\def\TriDiagram(#1|#2|#3){\Bigg(%
  \hspace*{1mm}\YoungDiagram{#1}\hspace*{2mm},\hspace*{2mm}\YoungDiagram{#2}\hspace*{2mm}%
    ,\hspace*{2mm}\YoungDiagram{#3}\hspace*{1mm}\Bigg)
}

\begin{document}
\title[Modified affine Hecke algebras and quiver Hecke algebras]
{Modified affine Hecke algebras and quiver Hecke algebras of type $A$}


\author{Jun Hu}
 \address{School of Mathematics and Statistics\\
  Beijing Institute of Technology\\
  Beijing, 100081, P.R. China}
  \email{junhu404@bit.edu.cn}

\author{Fang Li}
  \address{Department of Mathematics\\
  Zhejiang University\\
  Hangzhou, 310027, P.R. China}
  \email{fangli@zju.edu.cn}

\begin{abstract} We introduce some modified forms for the degenerate and non-degenerate affine Hecke algebras of type $A$. These are certain subalgebras living inside the inverse limit of cyclotomic Hecke algebras. We construct faithful representations and standard bases for these algebras and give some explicit description of their centers. We show that there are algebra isomorphisms between some generalized Ore localizations of these modified affine Hecke algebras and of the quiver Hecke algebras of type $A$. As an application, we show that the center conjecture for the cyclotomic quiver Hecke algebra of type $A$ holds if and only if the center conjecture for the cyclotomic Hecke algebra of type $A$ holds.
\end{abstract}

\maketitle


\section{Introduction}

Quiver Hecke algebras and their cyclotomic quotients are some remarkable $\Z$-graded algebra which are introduced by Khovanov and Lauda \cite{KhovLaud:diagI}, and by Rouquier \cite{Rou0}. These algebra have been a hot topic in recent years as they play an important role in the study of categorification of quantum groups. A recent significant theorem of Brundan and Kleshchev \cite[Theorem~1.1]{BK:GradedKL} says that the group algebra of the symmetric group (and more generally any cyclotomic Hecke algebra of type $A$) is isomorphic to a cyclotomic quiver Hecke algebra, which shows in particular that these cyclotomic Hecke algebras are $\Z$-graded. The focus of the paper is to try to remove the word ``cyclotomic" and lift to give an isomorphism from certain modified form of the affine Hecke algebra in type $A$ to certain modified form of the quiver Hecke algebra in type $A$.

We start with some basic notations and definitions. Let $\Z$ be the set of integers and $\N$ the set of non-negative integers. Let $e\in\{0,2,3,4,\dots\}$ be a fixed integer and $I:=\Z/e\Z$. Let $\Gamma_e$ be the quiver with vertex set $I$ and edges $i\longrightarrow i+1$, for all $i\in I$. Following \cite[Chapter~1]{Kac},  attach to $\Gamma_e$ the standard Lie theoretic data of a Cartan matrix $(a_{ij})_{i,j\in I}$, simple roots $\set{\alpha_i|i\in I}$, simple coroots $\set{\alpha_i^\vee|i\in I}$, fundamental weights $\set{\Lambda_i|i\in I}$ and positive root lattice $Q^+=\bigoplus_{i\in I}\N\alpha_i$. For each $\beta=\sum_{i\in I}k_i\alpha_i\in Q^+$ we define $\height(\beta):=\sum_{i\in I}k_i$. For each $n\in\N$, we set $Q_n^{+}:=\{\beta\in Q^{+}|\height(\beta)=n\}$.

Let $P^+=\bigoplus_{i\in I}\N\Lambda_i$ be the set of dominant weights, $\ell\in\N$ and $\kappa_1,\dots,\kappa_\ell\in\Z/e\Z$. We define \begin{equation}\label{Lam}
\Lambda:=\Lambda_{\kappa_1}+\dots+\Lambda_{\kappa_\ell}\in P^+ £¬
\end{equation}
and call $\ell$ the level of $\Lam$. In this paper we shall consider both the non-degenerate and the degenerate settings as follows.

In the non-degenerate setting, we assume $K$ is a field, $1\neq q\in K^{\times}$ such that either $e$ is the minimal positive integer $k$ which satisfies that $1+q+q^2+\dots+q^{k-1}=0$; or $e=0$ when there is no such positive integer $k$. In this case, let $\H_n^{\Lambda}(q)$ be the  {\bf non-degenerate cyclotomic Hecke algebra} of type $A$ over $K$ with Hecke parameter $q$ and cyclotomic parameters $q^{\kappa_1},\dots,q^{\kappa_\ell}$ (cf. \cite{AK1}, \cite{Ch}). By definition, $\H^{\Lambda}_n(q)$ is generated by $T_0, T_1,\dots,T_{n-1}$ which satisfy the following relations:
$$\begin{matrix}
(T_0-q^{\kappa_1})\cdots (T_0-q^{\kappa_\ell})=0,\\
T_0T_1T_0T_1=T_1T_0T_1T_0,\\
(T_i-q)(T_i+1)=0,\,\,T_iT_k=T_kT_i,\quad 1\leq i<n, 0\leq k<n, |i-k|>1,\\
 T_iT_{i+1}T_i=T_{i+1}T_iT_{i+1},\,\,\quad 1\leq i\leq n-2 .
\end{matrix} $$
Let $L_1:=T_0$ and $L_{i+1}=q^{-1}T_iL_iT_i$ for $1\leq i<n$. The elements $L_1,L_2,\dots,L_n$ are called the {\bf Jucys--Murphy operators} of $\H_n^{\Lambda}(q)$.

In the degenerate setting, we assume that $K$ is field with $\cha K=e$. In particular, this means either $e=0$ or $e$ is a prime number. In this case, let $H_n^{\Lambda}$ be the  {\bf degenerate cyclotomic Hecke algebra} of type $A$ over $K$ with cyclotomic parameters $\kappa_1\cdot 1_K,\dots,\kappa_\ell\cdot 1_K$ (cf. \cite{G}, \cite{K}), where $1_K$ is the identity element of $K$. By definition, $H^{\Lambda}_n$ is generated by $s_1,\dots,s_{n-1},L_1,\dots,L_n$ which satisfy the following relations:
$$\begin{matrix}
(L_1-\kappa_1\cdot 1_K)\cdots (L_1-\kappa_\ell\cdot 1_K)=0,\\
s_i^2=1,\,\, s_is_k=s_ks_i,\quad \text{for}\,\, 1\leq i,k<n, |i-k|>1,\\
 s_is_{i+1}s_i=s_{i+1}s_is_{i+1},\,\,\quad \text{for}\,\, 1\leq i\leq n-2,\\
L_iL_k=L_kL_i,\,\,s_iL_l=L_ls_i,\,\,\,\,\text{for}\,\, 1\leq i<n, 1\leq k,l\leq n, l\neq i,i+1,\\
L_{i+1}=s_iL_is_i+s_i,\,\,\,\,\text{for}\,\, 1\leq i<n .
\end{matrix} $$
The elements ${L}_1,{L}_2,\dots,{L}_n$ are called the {\bf Jucys--Murphy operators} of $H_n^{\Lambda}$. Note that in general $L_1s_1L_1s_1\neq s_1L_1s_1L_1$ in $H_n^{\Lambda}$.

For any $\beta\in Q_n^+$, we define $$
I^{\beta}:=\{\bi=(i_1,\dots,i_n)\in I^n|\alpha_{i_1}+\dots+\alpha_{i_n}=\beta\} .
$$
For each $\bi\in I^n$, Brundan and Kleshchev have introduced in \cite[\S3.1, \S4.1]{BK:GradedKL} an idempotent in $\H_n^{\Lambda}(q)$ and an idempotent in $H_n^{\Lambda}$ which (by abuse of notations) are both denoted by $e(\bi)$. We define $e(\beta):=\sum_{\bi\in I^\beta}e(\bi)$. Then $e(\beta)$ is either $0$ or a block idempotent of $\H_n^{\Lambda}(q)$ (resp., of $H_n^{\Lambda}$).


By \cite{LM:AKblocks} and \cite{Brundan:degenCentre}, both the blocks of $\H_n^{\Lambda}(q)$ and of $H_n^{\Lambda}$ are parameterized by $\{\beta\in I^n|e(\beta)\neq 0\}$. For any block idempotents $e(\beta)$ of $\H_n^{\Lambda}(q)$ and of $H_n^{\Lambda}$, we define
\begin{equation}\label{block1}
\H_{\beta}^{\Lambda}(q):=e(\beta)\H_n^{\Lambda}(q),\quad\,\,H_{\beta}^{\Lambda}:=e(\beta)H_n^{\Lambda} ,
\end{equation}
which are the block subalgebras corresponding to $\beta$ of  $\H_n^{\Lambda}(q)$ and of $H_n^{\Lambda}$ respectively.

Let $\mathscr{R}_{\beta}:=\mathscr{R}_{\beta}(\Gamma_e)$ be the quiver Hecke algebra associated to $\Gamma_e$ and $\beta\in Q_n^{+}$ introduced by Khovanov and Lauda \cite{KhovLaud:diagI}, and by Rouquier \cite{Rou0}.
By definition, $\mathscr{R}_{\beta}$ is generated by the elements $\{\psi_1,\dots,\psi_{n-1}\} \cup   \{ y_1,\dots,y_n \} \cup \set{e(\bi)|\bi\in I^\beta}$. These elements are usually called KLR generators, and each $e(\bi)$ is called a KLR idempotent. We refer the readers to Section 2 for a list of defining relations. Let $\R[\beta]:=\R[\beta](\Gamma_e)$ be the quotient of $\mathscr{R}_{\beta}$ by the two-sided ideal generated by \begin{equation}\label{cyclotomic2}
y_1^{\<\Lambda,\alpha_{i_1}^\vee\>}e(\bi),\,\,\,\bi\in I^\beta .
\end{equation}
We call the algebra $\R[\beta]$ the {\bf cyclotomic quiver Hecke algebra} of type $A$ associated to $\beta$ and $\Lambda$. When the context is clear, we use the same letters to denote both the KLR generators of $\mathscr{R}_\beta$ and their canonical image in $\R[\beta]$, and use the same letter $e(\bi)$ to denote both the idempotents of $\H_n^{\Lambda}(q)$ and of $H_n^{\Lambda}$ and the KLR idempotent of $\R[\beta]$.

\begin{thm} [\protect{Brundan-Kleshchev~\cite[Theorem~1.1]{BK:GradedKL}}] \label{BKiso} Let $\beta\in Q_n^{+}$ and  $\H_\beta^{\Lambda}\in\{\H_\beta^{\Lambda}(q),H_\beta^{\Lambda}\}$.
 Then there is an isomorphism of $K$-algebras $\theta^\Lambda:  \R[\beta]\cong \H_\beta^{\Lambda}$ that sends $e(\bi)\mapsto e(\bi)$, for all $\bi\in I^\beta$ and \begin{align*}
y_re(\bi)&\mapsto\begin{cases}
  (1-q^{-i_r}L_r)e(\bi),&\text{if $\H_\beta^{\Lambda}=\H_\beta^{\Lambda}(q)$},\\[5mm]
  (L_r - i_r)e(\bi),&\text{if $\H_\beta^{\Lambda}=H_\beta^{\Lambda}$}.
\end{cases}\\
\psi_k e(\bi)&\mapsto\begin{cases}(T_k+P_k(\bi))Q_k(\bi)^{-1}e(\bi),&\text{if $\H_\beta^{\Lambda}=\H_\beta^{\Lambda}(q)$,}\\
(s_k+P_k(\bi))Q_k(\bi)^{-1}e(\bi), &\text{if $\H_\beta^{\Lambda}=H_\beta^{\Lambda}$,}
\end{cases}
\end{align*}
where $1\le r\le n$, $1\le k<n$, $P_k(\bi),Q_k(\bi)\in K[y_k,y_{k+1}]$ are certain polynomials introduced in
\cite[(3.22),(3.27--3.29),(4.27),(4.33--4.35)]{BK:GradedKL}. In particular, $\R[\beta]\neq 0$ if and only if $\sum_{\bi\in I^\beta}e(\bi)\neq 0$ in $\R[\beta]$, and $\H_\beta^{\Lambda}\neq 0$ if and only if $e(\beta)\neq 0$ in $\H_\beta^\Lambda$.
\end{thm}

Let $\H_n(q)$ and $H_n$ be the non-degenerate and the degenerate type $A$ affine Hecke algebras respectively. Then $\H_n^{\Lambda}(q)$ and $H_n^{\Lambda}$ are isomorphic to the quotients of  $\H_n(q)$ and $H_n$ by the two-sided ideals generated by $(X_1-q^{\kappa_1})\cdots (X_1-q^{\kappa_\ell})$ and $(x_1-\kappa_1\cdot 1_K)\cdots (x_1-\kappa_\ell\cdot 1_K)$ respectively. We refer the readers to Section 2 for unexplained notations and more details. There are many similarities on the structure and representation theory between the algebras $\H_\beta^{\Lambda}\in\{\H_\beta^{\Lambda}(q),H_\beta^{\Lambda}\}$ and $\R[\beta]$, and between the algebras $\H_n\in\{\H_{n}(q), H_{n}\}$ and $\mathscr{R}_n:=\bigoplus_{\beta\in Q_n^{+}}\mathscr{R}_\beta$. Both algebras $\mathscr{R}_n$ and $\R[\beta]$ are $\Z$-graded and many results on the representations of $\H_n$ and $\H_\beta^{\Lambda}$ have found their $\Z$-graded versions for the algebras $\mathscr{R}_n$ and $\R[\beta]$, see \cite{G}, \cite{K}, \cite{K2010}, \cite{LV}, \cite{M2015} and the references therein. In view of Brundan--Kleshchev's isomorphism Theorem \ref{BKiso}, one can roughly think of $\R[\beta]$ as a $\Z$-graded analogue of $\H_\beta^{\Lambda}$.

It is natural to ask if there is a similar isomorphism directly on the level of the affine Hecke algebra $\H_n$ and the quiver Hecke algebra $\mathscr{R}_n$. In this paper, for each $\beta\in Q_n^+$, we introduce some modified forms $\widehat{\H}_\beta\in\{\widehat{\H}_\beta(q),\widehat{H}_\beta\}$ for both the non-degenerate and the degenerate type $A$ affine Hecke algebras.
We construct explicit $K$-algebra isomorphisms $\theta: \widetilde{\mathscr{R}}_{\beta}\cong\widetilde{\H}_\beta(q)$,
$\theta': \widetilde{\mathscr{R}}'_{\beta}\cong\widetilde{H}_\beta$ between certain generalized Ore localizations $\widetilde{\H}_\beta(q), \widetilde{H}_\beta$, $\widetilde{\mathscr{R}}_{\beta}$ (or $\widetilde{\mathscr{R}}'_{\beta}$) of $\widehat{\H}_\beta(q)$, $\widehat{H}_\beta$ and $\mathscr{R}_\beta$ respectively, generalizing Brundan--Kleshchev's isomorphism between $\H_\beta^{\Lambda}$ and $\R[\beta]$. These modified affine Hecke algebras $\widehat{\H}_\beta$ and their generalized Ore localizations are both subalgebras of the inverse limit of the cyclotomic Hecke algebras $\mathscr{H}_\beta^{\Lambda}$. One of the key ingredients in our argument is a lift of the KLR idempotent $e(\bi)$ in the inverse limits of the cyclotomic Hecke algebras $\mathscr{H}_\beta^{\Lambda}$. The algebras $\widehat{\H}_\beta$ can be regarded as certain idempotent completions of the ``block" of the type $A$ affine Hecke algebra $\H_n$, which may remind the reader of Lusztig's definition of modified forms of quantized enveloping algebras \cite[Chapter 23]{Lusz2}. They are closely related to the original type $A$ affine Hecke algebra $\H_n$ in that every finite dimensional module over $\H_n$ which belongs to the block labelled by $\beta$ naturally becomes a module over $\widehat{\H}_\beta$ and this correspondence gives rise to a categorical  equivalence. Many classical results (including faithful representations, standard bases and description of the centers) for the original affine Hecke algebras are generalized to these  modified affine Hecke algebras, see Proposition \ref{rep}, Theorem \ref{basis1} and Theorem \ref{center0}.

There is a famous conjecture for $\HH_n^\Lam(q)$ which asserts that every central element of $\HH_n^\Lam(q)$ can be written as a symmetric polynomial in its Jucys-Murphy operators $L_1,\dots,L_n$. Another more recent conjecture for $\R[\beta]$ asserts that every central element of $\R[\beta]$ can be written as a symmetric element in its KLR $y$ generators $y_1,\dots,y_n$ and KLR idempotents $e(\bi), \bi\in I^\beta$. As one of the main application of our result, we show that the center conjecture for $\HH_n^\Lam(q)$ holds if and only if the center conjecture for $\R[\beta]$ holds for each $\beta\in Q_n^+$. Using the isomorphisms $\theta,\theta'$ one can identify the convolution products in the category of finite dimensional modules over affine Hecke algebras with the convolution product in the category of finite dimensional modules over quiver Hecke algebras. As further applications of our main results, we derive an equivalence of categories for quiver Hecke algebras and tensor products of its non-unital quiver Hecke subalgebras under certain conditions as well as a simplicity result for the convolution products of simple modules over quiver Hecke algebras.

We note that Rouquier has presented a similar isomorphism (without proof) between certain different localized from of $\H_n$ and $\mathscr{R}_n$ in the preprint
\cite[3.15, 3.18]{Rou1}. Our isomorphism $\theta,\theta'$ are different with Rouquier's isomorphisms and the algebra $\widehat{\H}_\beta$ we introduced in this paper does not appear in \cite{Rou1}.

The content of this paper is organised as follows. In Section 2, we recall some preliminary knowledge about the type $A$ affine Hecke algebras, the type $A$ quiver Hecke algebras, and their cyclotomic quotients. In Section 3, we introduce the modified forms of the affine Hecke algebras of type $A$. We construct faithful representations, standard bases and describe the centers for these modified affine Hecke algebras. We also introduce some generalized Ore localization for these modified affine Hecke algebras and quiver Hecke algebras. The general definition of generalized Ore localization is given in the appendix of this paper. The main results (Theorem \ref{mainthm0a} and \ref{mainthm0b}) are given in Section 4, where we set up isomorphisms between these generalized Ore localization for modified affine Hecke algebras and the generalized Ore localization for quiver Hecke algebras. The main idea in our approach is to construct a nice lift of the KLR idempotent $e(\bi)$ in the inverse limits of the cyclotomic Hecke algebras $\mathscr{H}_\beta^{\Lambda}$ and to embed the generalized Ore localizations for these modified affine Hecke algebras (resp., for the quiver Hecke algebras) into the inverse limits of cyclotomic Hecke algebras (resp., of the cyclotomic quiver Hecke algebras), see Lemma \ref{keyGeLilemma} and Corollary \ref{keycor1}. In Section 5 we give some applications of Theorem \ref{mainthm0a} and \ref{mainthm0b}. We show (Theorem \ref{mainapp1}) that the center conjecture for $\HH_n^\Lam(q)$ holds if and only if the center conjecture for $\R[\beta]$ holds for any $\beta\in Q_n^+$. We also obtain (Corollary \ref{app1}) an equivalence of categories for quiver Hecke algebras and tensor products of its non-unital quiver Hecke subalgebras under certain conditions and a simplicity result (Corollary \ref{app2}) for the convolution products of simple modules over quiver Hecke algebras.

\section*{Acknowledgements}

 The first author is supported by the National Natural Science Foundation of China (No. 11525102). The second author is supported by the National Natural Science Foundation of China (No. 11271318 and No. 11671350) and the Zhejiang Provincial Natural Science Foundation of China (No. LY19A010023).

\section{Preliminary}

In this section, we shall recall some preliminary results on the non-degenerate and the degenerate affine Hecke algebras of type $A$, and the quiver Hecke algebras associated to $\Gamma_e$ and $\beta\in Q^+$ and their cyclotomic quotients. In particular, we shall fix some choices of the polynomials $P_k(\bi),Q_k(\bi)\in K[y_k,y_{k+1}]$ in the construction of Brundan--Kleshchev's isomorphism Theorem \ref{BKiso}.

Recall that $\ell,n\in\N$, $e\in\{0,2,3,\dots\}$ and $\kappa_1,\dots,\kappa_\ell\in I:=\Z/e\Z$. Let $K$ be a field. In the non-degenerate setting, we assume that $1\neq q\in K^{\times}$ such that $e$ is the minimal positive integer $k$ satisfying $1+q+q^2+\dots+q^{k-1}=0$; or $e=0$ when there is no such positive integer $k$. Let $\H_n(q)$ be the {\bf non-degenerate type $A$ affine Hecke algebra} over $K$. By definition, $\H_n(q)$ is the unital associative $K$-algebra with generators $T_1,\dots,T_{n-1}$, $X_1^{\pm 1},\dots,X_n^{\pm 1}$ and relations:
\begin{align}
(T_i-q)(T_i+1)&=0,\,\,\quad 1\leq i<n, \label{1} \\
 T_iT_{i+1}T_i&=T_{i+1}T_iT_{i+1},\,\,\quad 1\leq i\leq n-2,\label{2}\\
  T_iT_k&=T_kT_i,\,\,\quad |i-k|>1, \label{3}\\
 X_i^{\pm 1}X_k^{\pm 1}&=X_k^{\pm 1}X_i^{\pm 1},\,\,\quad 1\leq i,k\leq n, \label{4}\\
 X_kX_k^{-1}&=1=X_k^{-1}X_k,\,\, \quad 1\leq k\leq n, \label{5}\\
 T_iX_k&=X_kT_i,\,\,\quad k\neq i,i+1, \label{6}\\
 X_{i+1}&=q^{-1}T_iX_{i}T_i,\,\,\quad 1\leq i<n . \label{7}
\end{align}
Note that one can also replace the last relation above with the following: \begin{align}
X_{i+1}T_i=T_iX_i+(q-1)X_{i+1},\,\,\quad 1\leq i<n . \label{7a}
\end{align}

The non-degenerate type $A$ cyclotomic Hecke algebra $\H_n^{\Lambda}(q)$ introduced in Section 1 is isomorphic to the quotient of $\H_n(q)$ by the two-sided ideal generated by \begin{equation}\label{cyclotomic}
\bigl(X_1-q^{\kappa_1}\bigr)\bigl(X_1-q^{\kappa_2}\bigr)\dots \bigl(X_1-q^{\kappa_\ell}\bigr) .
\end{equation}
Under this isomorphism,  $T_0$ is identified with the image of $X_1$ in $\H_n^{\Lambda}(q)$, and each $L_i$ is identified with the image of $X_i$ in $\H_n^{\Lambda}(q)$ for $1\leq i\leq n$.
For each $1\leq j<n$, we still use $T_j$ to denote the image of $T_j$ in $\H_n^{\Lambda}(q)$.

In the degenerate setting, we assume that $\cha K=e$ and hence either $e=0$ or $e$ is a prime number. Let $H_n$ be the {\bf degenerate type $A$ affine Hecke algebra}  over $K$. By definition, $H_n$ is the unital associative $K$-algebra with generators $s_1,\dots,s_{n-1}$, $x_1,\dots,x_n$ and relations:
\begin{align}
s_i^2&=1,\,\,\quad 1\leq i<n, \label{8} \\
 s_is_{i+1}s_i&=s_{i+1}s_is_{i+1},\,\,\quad 1\leq i\leq n-2, \label{9}\\
 s_is_k&=s_ks_i,\,\,\quad |i-k|>1, \label{10}\\
 x_i x_k&=x_k x_i,\,\, \quad 1\leq i,k\leq n, \label{11}\\
 s_ix_k&=x_ks_i,\,\,\quad k\neq i,i+1, \label{12}\\
 x_{i+1}&=s_ix_{i}s_i+s_i,\,\,\quad 1\leq i<n, \label{13}.
\end{align}
Note that one can also replace the last relation above with the following: \begin{align}
x_{i+1}s_i=s_ix_i+1,\,\,\quad 1\leq i<n . \label{13a}
\end{align}

The degenerate type $A$ cyclotomic Hecke algebra $H_n^{\Lambda}$ introduced in Section 1 is isomorphic to the quotient of $H_n$ by the two-sided ideal generated by \begin{equation}\label{cyclotomic11}
\bigl(x_1-{\kappa_1}\cdot 1_K\bigr)\bigl(x_1-{\kappa_2}\cdot 1_K\bigr)\dots \bigl(x_1-{\kappa_\ell}\cdot 1_K\bigr) .
\end{equation}
Under this isomorphism, each $L_i$ is identified with the image of $x_i$ in $H_n^{\Lambda}$ for $1\leq i\leq n$.
For each $1\leq j<n$, we still use $s_j$ to denote the image of $s_j$ in $H_n^{\Lambda}$. Inside both $H_n$ and $H_n^{\Lambda}$, the subalgebra generated by $s_1,\dots,s_{n-1}$ is isomorphic to the symmetric group algebra associated to the symmetric group $\Sym_n$ on $n$ letters (with $s_r$ being identified with the basic permutation $(r,r+1)$ for each $r$). Similarly, the subalgebra of $\HH_n(q)$ (resp., of $\HH_n^\Lam(q)$) generated by $T_1,\dots,T_{n-1}$ is isomorphic to the Iwahori--Hecke algebra associated to the symmetric group $\Sym_n$.

Let $\{t_k|1\leq k\leq n\}$ be a set of $n$ algebraically independent indeterminates over $K$. Let $\mathcal{P}_n:=K[t_1^{\pm 1},\dots,t_n^{\pm 1}]$ and $P_n:=K[t_1,\dots,t_n]$. Clearly there is a natural left action of $\Sym_n$ on $I^n$, $\mathcal{P}_n$ and $P_n$ respectively. More explicitly, if $1\leq r<n$, $\bi =(i_1,\dots,i_n)\in I^n$ and $f=f(t_1^{\pm 1},\dots,t_n^{\pm1})\in\mathcal{P}_n$, then $$\begin{aligned}
s_r\bi &=(i_1,\dots,i_{r-1},i_{r+1},i_r,i_{r+2},\dots,i_n),\\
s_r f &=f(t_1^{\pm 1},\dots,t_{r-2}^{\pm 1},t_{r+1}^{\pm 1},t_r^{\pm 1},t_{r+2}^{\pm 1},\dots,t_n^{\pm1}).
\end{aligned}$$

For any $f\in \mathcal{P}_n$, $g\in P_n$, $1\leq r<n$ and $1\leq k\leq n$, we define \begin{equation}\label{affAct1}\left\{\begin{aligned}
 X_k^{\pm 1}\ast f:&=t_k^{\pm 1} f,\\
 T_r\ast f:&=\bigl(t_{r+1}-qt_r\bigr)\frac{s_r(f)-f}{t_{r+1}-t_r}+qf,
\end{aligned}\right.
\end{equation}
and \begin{equation}\label{affAct2}\left\{\begin{aligned}
 x_k\ast g:&=t_k g,\\
 s_r\ast g:&=-\frac{s_r(g)-g}{t_{r+1}-t_r}+s_r(g),
\end{aligned}\right.
\end{equation}

If $w\in\Sym_n$ then the length of $w$ is $$
\ell(w):=\min\{k\in\N|\text{$w=s_{i_1}\dots s_{i_k}$ for some $1\leq i_1,\dots,i_k<n$}\}.
$$
 Let $w_{0,n}$ be the unique longest element in $\Sym_n$. It is well-known that $\ell(w_{0,n})=n(n-1)/2$.
If $w=s_{i_1}\dots s_{i_k}$ with $k=\ell(w)$ then $s_{i_1}\dots s_{i_k}$ is a reduced expression for $w$. In this case, we define $T_w:=T_{i_1}\dots T_{i_k}$. The braid relations (\ref{2}), (\ref{3}) ensure that $T_w$ does not depend on the choice of the reduced expression of $w$.
The following results are well known, see \cite{Mac}.

\begin{lem}\label{affAct1b} Let $w\in\Sym_n$ and $1\leq r\leq n$. Then $$
T_wX_r=X_{w(r)}T_w+\sum_{w>u\in\Sym_n}g_u(X_1,\dots,X_n)T_{u},
$$
where $g_u(X_1,\dots,X_n)\in K[X_1,\dots,X_n]$ for each $u$, ``$>$" is the Bruhat partial order on $\Sym_n$. A similar statement also holds if we replace $T_w, X_r$ by $w, x_r$ respectively.
\end{lem}

\begin{proof} This follows from an induction on $\ell(w)$.
\end{proof}

\begin{lem} \label{faithful1} The above rules (\ref{affAct1}), (\ref{affAct2}) extend uniquely to a faithful representation $\rho_q$ of $\H_n(q)$ on $\mathcal{P}_n$ as well as a faithful representation $\rho_1$ of $H_n$ on $P_n$.
\end{lem}

\begin{lem} \label{basis2aff} The following set $$
\bigl\{X_1^{a_1}\dots X_n^{a_n}T_w\bigm|w\in\Sym_n, a_1,\dots,a_n\in\Z\bigr\}
$$
is a $K$-basis of $\H_n(q)$. Similarly, the following set $$
\bigl\{x_1^{a_1}\dots x_n^{a_n}w\bigm|w\in\Sym_n, a_1,\dots,a_n\in\N\bigr\}
$$
is a $K$-basis of $H_n$.
\end{lem}

Let $\ast$ be the $K$-algebra anti-isomorphism of $\H_n(q)$ which is defined on generators by $T_i^{\ast}:=T_i$, $X_k^{\ast}:=X_k$ for $1\leq i<n, 1\leq k\leq n$. By abuse of notations, we also use $\ast$ to denote
the $K$-algebra anti-isomorphism of $H_n$ which is defined on generators by $s_i^{\ast}:=s_i$, $x_k^{\ast}:=x_k$ for $1\leq i<n, 1\leq k\leq n$. Applying the anti-isomorphism $\ast$, we see that the set $\bigl\{T_wX_1^{a_1}\dots X_n^{a_n}\bigm|w\in\Sym_n, a_1,\dots,a_n\in\Z\bigr\}$
is another $K$-basis of $\H_n(q)$, and the set $\bigl\{wx_1^{a_1}\dots x_n^{a_n}\bigm|w\in\Sym_n, a_1,\dots,a_n\in\N\bigr\}$ is another $K$-basis of $H_n$.

Note that the subalgebra of $\H_n(q)$ generated by $X_1^{\pm 1},\dots,X_n^{\pm 1}$ is canonically isomorphic to the Laurent polynomial $K$-algebra $\mathcal{P}_n$, while the subalgebra of $H_n$ generated by $x_1,\dots,x_n$ is canonically isomorphic to the polynomial $K$-algebra ${P}_n$.

\begin{lem} [Bernstein] \label{center1} The center of $\H_n(q)$ is equal to the set of symmetric Laurent polynomials in $X_1^{\pm 1},\dots,X_n^{\pm 1}$, while the center of $H_n$ is equal to the set of symmetric polynomials in $x_1,\dots,x_n$.
\end{lem}

Let ${\H}^+_n(q)$ be the $K$-subalgebra of ${\H}_n(q)$ generated by $T_1,\dots,T_{n-1},X_1,\dots,X_n$. Then the following set \begin{equation}\label{pluspart}
\Bigl\{X_1^{a_1}X_2^{a_2}\dots X_n^{a_n}T_w\Bigm|w\in\Sym_n, a_1,\dots,a_n\in\N\Bigr\}
\end{equation}
is a $K$-basis of ${\H}^{+}_n(q)$. Furthermore, for any $(i_1,\dots,i_n)\in I^n$, the following set \begin{equation}\label{pluspart2}
\Bigl\{(X_1-q^{i_1})^{a_1}(X_2-q^{i_2})^{a_2}\dots (X_n-q^{i_n})^{a_n}T_w\Bigm|w\in\Sym_n, a_1,\dots,a_n\in\N\Bigr\}
\end{equation}
is a $K$-basis of ${\H}^{+}_n(q)$ too.

\begin{lem} \label{plusversion} The $K$-algebra ${\H}^+_n(q)$ is isomorphic to the abstract $K$-algebra defined by generators $T_1,\dots,T_{n-1},X_1,\dots,X_{n}$
and relations (\ref{1}), (\ref{2}), (\ref{3}), (\ref{6}), (\ref{7}) together with the relations $X_iX_k=X_kX_i$, $\forall\,1\leq i,k\leq n$.
\end{lem}

\begin{proof} Let $\H^+$ be the abstract $K$-algebra which is defined by generators $T_1,\dots,T_{n-1}$, $X_1,\dots,X_{n}$
and relations (\ref{1}), (\ref{2}), (\ref{3}), (\ref{6}), (\ref{7}) together with the relations $X_iX_k=X_kX_i$, $\forall\,1\leq i,k\leq n$. To prove that $\H^+\cong {\H}^+_n(q)$, it suffices to show that
the elements in $\H^+$ which are of the form  (\ref{pluspart}) form a $K$-basis of $\H^+$.

It is easy to see that the elements in $\H^+$ which are of the form  (\ref{pluspart}) generates $\H^+$ as a $K$-linear space. Moreover, the following formulae $$
X_k\cdot f:=t_k f,\quad\,\, T_r\cdot f:=\bigl(t_{r+1}-qt_r\bigr)\frac{s_r(f)-f}{t_{r+1}-t_r}+qf,
$$  also defines a representation $\rho^+_q$ of $\H^+$ on $P_n$. Using this representation $\rho^+_q$ it is easy to check  that the elements in $\H^+$ which are of the form  (\ref{pluspart}) are $K$-linearly independent and hence form a $K$-basis of $\H^+$ and
$\rho^+_q$ is a faithful polynomial representation of $\H^+\cong {\H}^+_n(q)$.
\end{proof}

\begin{cor} \label{centerPosi1} The center $Z(\H^+_n(q))$ of $\H^+_n(q)$ is equal to the set of symmetric polynomials in $X_1,\dots,X_n$.
\end{cor}

\begin{dfn}\label{D:QuiverRelations}
  Suppose $\beta\in Q_n^{+}$. Define {\bf the type $A$ quiver Hecke algebra} $\mathscr{R}_\beta$ to be the unital associative $K$-algebra with generators
  $$\{\psi_1,\dots,\psi_{n-1}\} \cup   \{ y_1,\dots,y_n \} \cup \set{e(\bi)|\bi\in I^\beta} $$
  and relations
  \bgroup
      \setlength{\abovedisplayskip}{1pt}
      \setlength{\belowdisplayskip}{1pt}
  \begin{align*}
         e(\bi) e(\bj) &= \delta_{\bi\bj} e(\bi),
    & \sum_{\bi\in I^\beta}e(\bi)=1,& & &\\
    y_r e(\bi) &= e(\bi) y_r,
    &\psi_r e(\bi)&= e(s_r\bi) \psi_r,
    &y_r y_s &= y_s y_r,
  \end{align*}
  \begin{align*}
    \psi_r y_{r+1} e(\bi)&=(y_r\psi_r+\delta_{i_ri_{r+1}})e(\bi),&
    y_{r+1}\psi_re(\bi)&=(\psi_r y_r+\delta_{i_ri_{r+1}})e(\bi),\\
    \psi_r y_s  &= y_s \psi_r,&&\text{if }s \neq r,r+1,\\
    \psi_r \psi_s &= \psi_s \psi_r,&&\text{if }|r-s|>1,\notag
\end{align*}
\begin{align*}
  \psi_r^2e(\bi) &= \begin{cases}
       0,&\text{if }i_r = i_{r+1},\\
       (y_{r+1}-y_r)e(\bi),&\text{if  }i_r\rightarrow i_{r+1},\\
       (y_r - y_{r+1})e(\bi),&\text{if }i_r\leftarrow i_{r+1},\\
       (y_{r+1} - y_{r})(y_{r}-y_{r+1}) e(\bi),&\text{if }i_r\rightleftarrows i_{r+1}\\
      e(\bi),&\text{otherwise},
\end{cases}\\
\psi_{r}\psi_{r+1} \psi_{r} e(\bi) &= \begin{cases}
    (\psi_{r+1} \psi_{r} \psi_{r+1} +1)e(\bi),\hspace*{7mm} &\text{if }i_r=i_{r+2}\rightarrow i_{r+1} ,\\
  (\psi_{r+1} \psi_{r} \psi_{r+1} -1)e(\bi), &\text{if }i_r=i_{r+2}\leftarrow i_{r+1},\\
  \rlap{$\big(\psi_{r+1} \psi_{r} \psi_{r+1} +y_r -2y_{r+1}+y_{r+2}\big)e(\bi)$,}\\
           &\text{if }i_r=i_{r+2} \rightleftarrows i_{r+1},\\
  \psi_{r+1} \psi_{r} \psi_{r+1} e(\bi),&\text{otherwise.}
\end{cases}
\end{align*}
\egroup
for $\bi,\bj\in I^\beta$ and all admissible $r$ and $s$.
\end{dfn}

Let $\R[\beta]$ be the quotient of $\mathscr{R}_\beta$ by the two-sided ideal generated by \begin{equation}\label{cyclotomic22}
y_1^{\<\Lambda,\alpha_{i_1}^\vee\>}e(\bi),\,\,\,\bi\in I^\beta .
\end{equation}
The algebra $\R[\beta]$ is called the type $A$ {\bf cyclotomic quiver Hecke algebra} associated to $\beta$ and $\Lambda$.

Let $\bi\in I^n$ and $r$ be an integer with $1\leq r<n$. Recall the definition of $P_r(\bi)$ given in \cite[(3.22), (4.27)]{BK:GradedKL}: if $i_r=i_{r+1}$ then
$P_r(\bi)=1$; if $i_r\neq i_{r+1}$ and in the non-degenerate setting, then \begin{equation}\label{Pri}
P_r(\bi):=\frac{1-q}{1-q^{i_r-i_{r+1}}}\Biggl(1+\frac{y_r-y_{r+1}}{1-q^{i_{r+1}-i_r}}+\sum_{k\geq 1}\frac{y_r-y_{r+1}}{1-q^{i_{r+1}-i_r}}
\Bigl(\frac{q^{i_{r+1}}y_{r+1}-q^{i_r}y_r}{q^{i_{r+1}}-q^{i_r}}\Bigr)^k\Biggr) ;
\end{equation}
while if $i_r\neq i_{r+1}$ and in the degenerate setting, then \begin{equation}\label{Pri2}
P_r(\bi):=\frac{1}{i_r-i_{r+1}}\Biggl(1+\sum_{k\geq 1}\Bigl(\frac{y_r-y_{r+1}}{i_{r+1}-i_r}\Bigr)^k\Biggr).
\end{equation}
The Brundan--Kleshchev's isomorphism in Theorem \ref{BKiso} between $\H_{\beta}^{\Lambda}$ and $\R[\beta]$ depends on the choice of certain polynomials $Q_r(\bi)$ for $1\leq r<n$. see \cite[(3.27--3.29), (4.33-4.35)]{BK:GradedKL}. Instead of following Brundan--Kleshchev's choice given in \cite[(3.30), (4.36)]{BK:GradedKL}, we make a different choice for our purpose. In the degenerate setting, we set
\begin{equation}\label{qri1}
Q_r(\bi):=\begin{cases} 1+y_{r+1}-y_r, &\text{if $i_{r+1}=i_r$;}\\
1+\sum_{k\geq 1}(y_{r+1}-y_r)^k, &\text{if $i_r=i_{r+1}+1$;}\\
P_r(\bi)-1, &\text{if $i_r\neq i_{r+1},i_{r+1}+1$.}
\end{cases}.
\end{equation}
In the non-degenerate setting, following Stroppel--Webster \cite[(27)]{StroppelWebster:QuiverSchur}, we set \begin{equation}\label{qri2}
Q_r(\bi):=\begin{cases} 1-q+q y_{r+1}-y_r, &\text{if $i_{r+1}=i_r$;}\\
\frac{1}{1-q^{-1}}\bigl(1+\sum_{k\geq 1}(\frac{y_{r+1}-qy_r}{1-q})^k\bigr), &\text{if $i_r=i_{r+1}+1$;}\\
P_r(\bi)-1, &\text{if $i_r\neq i_{r+1},i_{r+1}+1$.}
\end{cases}.
\end{equation}
Note that in both (\ref{qri1}) and (\ref{qri2}), \begin{equation}\label{Qri}
\text{$Q_r(\bi)=\frac{P_r(\bi)-1}{y_r-y_{r+1}}$ whenever $i_r=i_{r+1}+1$}.
\end{equation}
Since $y_1,\dots,y_n$ are nilpotent elements in $\H_{\beta}^{\Lambda}$ (cf. \cite[Lemma 2.1]{BK:GradedKL}), the sums in (\ref{Pri}), (\ref{Pri2}), (\ref{qri1}) and (\ref{qri2}) are always finite sums. One can verify that the definitions in both (\ref{qri1}) and (\ref{qri2}) satisfy the requirement in  \cite[(3.27--3.29), (4.33--4.35)]{BK:GradedKL}. Thus they can be used to define
Brundan--Kleshchev's isomorphism in Theorem \ref{BKiso}. \smallskip

Henceforth, we shall use these specific choices of Brundan--Kleshchev's isomorphisms to identify $\H_{\beta}^{\Lambda}(q)$ and $\R[\beta]$ in the non-degenerate setting and to
identify $H_{\beta}^{\Lambda}$ and $\R[\beta]$ in the degenerate setting.

\bigskip
\section{The modified forms of affine Hecke algebras and their generalized Ore localization}

The purpose of this section is to introduce the modified forms for both the non-degenerate and the degenerate affine Hecke algebras of type $A$ and to construct their standard bases and faithful representations. The construction will make use of inverse limit of cyclotomic Hecke algebras and certain generalized Ore localizations with respect to some multiplicatively closed subsets (cf. Appendix A).

Let $m\in\N$. A partition of $m$ is a weakly decreasing sequence $\lambda=(\lambda_1\ge\lambda_2\ge\dots)$ of non-negative integers such
that $|\lambda|=\sum_i\lambda_i=m$. The diagram $[\lambda]$ of a partition $\lambda=(\lambda_1,\lambda_2,\dots)$ is the set $\{(i,j)|i\geq 1, 1\leq j\leq\lambda_i\}$, which is regarded as an array of nodes, or boxes, arranged in left-justified rows, with the row sizes weakly decreasing in the plane. For example, if $n=7$, $\lam=(4,2,1)$, then
$$[\lambda]=[(4,2,1)]=\YoungDiagram{4,2,1}.$$

A multipartition, or $\ell$-partition, of~$n$ is an ordered $\ell$-tuple
$\blam=(\lambda^{(1)},\dots,\lambda^{(\ell)})$ of partitions such that $|\blam|=|\lambda^{(1)}|+\dots+|\lambda^{(\ell)}|=n$. In this case, we write $\blam\vdash n$. The diagram of $\blam$ is the set
\[[\blam]=\set{(k,r,c)|r\ge 1, 1\le c\le\lambda^{(k)}_r \text{ and }1\le k\le \ell},\]
which is regarded as the ordered $\ell$-tuple of the diagrams of its components. For example, if $n=12, \ell=3$, $\blam=((2,1), (1^2), (4,2,1))$, then
$$[\blam]=[((2,1), (1^2), (4,2,1))]=\TriDiagram(2,1|1,1|4,2,1).$$

A standard $\blam$-tableau is a map
$\t\map{[\blam]}\{1,2,\dots,n\}$ such that for $s=1,\dots,\ell$ the
entries in each row of $\t^{(s)}$ increase from left to right and the
entries in each column of $\t^{(s)}$ increase from top to bottom.  Let
$\Std(\blam)$ be the set of standard $\blam$-tableaux. Let $\t^{\blam}$ be the standard $\blam$-tableau such that the numbers $1,2,\dots,n$ are entered
in order from left to right along the rows of $\t^{\lam^{(1)}}$, and then $\t^{\lam^{(2)}},\dots,\t^{\lam^{(\ell)}}$. For example, in the example of the last paragraph, $$
\t^{\blam}=\Tritab({1,2},{3}|{4},{5}|{6,7,8,9},{10,11},{12}) .
$$

Let $\blam$ be a multipartition of $n$. For any $1\leq k\leq n$ and $\t\in\Std(\blam)$, we define $$
\res_{\t}(k):=\res(\gamma):=\kappa_l+c-r\in\Z/e\Z ,
$$
whenever the node occupied by the integer $k$ in $\t$ is $\gamma:=(l,r,c)\in[\blam]$. For any $\t\in\Std(\blam)$, we define the residue sequence $\bi^\t$ of $\t$ to be the ordered $n$-tuple $(\res_\t(1),\dots,\res_\t(n))\in I^n$. For simplicity, we denote by $\bi^\blam$ the residue sequence of $\t^\blam$.

{\bf Henceforth, we shall fix $\beta:=\sum_{i\in I}k_i\alpha_i\in Q_n^+$} and $\bi=(i_1,\dots,i_n)\in I^\beta$. We define \begin{equation}\label{Ibeta}
I(\beta):=\bigl\{i\in I\bigm|k_i\neq 0\bigr\} .
\end{equation}
It is clear that $I(\beta)$ is a finite subset of $I$.

Let $\HH_n\in\{\HH_n(q), H_n\}$. For any $\Lambda\in P^+$ with level $\ell$, let $\HH_n^{\Lam}\in\{\HH_n^{\Lam}(q), H_n^{\Lam}\}$. Let $\pi^\Lam: \HH_n\twoheadrightarrow\HH_n^\Lam$ be the canonical surjection. Following \cite[\S3.1]{Li},  for any $1\leq r\leq n$ and any $j\in I(\beta)$ with $j\neq i_r$, we choose $N>\dim\HH_n^\Lam=\ell^n n!$, and define $$
L^\Lam_{i_r,j,N}:=\begin{cases}1-\Bigl(\frac{q^{i_r}-L_r}{q^{i_r}-q^j}\Bigr)^N, &\text{if $q\neq 1$;}\\
1-\Bigl(\frac{i_r-L_r}{i_r-j}\Bigr)^N, &\text{if $q=1$.}\end{cases},\quad
X_{i_r,j,N}:=\begin{cases}1-\Bigl(\frac{q^{i_r}-X_r}{q^{i_r}-q^j}\Bigr)^N, &\text{if $q\neq 1$;}\\
1-\Bigl(\frac{i_r-x_r}{i_r-j}\Bigr)^N, &\text{if $q=1$.}\end{cases} $$
and \begin{equation}\label{LrXrN}
L^\Lam_{r,N}(\bi):=\prod_{i_r\neq j\in I(\beta)}L^\Lam_{i_r,j,N}\in\H_n^\Lam,\quad X_{r,N}(\bi):=\prod_{i_r\neq j\in I(\beta)}X_{i_r,j,N}\in\H_n .
\end{equation}
The key point in the above definition lies in that $X_{r,N}(\bi)$ depends only on $r, \bi$ and $N>\ell^n n!$ but not on $\Lambda$.

For any $\Lambda,\Lam'\in P^+$, we define $\Lam'>\Lam$ whenever $\Lam'-\Lam\in P^+$. Then $(P^+,>)$ becomes a directed poset (i.e., a set $S$ with a partial order ``$\leq$" such that for any $a,b\in S$ there exists $c\in S$ satisfying $a\leq c$ and $b\leq c$). If $\Lam'>\Lam$ in $P^+$, then there is a canonical surjective homomorphism $$\pi^{\Lam',\Lam}: \H_n^{\Lambda'}\twoheadrightarrow\H_n^{\Lambda}$$ such that
$\pi^{\Lam}=\pi^{\Lam',\Lam}\circ\pi^{\Lam'}$. Let $\pi_\Lam: \underset{\rm{M}}{\underleftarrow{\lim}}\,\H_n^{\rm{M}}\rightarrow\H_n^\Lam$ be the canonical map and $\check{\pi}: \HH_n\rightarrow \underset{\rm{M}}{\underleftarrow{\lim}}\,\H_n^{\rm{M}}$ be the induced map.
Then $\pi^\Lam=\pi_\Lam\circ\check{\pi}$.

\begin{dfn} \label{checkdfn} We define $\check{\H}_n:=\image(\check{\pi})$ to be the image of $\H_n$ (under $\check{\pi}$) in $\underset{\Lambda}{\underleftarrow{\lim}}\,\H_n^\Lambda$. For each $1\leq k<n, 1\leq r\leq n$, we set $$
\hat{T}_k:=\check{\pi}(T_k),\,\,\hat{X}_r:=\check{\pi}(X_r),\,\,\hat{s}_k:=\check{\pi}(s_k),\,\,\hat{x}_r:=\check{\pi}(x_r).
$$
\end{dfn}

Note that for any
$z\in\underset{\Lambda}{\underleftarrow{\lim}}\,\H_n^\Lambda$, \begin{align}\label{z0Lam}
\text{$z=0$ in $\underset{\Lambda}{\underleftarrow{\lim}}\,\H_n^\Lambda$ if and only if $\pi_\Lam(z)=0$ in $\H_n^\Lam$ for all $\Lam\in P^+$.}
\end{align}

\begin{lem} \label{AffInj} Let $Z\in\HH_n$. Then

1) $Z=0$ if and only if $\pi^\Lam(Z)=0$ for any $\Lam\in P^+$.

2) The canonical map $\check{\pi}: \HH_n\rightarrow \underset{\Lambda}{\underleftarrow{\lim}}\,\H_n^\Lambda$ is injective.
\end{lem}

\begin{proof} It is clear that 2) follows from 1). It suffices to prove 1). We only consider the non-degenerate case as the degenerate case is similar. For 1), the ``only if" part is clear. It remains to consider the ``if" part. Suppose that $Z\neq 0$ and $\pi^\Lam(Z)=0$ for all $\Lam\in P^+$. Without loss of generality, we can assume that $Z\in {\H}_n^+(q)$. We can choose a sufficiently large integer $m>\ell(w_{0,n})=n(n-1)/2$ such that \begin{equation}\label{Zdescrip0}
Z\in \text{$K$-Span}\Biggl\{(X_1-1)^{a_1}\dots(X_n-1)^{a_n}T_u\Biggm|\begin{matrix}\text{$u\in\Sym_n, a_1,\dots,a_n\in\N$,}\\
\text{$\sum_{i=1}^{n}a_i<m-n(n-1)/2$.}\end{matrix}\Biggr\} .
\end{equation}
Now we take $\Lam=m\Lam_0$. Since $\pi^\Lam(Z)=0$, it follows that $Z$ lies in the two-sided ideal of ${\H}_n^+(q)$ generated by $(X_1-1)^m$. Using the commutator relations (\ref{6}), (\ref{7a}) and (\ref{pluspart2}), we can deduce that   $$
Z\in \text{$K$-Span}\Biggl\{(X_1-1)^{b_1}\dots(X_n-1)^{b_n}T_u\Biggm|\begin{matrix}\text{$u\in\Sym_n, b_1,\dots,b_n\in\N$,}\\
\text{$\sum_{i=1}^{n}b_i\geq m-n(n-1)/2$.}\end{matrix}\Biggr\} .
$$
We get a contradiction. This proves 1).
\end{proof}

\begin{lem} \label{keyGeLilemma} Let $\bi\in I^\beta$ and $\Lambda\in P^+$ with level $\ell$. Then for any $N>\ell^n n!$, we have that $$
e(\bi)=\prod_{r=1}^{n}(L^\Lam_{r,N}(\bi))^N\in\HH_n^\Lam. $$
Moreover, if we set \begin{equation}\label{EnLam}
E_{N}(\bi):=\prod_{r=1}^{n}(X_{r,N}(\bi))^N\in\H_n ,
\end{equation}
then $\pi^\Lam(E_{N}(\bi))=e(\bi)$. Furthermore, for any sufficiently large $N$,  we have that \begin{equation}\label{ENI}
E_{N}(\bi)=\begin{cases} c\prod_{r=1}^n\prod_{q^{i_r}\neq z\in K_0}(X_r-z)^{b_{r,z}},&\text{if $q\neq 1$;}\\
c\prod_{r=1}^n\prod_{{i_r}\neq z\in K_0}(x_r-z)^{b_{r,z}},&\text{if $q=1$,}\end{cases}
\end{equation}
where $c\in K^\times$, $K_0$ is a finite subset of the algebraic closure $\overline{K}$ of $K$ and $b_{r,z}\in\N$ for each pair $(r,z)$.
\end{lem}

\begin{proof} The proof of the first part of the lemma is essentially the same as the proof of \cite[Corollary 3.9]{Li}. Note that Ge Li \cite[Corollary 3.9]{Li} use the set $I$ instead of the finite subset $I(\beta)$ in the definition of $L_r(\bi)$ and he consider only the case when $e>0$ (so that $|I|$ is finite). However, by \cite[Lemma 4.1, Theorem 5.8]{HuMathas:GradedCellular}, we know that $e(\bi)\neq 0$ in $\HH_n^\Lam$ if and only if $\bi=\bi^\t$ for some $\t\in\Std(\blam)$ and $\blam=(\lam^{(1)},\dots,\lam^{(\ell)})\vdash n$. Therefore, the same argument used in the proof of \cite[Corollary 3.9]{Li} apply equally well to the proof of the current lemma for any $e$ and with $I$ replaced by $I(\beta)$, and we can replace the sufficiently large integers $N_j, N_r(\bi)$ in the proof of \cite[Corollary 3.9]{Li} by any $N>\ell^n n!$. This proves the first part of the lemma.

It remains to the prove the second part of the lemma. Let $\bi=(i_1,\dots,i_n)\in I^\beta$ be a fixed $n$-tuple and $\Lam\in P^+$ be a fixed dominant weight with level $\ell$. For any $\Lam'\in P^+$ with level $\ell'$, if $\Lam'\geq\Lam_{i_1}+\dots+\Lam_{i_n}$ then $\bi=\bi^\t$ for some $\t\in\Std(\blam)$ and $\blam=(\lam^{(1)},\dots,\lam^{(\ell')})\vdash n$. In fact, one can take for example $\t$ to be a standard $\blam$-tableau such that each component $\lambda^{(j)}$ of $\blam$ is equal to either $\emptyset$ or $(1)$. In particular, by \cite[Lemma 4.1, Lemma 5.2, Theorem 5.8]{HuMathas:GradedCellular}, $e(\bi)\neq 0$ in $\HH_n^{\Lam'}$. Now we take $\Lam'\in P^+$ with level $\ell'$ such that $\Lam<\Lam'\geq\Lam_{i_1}+\dots+\Lam_{i_n}$. Note that the canonical map $\pi^{\Lam',\Lam}$ sends the idempotent $e(\bi)\in \HH_n^{\Lam'}$ to the element $e(\bi)\in \HH_n^{\Lam}$. For any $N>(\ell')^n n!$, we consider the nonzero idempotent $e(\bi)=\pi^{\Lam'}(E_N(\bi))\in \HH_n^{\Lam'}$. We have  $$
0\neq e(\bi)=e(\bi)^2=\pi^{\Lam'}(E_N(\bi))e(\bi)=\prod_{r=1}^{n}\pi^{\Lam'}(X_{r,N}(\bi))^{N}e(\bi)\in \HH_n^{\Lam'} .
$$
Note that $(L_r-q^{i_r})^ke(\bi)=0$ and $(L_r-i_r)^ke(\bi)=0$ holds in $\HH_n^{\Lam'}$ for any $k>(\ell')^n n!$. It follows that for any $N>(\ell')^n n!>l^n n!$, $E_{N}(\bi)$ can not contain $(X_r-q^{i_r})$ or $(x_r-i_r)$ as a factor so that it  must have the desired decomposition as in (\ref{ENI}). This proves the second part of the lemma.
\end{proof}

Now for any $\widetilde{\Lam}\in P^+$ with $\widetilde{\Lam}>\Lam$, the dominant weight $\widetilde{\Lam}$ has level $\widetilde{\ell}>\ell$. To avoid confusion, we put a tilde on the notation of every generator of the algebra $\mathscr{H}_n^{\widetilde{\Lam}}$. We take an integer $\widetilde{N}$ such that $\widetilde{N}>\widetilde{\ell}^n n!$. Then $\widetilde{N}>\ell^n n!$. Since $$
\pi^{\widetilde{\Lam},\Lam}(\widetilde{L}^{\widetilde{\Lam}}_{r,\widetilde{N}}(\bi)^{\widetilde{N}})=L^\Lam_{r,\widetilde{N}}(\bi)^{\widetilde{N}} ,
$$
it follows from Lemma \ref{keyGeLilemma} that $\pi^{\Lam,\widetilde{\Lam}}$ sends the idempotent $\widetilde{e}(\bi)$ of $\mathscr{H}_n^{\widetilde{\Lam}}$ to the idempotent $e(\bi)$ of $\mathscr{H}_n^\Lam$. As a result, we can make the following definition.

\begin{dfn} \label{HEidemp} For each $\bi\in I^\beta$, let $\hat{e}(\bi)$ be the unique idempotent in $\underset{\Lambda}{\underleftarrow{\lim}}\,\H_n^\Lambda$ such that
$\pi_\Lam(\he(\bi))=e(\bi)\in\HH_n^\Lam$ for any $\Lam\in P^+$.
\end{dfn}

Let $\bi=(i_1,\dots,i_n)\in I^\beta$ be a fixed $n$-tuple. As we remarked in the proof of Lemma \ref{keyGeLilemma}, for any $\Lam\in P^+$ with level $\ell$, if $\Lam\geq\Lam_{i_1}+\dots+\Lam_{i_n}$ then $\bi=\bi^\t$ for some $\t\in\Std(\blam)$ and $\blam=(\lam^{(1)},\dots,\lam^{(\ell)})\vdash n$. In particular, $e(\bi)\neq 0$ in $\HH_n^\Lam$ and hence $\he(\bi)\neq 0$.

\begin{lem}\label{zeiKeylem} Let $\bi=(i_1,\dots,i_n)\in I^\beta$. Then $\he(\bi)\neq 0$. Moreover, for any $z\in\check{\H}_n$, if $\he(\bi)z=0$ or $z\he(\bi)=0$ in $\underset{\Lambda}{\underleftarrow{\lim}}\,\H_n^\Lambda$, then $z=0$.
\end{lem}

\begin{proof} We have prove the first part of the lemma in the paragraph above this lemma. It remains to prove the second part of this lemma.
We only prove the statement for the non-degenerate case as the degenerate case is similar.

Suppose that $\he(\bi)z=0$. Without loss of generality we can assume that $z\in\check{\pi}\bigl({\H}_n^+(q)\bigr)$. We fix an $Z\in{\H}_n^+(q)$ such that $\check{\pi}(Z)=z$. Suppose that $z\neq 0$. Then it is clear that $Z\neq 0$.

Since $Z$ is prefixed, we can choose a sufficiently large integer $m$ (relative to the fixed integer $n$) such that \begin{equation}\label{Zdescrip}
Z\in \text{$K$-Span}\Biggl\{(X_1-q^{i_1})^{a_1}\dots(X_n-q^{i_n})^{a_n}T_u\Biggm|\begin{matrix}\text{$u\in\Sym_n, a_1,\dots,a_n\in\N$,}\\
\text{$\sum_{i=1}^{n}a_i<m-n(n-1)/2$.}\end{matrix}\Biggr\} .
\end{equation}

We now take $\Lam:=m\Lam_{i_1}+\dots+m\Lam_{i_n}\in P^+$ with level $\ell:=mn$. By construction, for any sufficiently large integer $N>\ell^n n!$, we have that  \begin{equation}\label{EnLam2}
E_{N}(\bi)= c\prod_{r=1}^n\prod_{q^{i_r}\neq z\in K_0}(X_r-z)^{b_{r,z}} ,
\end{equation}
where $c\in K^\times$, $K_0$ is a finite subset of the algebraic closure $\overline{K}$ of $K$ and $b_{r,z}\in\N$ for each pair $(r,z)$.
The assumption $q^{i_r}\neq z\in K_0$ in the above equality (\ref{EnLam2}) implies that there exists some $c_1,\dots,c_n\in\N$ satisfying $\sum_{i=1}^{n}c_i<m-n(n-1)/2$ and some $w\in\Sym_n$ such that the term $(X_1-q^{i_1})^{c_1}\dots(X_n-q^{i_n})^{c_n}T_w$ occurs with nonzero coefficient in $E_{N}(\bi)Z$ when it is expressed as a $K$-linear combination of the basis elements of the form $$
\biggl\{(X_1-q^{i_1})^{a_1}\dots(X_n-q^{i_n})^{a_n}T_u\biggm|\text{$u\in\Sym_n, a_1,\dots,a_n\in\N$}\biggr\}.
$$

On the other hand, by assumption, $\he(\bi)z=0$ in $\underset{\Lambda}{\underleftarrow{\lim}}\,\H_n^\Lambda$, which implies that $\pi^\Lam(E_{N}(\bi)Z)=e(\bi)\pi^\Lam(Z)=0$ in $\HH_n^\Lam$ by Lemma \ref{keyGeLilemma}. Thus $E_{N}(\bi)Z\in\Ker{\pi}^\Lam$ and hence $E_{N}(\bi)Z$ lies in the two-sided ideal of $\HH_n^+(q)$ generated by $$
(X_1-q^{i_1})^{m}(X_1-q^{i_2})^{m}\dots (X_1-q^{i_n})^{m}.$$

We claim that $E_{N}(\bi)Z$ lives inside the following $K$-subspace:  $$
\text{$K$-Span}\Biggl\{(X_1-q^{i_1})^{b_1}\dots(X_n-q^{i_n})^{b_n}T_u\Biggm|\begin{matrix}\text{$u\in\Sym_n, b_1,\dots,b_n\in\N$,}\\
\text{$\sum_{i=1}^{n}b_i\geq m-n(n-1)/2$.}\end{matrix}\Biggr\} .
$$
In fact, by the last paragraph, it suffices to show that for any $w\in\Sym_n$, $T_w(X_1-q^{i_1})^{m}(X_1-q^{i_2})^{m}\dots (X_1-q^{i_n})^{m}$ lies in the above $K$-subspace. In fact, this follows from an induction on $n$ (by writing $T_w$ as $T_uT_{n-1}\cdots T_k$ for some $u\in\Sym_{n-1}$ and $1\leq k<n$ if $w\notin\Sym_{n-1}$) and the commutator relations (\ref{6}), (\ref{7a}) and (\ref{pluspart2}). This proves our claim. However, this claim contradicts to the last sentence in the second last paragraph. Hence we must have that $z=0$.

Finally, if $z\he(\bi)=0$, the lemma follows from a similar argument.
\end{proof}

\begin{dfn} \label{checkBeta} In the non-degenerate setting, we define $\check{\H}_\beta(q)$ to be the $K$-subalgebra of $\underset{\Lambda}{\underleftarrow{\lim}}\,\H_n^\Lambda(q)$ generated by the following elements: $$
\hT_k\he(\bi),\,\, \hX_r^{\pm 1}\he(\bi),\,\,\he(\bi),\,\,\,\bi\in I^\beta,\,1\leq k<n,\,1\leq r\leq n .
$$
In the degenerate setting, we define $\check{H}_\beta$ to be the $K$-subalgebra of $\underset{\Lambda}{\underleftarrow{\lim}}\,H_n^\Lambda$ generated by the following elements: $$
\hs_k\he(\bi),\,\, \hx_r\he(\bi),\,\,\he(\bi),\,\,\,\bi\in I^\beta,\,1\leq k<n,\,1\leq r\leq n .
$$
\end{dfn}

Let $\check{\H}_\beta\in\{\check{\H}_\beta(q),\check{H}_\beta\}$. We set $\he(\beta):=\sum_{\bi\in I^\beta}\he(\bi)\in\check{\H}_\beta$. Then $\he(\beta)$ is the identity element of $\check{\H}_\beta$. By the definition of inverse limit, it is clear that $$
\he(\beta)\he(\gamma)=\delta_{\beta,\gamma}\he(\beta),\quad\text{for any $\gamma\in Q_n^+$.}  $$

To simplify notations, we shall abbreviate $\hX_k\he(\beta), \hx_k\he(\beta), \hT_r\he(\beta),\hs_r\he(\beta)$ as $\hX_k, \hx_k, \hT_r,\hs_r$ when it causes no confusion.

\begin{lem}[Non-degenerate cases] \label{modifiedformV1} In the non-degenerate case, the following relations hold: \begin{align}
\hX_k^{\pm 1}\he(\bi)=\he(\bi)\hX_k^{\pm 1},\,\,\he(\bi)\he(\bj)=\delta_{\bi\bj}\he(\bi), \quad\text{if $1\leq k\leq n, \bi,\bj\in I^\beta$,} \label{XkEi}
\end{align}
\begin{align}
\left.\begin{matrix}
\he(\bi)\hT_r(\hX_{r+1}-\hX_r)\he(\bi)=(q-1)\he(\bi)\hX_{r+1}\he(\bi),\\
\he(\bi)\hT_r\hX_{r}\he(\bi)=\he(\bi)\hX_r\hT_r\he(\bi),\\
\he(\bi)\hT_r\hX_{r+1}\he(\bi)=\he(\bi)\hX_{r+1}\hT_r\he(\bi),\end{matrix}\right\}
\quad \text{if}\,\,\, \bi\in I^{\beta}, i_r\neq i_{r+1}, \label{TX3}
\end{align}
\begin{align}
\he(\bi)f \he(\bj)=0,\,\quad \text{if}\,\,\, \bi,\bj\in I^{\beta}, \bi\neq\bj, f\in K[\hX_1^{\pm1},\dots,\hX_n^{\pm 1}], \label{FEI}
\end{align}
\begin{align}
\he(\bi)\hT_r\he(\bj)=0,\,\quad \text{if}\,\,\, \bi,\bj\in I^{\beta}, \bi\notin\{\bj,s_r\bj\},  \label{TREI}
\end{align}
\begin{align}
\left.\begin{matrix}
\he(\bi)(\hT_r-q)(\hT_r+1)\he(\bj)=0, \\
\he(\bi)\hT_i\hT_{i+1}\hT_i\he(\bj)=\he(\bi)\hT_{i+1}\hT_i\hT_{i+1}\he(\bj),\\
\he(\bi)\hX_r^{\pm 1}\hX_k^{\pm 1}\he(\bj)=\he(\bi)\hX_k^{\pm 1}\hX_r^{\pm 1}\he(\bj),\\
\he(\bi)\hX_k\hX_k^{-1}\he(\bi)=\he(\bi)=\he(\bi)\hX_k^{-1}\hX_k\he(\bi),\\
\hX_{r+1}\he(\bi)=q^{-1}\hT_r\hX_{r}\hT_r\he(\bi)=q^{-1}\he(\bi)\hT_r\hX_{r}\hT_r,\end{matrix}\right\}\quad \text{if}\,\,\, \bi,\bj\in I^{\beta},\label{IJ5}
\end{align}
\begin{align}
\he(\bi)\hT_a\hT_k\he(\bj)=\he(\bi)\hT_k\hT_a\he(\bj),\,\,\text{if $|a-k|>1$ and $\bi,\bj\in I^{\beta},$,}\label{IJ5a}
\end{align}
\begin{align}\he(\bi)\hT_b\hX_k\he(\bj)=\he(\bi)\hX_k\hT_b\he(\bj),\,\,\text{if $k\neq b,b+1$ and $\bi,\bj\in I^{\beta}$,}\label{IJ5b}
\end{align}
where $1\leq k\leq n$, $1\leq r,a,b<n, 1\leq i<n-1$.
\end{lem}

\begin{proof} The relations (\ref{IJ5}), (\ref{IJ5a}) and (\ref{IJ5b}) follow directly from (\ref{1})--(\ref{7}), (\ref{7a}). The remaining relations follows from (\ref{Pri}), (\ref{qri2}), (\ref{Qri}) and Theorem \ref{BKiso}. For example,
assume that $i_r\neq i_{r+1}$, in order to show that $\he(\bi)\hT_r(\hX_{r+1}-\hX_r)\he(\bi)=(q-1)\he(\bi)\hX_{r+1}\he(\bi)$, it suffices (by (\ref{z0Lam})) to show that
for any $\Lam\in P^+$, $e(\bi)T_r(L_{r+1}-L_r)e(\bi)=(q-1)e(\bi)L_{r+1}e(\bi)$.

In fact, we have that $$\begin{aligned}
&\quad\,e(\bi)T_r(L_{r+1}-L_r)e(\bi)\\
&=e(\bi)\bigl(\psi_rQ_r(\bi)e(\bi)-P_r(\bi)e(\bi)\bigr)(L_{r+1}-L_r)=-e(\bi)P_r(\bi)e(\bi)(L_{r+1}-L_r)\\
&=\frac{1-q}{1-q^{i_r-i_{r+1}}}\biggl\{1+\frac{y_r-y_{r+1}}{1-q^{i_{r+1}-i_r}}+\sum_{k\geq 1}\frac{y_r-y_{r+1}}{1-q^{i_{r+1}-i_r}}
\Bigl(\frac{q^{i_{r+1}}y_{r+1}-q^{i_r}y_r}{q^{i_{r+1}}-q^{i_r}}\Bigr)^k\biggr\}\\
&\qquad\times (L_{r+1}-L_r)e(\bi)\\
&=\frac{(q-1)q^{i_{r+1}}(1-y_{r+1})}{q^{i_{r+1}}(1-y_{r+1})-q^{i_{r}}(1-y_{r})}(L_{r+1}-L_r)e(\bi)=(q-1)L_{r+1}e(\bi),
\end{aligned}
$$
as required. The other equalities can be verified in a similar way.
\end{proof}

\begin{lem}[Degenerate cases] \label{modifiedformV2} In the degenerate case, the following relations hold: \begin{align}
\hx_k\he(\bi)=\he(\bi)\hx_k,\,\,\he(\bi)\he(\bj)=\delta_{\bi\bj}\he(\bi), \quad\text{if $1\leq k\leq n, \bi,\bj\in I^\beta$,} \label{xkei}
\end{align}
\begin{align}
\left.\begin{matrix}
\he(\bi)\hs_r(\hx_{r+1}-\hx_r)\he(\bi)=\he(\bi),\\
\he(\bi)\hs_r\hx_{r}\he(\bi)=\he(\bi)\hx_r\hs_r\he(\bi),\\
\he(\bi)\hs_r\hx_{r+1}\he(\bi)=\he(\bi)\hx_{r+1}\hs_r\he(\bi),\end{matrix}\right\}
\quad \text{if}\,\,\, \bi\in I^{\beta}, i_r\neq i_{r+1}, \label{sx3}
\end{align}
\begin{align}
\he(\bi)f \he(\bj)=0,\,\quad \text{if}\,\,\, \bi,\bj\in I^{\beta}, \bi\neq\bj, f\in K[\hx_1,\dots,\hx_n], \label{fei}
\end{align}
\begin{align}
\he(\bi)\hs_r\he(\bj)=0,\,\quad \text{if}\,\,\, \bi,\bj\in I^{\beta}, \bi\notin\{\bj,s_r\bj\},  \label{sei}
\end{align}\begin{align}
\left.\begin{matrix}
\he(\bi)\hs_r^2\he(\bj)=\delta_{\bi\bj}\he(\bj), \\
\he(\bi)\hs_i\hs_{i+1}\hs_i\he(\bj)=\he(\bi)\hs_{i+1}\hs_i\hs_{i+1}\he(\bj),\\
\he(\bi)\hx_r \hx_k\he(\bj)=\he(\bi)\hx_k \hx_r\he(\bj),\\
\hx_{r+1}\he(\bi)=\hs_r\hx_{r}\hs_r\he(\bi)+\hs_r\he(\bi)=\he(\bi)\hs_r\hx_{r}\hs_r+\he(\bi)\hs_r,\end{matrix}\right\}\quad \text{if}\,\,\, \bi,\bj\in I^{\beta},\label{ij5}
\end{align}
\begin{align}
\he(\bi)\hs_a\hs_k\he(\bj)=\he(\bi)\hs_k\hs_a\he(\bj),\,\,\text{if $|a-k|>1$ and $\bi,\bj\in I^{\beta}$,}\label{ij5a}
\end{align}
\begin{align}\he(\bi)\hs_b\hx_k\he(\bj)=\he(\bi)\hx_k\hs_b\he(\bj),\,\,\text{if $k\neq b,b+1$ and $\bi,\bj\in I^{\beta}$,}\label{ij5b}
\end{align}
where $1\leq k\leq n$, $1\leq r,a,b<n, 1\leq i<n-1$.
\end{lem}

\begin{proof} The relations (\ref{ij5}), (\ref{ij5a}) and (\ref{ij5b}) follow directly from (\ref{8})--(\ref{13}), (\ref{13a}). The remaining relations
follows from (\ref{Pri2}), (\ref{qri1}), (\ref{Qri}) and Theorem \ref{BKiso}.
\end{proof}

For each $\Lam\in P^+$, if $i_r\neq i_{r+1}$, then in the non-degenerate setting, $$
(L_{r}-L_s)e(\bi)=q^{i_{r}}(1-y_{r})-q^{i_{s}}(1-y_{s})=\bigl((q^{i_{r}}-q^{i_{s}})-(q^{i_{r}}y_{r}-q^{i_{s}}y_{s})\bigr)e(\bi),
$$
which is invertible in $e(\bi)\H_n^\Lam(q)e(\bi)$ because $(q^{i_{r}}y_{r}-q^{i_{s}}y_{s})e(\bi)$ is nilpotent. Moreover,
its inverse can be expressed as a power series on $(q^{i_{r}}y_{r}-q^{i_{s}}y_{s})e(\bi)$. It follows that $(\hX_{r}-\hX_s)\he(\bi)$ is actually an invertible element in
$\he(\bi)\Bigl(\underset{\Lambda}{\underleftarrow{\lim}}\,\H_n^\Lambda(q)\Bigr)\he(\bi)$ (regarded as an algebra with the identity element $\he(\bi)$). We denote its inverse by  $(\hX_{r}-\hX_s)^{-1}\he(\bi)$.
Similarly, in the degenerate setting, the element  $(\hx_{r}-\hx_s)\he(\bi)$ is  an invertible element in
$\underset{\Lambda}{\underleftarrow{\lim}}\,H_\beta^\Lambda$. We denote its inverse by  $(\hx_{r}-\hx_s)^{-1}\he(\bi)$.

\begin{dfn}\label{MAHA} In the non-degenerate setting, we define the {\bf modified non-degenerate affine Hecke algebra} $\widehat{\H}_\beta(q)$ of type $A$ to be the $K$-subalgebra of $\underset{\Lambda}{\underleftarrow{\lim}}\,\H_\beta^\Lambda(q)$ generated by $\check{\H}_\beta(q)$ and all the elements of the form $$
(\hX_{r}-\hX_s)^{-1}\he(\bi),\,\,\text{where $1\leq r<s\leq n, \bi\in I^\beta$ and $i_r\neq i_{s}$.}
$$
In the degenerate setting, we define the {\bf modified degenerate affine Hecke algebra} $\widehat{H}_\beta(q)$ of type $A$ to be the $K$-subalgebra of $\underset{\Lambda}{\underleftarrow{\lim}}\,H_\beta^\Lambda$ generated by $\check{H}_\beta$ and all the elements of the form $$
(\hx_{r}-\hx_s)^{-1}\he(\bi),\,\,\text{where $1\leq r<s\leq n, \bi\in I^\beta$ and $i_r\neq i_{s}$.}
$$
\end{dfn}

\begin{rem} \label{modifiedplus} 1) Replacing $\HH_n(q)$ by $\HH_n^{+}(q)$ in the definition of $\check{\H}_\beta(q)$ and $X_1^{\pm 1},\dots X_n^{\pm 1}$ by $X_1,\dots,X_n$, we can get an algebra which will be denoted by $\widehat{\H}^+_\beta(q)$. It is clear that $\widehat{\H}^+_\beta(q)$ is a $K$-subalgebra of
$\widehat{\H}_\beta(q)$.

2) Inside the algebra $\check{\H}_\beta(q)$, we can rewrite the first relation in (\ref{TX3}) as \begin{equation}\label{3.2b}
\he(\bi)\hT_r\he(\bi)=(q-1)\he(\bi)\hX_{r+1}(\hX_{r+1}-\hX_r)^{-1}\he(\bi),\,\,\text{if $i_r\neq i_{r+1}$.}\end{equation}
Similarly, inside the algebra $\check{H}_\beta$, we can rewrite the first relation in (\ref{sx3}) as \begin{equation}\label{3.2c}
\he(\bi)\hs_r\he(\bi)=\he(\bi)(\hx_{r+1}-\hx_r)^{-1}\he(\bi),\,\,\,\text{if $i_r\neq i_{r+1}$.}
\end{equation}
\end{rem}

Let $\widehat{\H}_\beta\in\{\widehat{\H}_\beta(q), \widehat{H}_\beta\}$. It is easy to see that if $i_r\neq i_{r+1}$, then $$\begin{matrix}
\he(\bi)\hT_r(\hX_{r+1}-\hX_r)\he(\bi)=(q-1)\he(\bi)\hX_{r+1}\he(\bi)=\he(\bi)(\hX_{r+1}-\hX_r)\hT_r\he(\bi),\\
\he(\bi)\hs_r(\hx_{r+1}-\hx_r)\he(\bi)=\he(\bi)=\he(\bi)(\hx_{r+1}-\hx_r)\hs_r\he(\bi) .
\end{matrix}
$$

Let $\ast$ be the $K$-algebra anti-isomorphism of $\widehat{\H}_\beta$ which is uniquely determined by $$\begin{aligned}
&\he(\bi)^{\ast}:=\he(\bi),\quad \bigl(\he(\bi)f\he(\bj)\bigr)^{\ast}:=\he(\bj)f^{\ast}\he(\bi),\,\,\quad\forall\,\bi,\bj\in I^{\beta}, f\in\H_n.\\
&\Bigl((\hX_{r}-\hX_s)^{-1}\he(\bi)\Bigr)^\ast=(\hX_{r}-\hX_s)^{-1}\he(\bi),\,\, \Bigl((\hx_{r}-\hx_s)^{-1}\he(\bi)\Bigr)^\ast=(\hx_{r}-\hx_s)^{-1}\he(\bi) .
\end{aligned}
$$

For each $\bi\in I^{\beta}$, let $\{t_k(\bi)|1\leq k\leq n\}$ be a set of $n$ algebraically independent indeterminates over $K$. We define \begin{equation}\label{pol}
\pol=\bigoplus_{\bi\in I^{\beta}}\ppol(\bi),
\end{equation}
where \begin{equation}\label{polb}
\ppol(\bi):=\begin{cases}K[t_1(\bi)^{\pm 1},\dots,t_n(\bi)^{\pm 1}], &\text{if $\widehat{\H}_\beta=\widehat{\H}_\beta(q)$;}\\
K[t_1(\bi),\dots,t_n(\bi)], &\text{if $\widehat{\H}_\beta=\widehat{H}_\beta$.}
\end{cases}
\end{equation}

Let $\widetilde{\ppol}(\bi)$ be the localisation of $\ppol(\bi)$ with respect to the following multiplicatively closed subset \begin{equation}\label{units0}
\bigl\{(t_{r}(\bi)-t_s(\bi))^k\bigm|1\leq r\neq s\leq n, k\in\Z^{\geq 0}\bigr\} .
\end{equation}
We set \begin{equation}\label{pol2}
\wpol:=\bigoplus_{\bi\in I^{\beta}}\widetilde{\ppol}(\bi) .
\end{equation}

The symmetric group $\Sym_n$ acts on $\wpol$ by taking $t_k(\bi)$ to $t_{w(k)}(w\bi)$, and $(t_{r}(\bi)-t_s(\bi))^{k}$ to $(t_{w(r)}(w\bi)-t_{w(s)}(w\bi))^{k}$, where $w\in\Sym_n$, $k\in\Z$, $\bi\in I^\beta$. In particular, the transposition $s_k$ maps $t_a(\bi)$ to $t_a(s_k\bi)$ if $a\neq k,k+1$; $t_k(\bi)$ to $t_{k+1}(s_k\bi)$, and $t_{k+1}(\bi)$ to $t_{k}(s_k\bi)$.

Recall that $\{t_k|1\leq k\leq n\}$ is a set of $n$ algebraically independent indeterminates over $K$. Let $\widetilde{\mathcal{P}_n}$ be the localisation of $K[t_1^{\pm 1},\dots,t_n^{\pm 1}]$ if $\widehat{\H}_\beta=\widehat{\H}_\beta(q)$, or the localisation of $K[t_1,\dots,t_n]$ if $\widehat{\H}_\beta=\widehat{H}_\beta$, with respect to the following multiplicatively closed subset $$
\bigl\{(t_{r}-t_s)^k\bigm|1\leq r\neq s\leq n, k\in\Z^{\geq 0}\bigr\} .
$$
Let $$\theta_{\bi}:\,\, \widetilde{\mathcal{P}_n}\cong \widetilde{\ppol}(\bi)$$ be the canonical isomorphism induced by the map $t_k^{\pm 1}\mapsto t_k(\bi)^{\pm 1}$ for each $1\leq k\leq n$. For each $f\in \widetilde{\mathcal{P}_n}$, we set \begin{equation}\label{fi1}
f_{\bi}:=\theta_{\bi}(f)\in \widetilde{\ppol}(\bi) .
\end{equation}
The symmetric group $\Sym_n$ acts on $\widetilde{\mathcal{P}_n}$ by taking $t_k$ to $t_{w(k)}$, and $(t_{r}-t_s)^{k}$ to $(t_{w(r)}-t_{w(s)})^{k}$, where $w\in\Sym_n$, $k\in\Z$. For any $f\in\widetilde{\mathcal{P}_n}$, we have that
$w(f_\bi)=\bigl(w(f)\bigr)_{w\bi}$ for any $w\in\Sym_n, \bi\in I^{\beta}$.

For any $\bi,\bj\in I^{\beta}, f\in\widetilde{\mathcal{P}_n}$, $1\leq r<n$ and $1\leq k\leq n$, we define \begin{equation}\label{000}\left\{\begin{aligned}
 \hX_k^{\pm 1}\he(\bi)\cdot f_\bi:&=t_k(\bi)^{\pm 1}f_\bi,\\
 \he(\bj)\cdot f_\bi:&=\delta_{\bi\bj}f_{\bi},\quad\text{if $\bi,\bj\in I^{\beta}$,}\\
\hT_r\he(\bi)\cdot f_\bi:&=\Bigl(\frac{t_{r+1}-qt_r}{t_{r+1}-t_r}{s_r(f)}\Bigr)_{s_r\bi}+(q-1)\frac{t_{r+1}(\bi)}{t_{r+1}(\bi)-t_r(\bi)}f_{\bi},
\end{aligned}\right.
\end{equation}
and  \begin{equation}\label{111}\left\{\begin{aligned}
 \hx_k\he(\bi)\cdot f_\bi:&=t_k(\bi)f_\bi,\\
 \he(\bj)\cdot f_\bi:&=\delta_{\bi\bj}f_{\bi},\quad\text{if $\bi,\bj\in I^{\beta}$,}\\
 \hs_r\he(\bi)\cdot f_\bi:&=\Bigl(\frac{t_{r+1}-t_r-1}{t_{r+1}-t_r}{s_r(f)}\Bigr)_{s_r\bi}+\frac{1}{t_{r+1}(\bi)-t_r(\bi)}f_{\bi}.
\end{aligned}\right.
\end{equation}

\begin{prop} \label{rep} Let $\widehat{\H}_\beta\in\{\widehat{\H}_\beta(q), \widehat{H}_\beta\}$. The above rules extend uniquely to a faithful representation $\widehat{\rho}_\beta$ of $\widehat{\H}_\beta$ on ${\wpol}$.
\end{prop}

\begin{proof} We only consider the non-degenerate case as the degenerate case is similar. We divide the proof into three steps:
\smallskip

{\it Step 1.} For any $\bi\in I^{\beta}, f\in\widetilde{\mathcal{P}_n}$, $1\leq r<n$ and $1\leq k\leq n$, we define \begin{equation}\label{0000}\left\{\begin{aligned}
X_k^{\pm 1}\bullet   f_\bi:&=\delta_{\bi\bj}t_k(\bi)^{\pm 1}f_\bi,\\
T_r\bullet f_\bi:&=\Bigl(\frac{t_{r+1}-qt_r}{t_{r+1}-t_r}{s_r(f)}\Bigr)_{s_r\bi}+(q-1)\frac{t_{r+1}(\bi)}{t_{r+1}(\bi)-t_r(\bi)}f_{\bi},
\end{aligned}\right.
\end{equation}

We claim that the formulae (\ref{0000}) extends to a well-defined representation $\widetilde{\rho}_n$ of $\HH_n(q)$ on ${\wpol}$.

We need to verify the defining relations for $\HH_n(q)$. The only non-trivial relation that need to be checked is the braid relations and the quadratic relations. In other words, we need to prove $(T_s-q)(T_s+1)\bullet f_{\bi}=0$ for any $1\leq s<n$, and \begin{equation}\label{braid000}
T_{r+1}T_rT_{r+1}\bullet f_{\bi}=T_rT_{r+1}T_r\bullet f_{\bi},\quad  \,\,\forall\,1\leq r<n-1.
\end{equation}

The first  equality follows from a direct and easy verification. For the second one, it can be proved by a brute-force calculation via comparing  the coefficients of $$\begin{aligned}
&f_{\bi},\,\,\,\,(s_r(f))_{s_r\bi},\,\,\,\,(s_{r+1}(f))_{s_{r+1}\bi},\,\,\,\,(s_rs_{r+1}(f))_{s_rs_{r+1}\bi},\\
&(s_{r+1}s_r(f))_{s_{r+1}s_r\bi},\,\,\,\,
(s_rs_{r+1}s_r(f))_{s_rs_{r+1}s_r\bi}.
\end{aligned}
$$
on both sides of (\ref{braid000}). Most of the check is an easy job except for the coefficient of $f_{\bi}$. In fact, we can get the following coefficient $C_1$ of $f_{\bi}$ appearing in the LHS of (\ref{braid000}): $$\begin{aligned}
C_1&=(q-1)^3\frac{t_{r+2}(\bi)^2t_{r+1}(\bi)}{(t_{r+2}(\bi)-t_{r+1}(\bi))^2(t_{r+1}(\bi)-t_{r}(\bi))}\\
&\qquad +\frac{t_{r+2}(\bi)-qt_{r+1}(\bi)}{t_{r+2}(\bi)-t_{r+1}(\bi)}\times \frac{t_{r+1}(\bi)-qt_{r+2}(\bi)}{t_{r+1}(\bi)-t_{r+2}(\bi)}\times
\frac{(q-1)t_{r+2}(\bi)}{t_{r+2}(\bi)-t_{r}(\bi)},
\end{aligned}
$$
while the coefficient $C_2$ of $f_{\bi}$ appearing in the RHS of (\ref{braid000}) is as follows: $$\begin{aligned}
C_2&=(q-1)^3\frac{t_{r+1}(\bi)^2t_{r+2}(\bi)}{(t_{r+1}(\bi)-t_{r}(\bi))^2(t_{r+2}(\bi)-t_{r+1}(\bi))}\\
&\qquad +\frac{t_{r+1}(\bi)-qt_{r}(\bi)}{t_{r+1}(\bi)-t_{r}(\bi)}\times \frac{t_{r}(\bi)-qt_{r+1}(\bi)}{t_{r}(\bi)-t_{r+1}(\bi)}\times
\frac{(q-1)t_{r+2}(\bi)}{t_{r+2}(\bi)-t_{r}(\bi)} .
\end{aligned}
$$
We want to prove that $C_1=C_2$. It suffices to show that $$\begin{aligned}
&(q-1)^2t_{r+1}(\bi)t_{r+2}(\bi)(t_{r+2}(\bi)-t_{r}(\bi))(t_{r+1}(\bi)-t_{r}(\bi))-\\
&\qquad\quad (t_{r+2}(\bi)-qt_{r+1}(\bi))(t_{r+1}(\bi)-qt_{r+2}(\bi))(t_{r+1}(\bi)-t_{r}(\bi))^2\\
&=(q-1)^2(t_{r+1}(\bi))^2(t_{r+2}(\bi)-t_{r+1}(\bi))(t_{r+2}(\bi)-t_{r}(\bi))-\\
&\qquad\quad(t_{r+1}(\bi)-qt_{r}(\bi))(t_{r}(\bi)-qt_{r+1}(\bi))(t_{r+2}(\bi)-t_{r+1}(\bi))^2 .
\end{aligned}
$$
We regard the above equality as an equation on the indeterminate $t_{r+1}(\bi)$ with degree $\leq 2$. Set $t_{r+1}(\bi)=t_r(\bi), t_{r+1}(\bi), 0$, we always get an identity. This implies that it must be an identity forever. This proves that $C_1=C_2$ as required. This completes the proof of our claim.
\smallskip

{\it Step 2.} We claim that representation $\widetilde{\rho}_n$ constructed in Step 1 is faithful.

In fact, comparing the formulae (\ref{0000}) and (\ref{affAct1}), we see that for any $\bi\in I^\beta$ and $f\in\widetilde{\mathcal{P}}_n$, $$
\sum_{\bj\in I^\beta}\theta_{\bj}^{-1}(z\bullet f_\bi)=z\ast f ,
$$
where ``$\ast$" is the action introduced in (\ref{affAct1}) and (\ref{affAct2}), and by convention we understand that $\theta_{\bj}^{-1}(x):=0$ whenever $x\in\widetilde{\ppol}(\bi)$ with $\bj\neq\bi$.
As a result, the faithfulness of $\widetilde{\rho}_n$ follows from the faithfulness of $\rho_q$.

\smallskip
{\it Step 3.} By Lemma \ref{AffInj}, we know that $\check{\pi}$ is injective which induces a $K$-algebra isomorphism $\H_n(q)\cong\check{\H}_n(q)$. Therefore, the faithful representation $\widetilde{\rho}_n$ gives rise to a faithful representation $\check{\rho}_n$ of $\check{\HH}_n(q)$ on ${\wpol}$.

Now by Lemma \ref{zeiKeylem} every elements of $\check{\H}_\beta(q)$ can be written uniquely as $\sum_{\bi\in I^\beta}z_{\bi}\he(\bi)$, where $z_\bi\in\check{\H}_n$ for each $\bi$. We define $$
\check{\rho}_\beta\biggl(\sum_{\bi\in I^\beta}z_{\bi}\he(\bi)\biggr)(f_\bj):=\check{\rho}_n\bigl(z_\bj\bigr)f_{\bj} ,\quad\,\forall\,\,\bj\in I^\beta, f\in K[t_1^{\pm 1},\dots,t_n^{\pm 1}].
$$
From Lemma \ref{zeiKeylem} we see that $\check{\rho}_\beta$ is a well-defined representation of $\check{\HH}_\beta(q)$ on ${\wpol}$. The faithfulness of $\check{\rho}_\beta$ follows from the faithfulness of $\check{\rho}_n$.

Finally, for any $1\leq r<n$, $t_r(\bj)-t_{r+1}(\bj)$ is invertible in $\widetilde{\ppol}(\bj)$. It follows from Lemma \ref{Local1} and Lemma \ref{OreLocal} that the representation $\check{\rho}_\beta$ of $\check{\HH}_\beta(q)$ on ${\wpol}$ can be extended uniquely to the representation $\widehat{\rho}_\beta$ of $\widehat{\HH}_\beta(q)$ on ${\wpol}$ which is given exactly by the formula (\ref{000}). The faithfulness of
$\widehat{\rho}_\beta$ follows from the faithfulness of $\check{\rho}_\beta$.
\end{proof}

Let $\bi\in I^\beta$. For each $w\in\Sym_n$, we fix a reduced expression $s_{j_1}s_{j_2}\dots s_{j_k}$ of $w$, and define $$\begin{aligned}
\widehat{w}_{\bi}&:=\bigl(\he(w\bi)\hs_{j_1}\he(s_{j_1}w\bi)\bigr)
\bigl(\he(s_{j_1}w\bi)\hs_{j_2}\he(s_{j_2}s_{j_1}w\bi)\bigr)\dots \bigl(\he(s_{j_k}\bi)\hs_{j_k}\he(\bi)\bigr),\\
\widehat{T}_{w,\bi}&:=\bigl(\he(w\bi)\hT_{j_1}\he(s_{j_1}w\bi)\bigr)
\bigl(\he(s_{j_1}w\bi)\hT_{j_2}\he(s_{j_2}s_{j_1}w\bi)\bigr)\dots \bigl(\he(s_{j_k}\bi)\hT_{j_k}\he(\bi)\bigr).
\end{aligned} $$

\begin{thm} \label{basis1} The following set \begin{equation}\label{basis314}
\biggl\{\widehat{T}_{w,\bi}\hX_1^{a_1}\dots \hX_n^{a_n}\prod_{\substack{1\leq r<s\leq n\\ i_r\neq i_{s}}}(\hX_r-\hX_{s})^{-b_{r,s}}\he(\bi)\biggm|\begin{matrix}
\text{$w\in\Sym_n, \bi\in I^{\beta}$, $b_{r,s}\in\N$,}\\
\text{$a_1,\dots,a_n\in\Z$, $b_{r,s}>0$ only if either}\\
\text{$a_r=0\geq a_{s}$ or $a_r>0=a_{s}$}\end{matrix}\biggr\}
\end{equation}
is a $K$-basis of $\widehat{\H}_\beta(q)$, and the following set \begin{equation}\label{basis315}
\biggl\{\widehat{w}_{\bi}\hx_1^{a_1}\dots \hx_n^{a_n}\prod_{\substack{1\leq r<s\leq n\\ i_r\neq i_{s}}}(\hx_r-\hx_{s})^{-b_{r,s}}\he(\bi)\biggm|\begin{matrix}
\text{$w\in\Sym_n, \bi\in I^{\beta}$, $b_{r,s}, a_1,\dots,a_n\in\N$,}\\
\text{$b_{r,s}>0$ only if $a_{s}=0$}\end{matrix}\biggr\}
\end{equation}
is a $K$-basis of $\widehat{H}_\beta$.
\end{thm}

\begin{proof} We only prove (\ref{basis314}) as (\ref{basis315}) can be proved in a similar way.  We use $\mathcal{B}_\beta(q)$ to denote the set labelled by (\ref{basis314}).

For any $\bi\in I^\beta$,  $1\leq r<s\leq n$ with $i_r\neq i_s$, we have that $$ \begin{aligned}
&\quad\, \he(\bi)(\hX_r-\hX_s)^{-1}\hT_k\he(s_k\bi)-\he(\bi)\hT_k(\hX_{s_k(r)}-\hX_{s_k(s)})^{-1}\he(s_k\bi)\\
&=\begin{cases} 2(q-1)\he(\bi)\hX_{r+1}(\hX_r-\hX_s)^{-1}(\hX_{s_k(r)}-\hX_{s_k(s)})^{-1}\he(s_k\bi), &\text{if $s=r+1$ and $k=r$,}\\
(q-1)\he(\bi)\hX_{k+1}(\hX_r-\hX_s)^{-1}(\hX_{s_k(r)}-\hX_{s_k(s)})^{-1}\he(s_k\bi), &\text{if $s-1>k=r$ or $s=k+1>r+1$,}\\
-(q-1)\he(\bi)\hX_{k+1}(\hX_r-\hX_s)^{-1}(\hX_{s_k(r)}-\hX_{s_k(s)})^{-1}\he(s_k\bi), &\text{if $k=r-1$ or $k=s$,}\\
0, &\text{otherwise.}
\end{cases}
\end{aligned}
$$
As a result, using Lemma \ref{basis2aff} and the relation (\ref{3.2b}) we can see that set $\mathcal{B}_\beta(q)$ spans the $K$-linear space $\widehat{\H}_\beta(q)$. It remains to prove that they are $K$-linearly independent. To this end, by Proposition
\ref{rep}, it suffices to show that their images under $\widehat{\rho}_\beta$ are $K$-linearly independent.

By Lemma \ref{zeiKeylem} and Lemma \ref{modifiedformV1}, $\{\he(\bi)|\bi\in I^\beta\}$ is a set of nonzero pairwise  orthogonal idempotents. Suppose that the elements in (\ref{basis314}) are $K$-linearly dependent. Then we can find $\bi,\bj\in I^\beta$ and a nonempty finite subset $J\subseteq\Sym_n$, such that $w\bi=\bj$ for any $w\in J$, and \begin{equation}\label{sum0}
\sum_{w\in J, \underline{a}, \mathbf{b}}c_{w,\underline{a},\mathbf{b}}\he(\bj)\widehat{T}_{w,\bi}\hX_1^{a_1}\dots \hX_n^{a_n}\prod_{\substack{1\leq r<s\leq n\\ i_r\neq i_{s}}}(\hX_r-\hX_{s})^{-b_{r,s}}\he(\bi)=0,
\end{equation}
where $\underline{a}=(a_1,\dots,a_n)\in\Z^n,
\mathbf{b}=(b_{r,s})_{r,s}$, and $b_{r,s}>0$ only if $a_r=0\geq a_{s}$ or $a_r>0=a_{s}$ for any $1\leq r<s<n$ with $i_r\neq i_s$, and $0\neq c_{w,\underline{a},\mathbf{b}}\in K$ for any $3$-tuple $(w,\underline{a},\mathbf{b})$.
Furthermore, we assume that (\ref{sum0}) is chosen such that
\begin{equation}\label{min}
\text{$\#\{(w,\underline{a},\mathbf{b})|w\in J, c_{w,\underline{a},\mathbf{b}}\neq 0\}$ is minimal.}
\end{equation}

For each $w\in J$, we set $$
J(w):=\{(\underline{a},\mathbf{b})|c_{w,\underline{a},\mathbf{b}}\neq 0\}.
$$
By assumption $J\neq\emptyset$. Then $$
\sum_{w\in J}\he(\bj)\widehat{T}_{w,\bi}\he(\bi)\Bigl(\sum_{(\underline{a},\mathbf{b})\in J(w)}c_{w,\underline{a},\mathbf{b}}\hX_1^{a_1}\dots \hX_n^{a_n}\prod_{\substack{1\leq r<s\leq n\\ i_r\neq i_s}}(\hX_r-\hX_{s})^{-b_{r,s}}\Bigr)\he(\bi)=0.
$$
We divide the remaining proof into two steps:

\smallskip

{\it Step 1.} We claim that if $$
0\neq f\in K[\hX_i^{\pm 1}\he(\bi), (\hX_r-\hX_{s})^{-1}\he(\bi)|1\leq i\leq n, 1\leq r<s\leq n, i_r\neq i_{s}],$$
then $\sum_{w\in J}\he(\bj)\hT_w\he(\bi)f\neq 0$.

In fact, multiplying some monomial of the form $\hX_1^{a'_1}\hX_2^{a'_2}\dots \hX_n^{a'_n}\prod_{1\leq r<s\leq n}(\hX_r-\hX_{s})^{a'_{rs}}\he(\bi)$ on the RHS of (\ref{sum0}) if necessary, we can assume without loss of generality that $0\neq f\in K[\hX_1\he(\bi),\dots,\hX_n\he(\bi)]$ and $f$ is a $K$-linear combination of some monomials of the form $\hX_1^{c_1}\dots\hX_n^{c_n}\he(\bi)$ with $c_1\gg c_2\gg \dots\gg c_n>0$. Here, $\gg$ means that for each $1\leq r<s\leq n$, $c_r-c_s$ is sufficiently large with respect to $n$.

For any $\underline{c}=(c_1,\dots,c_n), \underline{c}':=(c'_1,\dots,c'_n)\in\mathbb{N}^n$, we define $$
\underline{c}>\underline{c}'\quad \text{if and only if}\quad \text{$\begin{matrix}\text{there exists some $1\leq i\leq n$ such that}\\
\text{$c_i\gg c'_i$ and $c_j=c'_j$ for any $j<i$.}\end{matrix}$}
$$

Let  $\underline{c}=(c_1,\dots,c_n)\in\N^n$ be the unique maximal element (under ``$<$") such that $\hX_1^{c_1}\dots\hX_n^{c_n}$ appears with nonzero coefficient in $f$. Let $w_{J}\in J$ be a fixed element such that
$\ell(w_J)\geq\ell(w)$ for any $w\in J$. We consider the representation $\widehat{\rho}_\beta$. For any $\underline{d}=(d_1,\dots,d_n)\in\N^n$ such that $\hX_1^{d_1}\dots\hX_n^{d_n}$ appear with nonzero coefficient in $f$, it follows from (\ref{0000}) that $$\begin{aligned}
&\quad\,\Bigl(\sum_{w\in J}\he(\bj)\hT_w\he(\bi)\hX_1^{d_1}\dots\hX_n^{d_n}\he(\bi)\Bigr)(1)\\
&=A_{\underline{d}}t_{w_{J}(1)}(\bj)^{d_1}\dots t_{w_{J}(n)}(\bj)^{d_n}\he(\bj)+\sum_{w<w_J}B_{\underline{d},w}(t_{w(1)}(\bj))^{d_1}\dots (t_{w(n)}(\bj))^{d_n}\he(\bj) ,\end{aligned}
$$
where both $A_{\underline{d}}$ and $B_{\underline{d},w}$ are certain rational functions on $t_1(\bj),\dots,t_n(\bj)$ such that the degrees of their denominators and numerators   (with respect to each variable $t_1(\bj),\dots,t_n(\bj)$) depend only on $w_J, w, n$ and not on $\underline{d}$. Now we multiply all the denominators on both sides. Note that $A_{\underline{d}}=A_{\underline{c}}$ and the coefficient of $A_{\underline{d}}t_{w_{J}(1)}(\bj)^{d_1}\dots t_{w_{J}(n)}(\bj)^{d_n}\he(\bj)$ in
$\sum_{w\in J}\he(\bj)\hT_w\he(\bi)f$ is a nonzero scalar multiple of the coefficient of $A_{\underline{c}}t_{w_{J}(1)}(\bj)^{c_1}\dots t_{w_{J}(n)}(\bj)^{c_n}\he(\bj)$ in
$\sum_{w\in J}\he(\bj)\hT_w\he(\bi)f$. Our assumption that $c_r-c_s$ is sufficiently large with respect to the given $n$ (for any $1\leq r<s\leq n$) implies that the term which involves $t_{w_{J}(1)}(\bj)^{c_1}\dots t_{w_{J}(n)}(\bj)^{c_n}$ as a factor can never be cancelled. This proves that $\sum_{w\in J}\he(\bj)\hT_w\he(\bi)f\neq 0$.
\smallskip

{\it Step 2.} In view of the result we proved in Step 1, we can get that for each $w\in J$,  $$
\sum_{(\underline{a},\mathbf{b})\in J(w)}c_{w,\underline{a},\mathbf{b}}\hX_1^{a_1}\dots \hX_n^{a_n}\prod_{\substack{1\leq r<s\leq n\\ i_r\neq i_{s}}}(\hX_r-\hX_{s})^{-b_{r,s}}\he(\bi)=0 .
$$
Applying Lemma \ref{zeiKeylem} and use the representation $\check{\rho}_n$, we get that
\begin{equation}\label{nonzero}
\sum_{(\underline{a},\mathbf{b})\in J(w)}c_{w,\underline{a},\mathbf{b}}t_1(\bi)^{a_1}\dots t_n(\bi)^{a_n}\prod_{\substack{1\leq r<s\leq n\\ i_r\neq i_{s}}}(t_r(\bi)-t_{s}(\bi))^{-b_{r,s}}=0 .
\end{equation}

If $b_{r,s}=0$ for any $1\leq r<s\leq n$, then it is clear that we will get a contradiction. Otherwise, we set $$\begin{aligned}
s_0&:=\max\{s|\text{$b_{r,s}>0$, $1\leq r<s\leq n$, $i_r\neq i_{s}$, $(\underline{a},\mathbf{b})\in J(w)$}\},\\
N_0&:=\max\{b_{r,s_0}|\text{$b_{r,s_0}>0$, $1\leq r<s_0$, $(\underline{a},\mathbf{b})\in J(w)$}\}.
\end{aligned}
$$

We fix $1\leq r_0<s_0\leq n$ such that $b_{r_0,s_0}=N_0$. We multiply $(t_{r_0}(\bi)-t_{s_0}(\bi))^{N_0}$ on both sides of (\ref{nonzero}) and then specialize $t_{s_0}(\bi):=t_{r_0}(\bi)$. Then by our construction, the fact that
$a_{r_0}=0\geq a_{s_0}$ or $a_{r_0}>0=a_{s_0}$ and (\ref{min}), we can deduce from (\ref{nonzero}) the following equality:  \begin{equation}\label{nonzero1}
\sum_{(\underline{a},\mathbf{b})\in J_1(w)}c'_{w,\underline{a},\mathbf{b}}z_0\prod_{\substack{1\leq k\leq n\\ k\neq r_0,s_0}}t_k(\bi)^{a_k}\prod_{\substack{1\leq r<s\leq s_0,\\ (r,s)\neq (r_0,s_0)}}(t_r(\bi)-t_{s}(\bi))^{-b_{r,s}}=0 ,
\end{equation}
where $\emptyset\neq J_1(w)\subseteq J(w)$, $0\neq c'_{w,\underline{a},\mathbf{b}}$ for each $3$-tuple $(w,\underline{a},\mathbf{b})$, $$
z_0:=\begin{cases} t_{r_0}(\bi)^{a_{r_0}}, &\text{if $a_{r_0}>0=a_{s_0}$;}\\
t_{r_0}(\bi)^{a_{s_0}}, &\text{if $a_{r_0}=0>a_{s_0}$;}\\
1, &\text{otherwise.}
\end{cases}
$$

Next we define $$\begin{aligned}
s_1&:=\max\{s|\text{$b_{r,s}>0$, $1\leq r<s\leq s_0$, $(r,s)\neq (r_0,s_0)$, $(\underline{a},\mathbf{b})\in J_1(w)$}\},\\
N_1&:=\max\{b_{r,s_1}|\text{$b_{r,s_1}>0$, $1\leq r<s_1$, $(\underline{a},\mathbf{b})\in J_1(w)$}\}.
\end{aligned}
$$
We fix $1\leq r_1<s_1\leq n$ such that $b_{r_1,s_1}=N_1$. Then we repeat the previous argument. After a finite step, we shall get a nonzero linear combination of some  monomials on $t_1(\bi)^{\pm 1},\dots,t_n(\bi)^{\pm 1}$ which is equal to zero. This is again a contradiction. This completes the proof of the Theorem.
\end{proof}

Applying the anti-isomorphism $\ast$, we see that the following set $$
\biggl\{\he(w\bi)\hX_1^{a_1}\dots \hX_n^{a_n}\prod_{\substack{1\leq r<s\leq n\\ i_r\neq i_{s}}}(\hX_r-\hX_{s})^{-b_{r,s}}\widehat{T}_{w,\bi}\he(\bi)\biggm|\begin{matrix}
\text{$w\in\Sym_n, \bi\in I^{\beta}$, $b_{r,s}\in\N$,}\\
\text{$a_1,\dots,a_n\in\Z$, $b_{r,s}>0$ only if either}\\
\text{$a_r=0\geq a_{s}$ or $a_r>0=a_{s}$}\end{matrix}\biggr\}
$$
is a $K$-basis of $\widehat{\H}_\beta(q)$, and the following set $$
\biggl\{\he(w\bi)\hx_1^{a_1}\dots \hx_n^{a_n}\prod_{\substack{1\leq r<s\leq n\\ i_r\neq i_{s}}}(\hx_r-\hx_{s})^{-b_{r,s}}\widehat{w}_{\bi}\he(\bi)\biggm|\begin{matrix}
\text{$w\in\Sym_n, \bi\in I^{\beta}$, $b_{r,s}, a_1,\dots,a_n\in\N$,}\\
\text{$b_{r,s}>0$ only if $a_{s}=0$}\end{matrix}\biggr\}
$$
is a $K$-basis of $\widehat{H}_\beta$.

\begin{cor} \label{nozerodivisor} For any $\bi\in I^{\beta}$, $0\neq f\in\widehat{\H}_\beta(q) \he(\bi)$, $0\neq g\in \he(\bi)\widehat{\H}_\beta(q)$, $0\neq h\in K[\hX_1^{\pm 1},\dots,\hX_n^{\pm 1}]$, we have that $f\he(\bi)h\neq 0\neq h\he(\bi)g$. The same is true if we replace $\widehat{\H}_\beta(q)$ and $K[\hX_1^{\pm 1},\dots,\hX_n^{\pm 1}]$ by $\widehat{H}_\beta$ and $K[\hx_1,\dots,\hx_n]$ respectively.
\end{cor}

\begin{proof} This follows directly from Lemma \ref{basis1}.
\end{proof}

Recall the definition of the subalgebra $\widehat{\H}^+_\beta(q)$ of $\widehat{\H}_\beta(q)$ in Remark \ref{modifiedplus}. By a natural restriction of (\ref{000}) to the subalgebra $\widehat{\H}^+_\beta(q)$, we can also get a faithful representation $\rho_q^+$ of $\widehat{\H}^+_\beta(q)$ on
$\oplus_{\bi\in I^\beta}\widetilde{P}_n(\bi)$, where $\widetilde{P}_n(\bi)$ is the localization of $K[t_1(\bi),\dots,t_n(\bi)]$ with respect to (\ref{units0}) for each $\bi$.

\begin{thm}\label{333} The following set $$
\biggl\{\he(w(\bi))\widehat{T}_{w,\bi}\hX_1^{a_1}\dots \hX_n^{a_n}\prod_{\substack{1\leq r<s\leq n\\ i_r\neq i_{s}}}(\hX_r-\hX_{s})^{-b_{r,s}}\he(\bi)\biggm|\begin{matrix}
\text{$w\in\Sym_n, \bi\in I^{\beta}$, $b_{r,s}, a_1,\dots,a_n\in\N$,}\\
\text{$b_{r,s}>0$ only if $a_{s}=0$}\end{matrix}\biggr\}
$$
is a $K$-basis of $\widehat{\H}^+_\beta(q)$.
\end{thm}

\begin{proof} This can be proved in a similar argument to that used in the proof of Theorem \ref{basis1}.
\end{proof}

The bases we have obtained in Theorem \ref{basis1} can be regarded as the modified affine Hecke algebras analogues of the standard bases of the usual affine Hecke algebras given in Lemma \ref{basis2aff}. In the rest of this section, we shall describe the center of these modified affine Hecke algebras. We shall introduce certain multiplicatively closed subsets of these modified affine Hecke algebras and consider their generalized Ore localization in the sense of Appendix A of this paper. This is equivalent to enlarge the modified affine Hecke algebra so that certain elements become locally invertible in the bigger rings.


\begin{dfn} Let $\bi\in I^{\beta}$. If $\widehat{\H}_\beta=\widehat{\H}_\beta(q)$, then we define $$
\widehat{M}(\bi):=\biggl\{\he(\bi)\biggr\}\bigcup\biggl\{(\hX_r-\hX_{s})\he(\bi)\biggm|1\leq r<s\leq n, i_r\neq i_{s}\biggr\}.
$$
If $\widehat{\H}_\beta=\widehat{H}_\beta$, then we define $$
\widehat{M}(\bi):=\biggl\{\he(\bi)\biggr\}\bigcup\biggl\{(\hx_r-\hx_{s})\he(\bi)\biggm|1\leq r<s\leq n, i_r\neq i_{s}\Biggr\}.
$$
Let $\widehat{\Sigma}(\bi)$ be the multiplicative closure of $\widehat{M}(\bi)$.
\end{dfn}

\begin{lem} \label{Local1} All the assumptions and conditions in Lemma \ref{OreLocal} of Appendix A are satisfied if we take $$\begin{aligned}
& A=\check{\H}_\beta(q),\,\, A_0:=K[\hX_k^{\pm1}\he(\bi),\he(\bi)|1\leq k\leq n, \bi\in I^\beta],\\
& J=I^\beta,\,\, e_{\bi}=e(\bi),\,\, S_{\bi}=\widehat{\Sigma}(\bi),\,\,\text{for $\bi\in I^\beta$}.
\end{aligned}$$
In particular, we can construct the generalized Ore localization of $\check{\H}_\beta(q)$ with respect to $(A_0,\{e_{\bi}\}_{\bi\in J}, \{S_{\bi}\}_{\bi\in J})$. Moreover the resulting generalized Ore localization is canonically isomorphic to $\widehat{\H}_\beta(q)$. A similar statement holds if we replace $\check{\H}_\beta(q), \hX_k^{\pm 1}$ and $\widehat{\H}_\beta(q)$  by $\check{H}_\beta, \hx_k$ and $\widehat{H}_\beta$ respectively.
\end{lem}

\begin{proof} We only consider the non-degenerate case as the degenerate case is similar. In view of Lemma \ref{OreLocal}, it suffices to verify the assumptions (O1) and (O2) in Lemma \ref{OreLocal}. Let $1\leq r<s\leq n$. Since
$\check{\H}_\beta(q)\subseteq\underset{\Lambda}{\underleftarrow{\lim}}\,\H_n^\Lambda(q)$ and if $i_r\neq i_{s}$ then $(\hX_r-\hX_{s})\he(\bi)$ is invertible in $\he(\bi)\biggl(\underset{\Lambda}{\underleftarrow{\lim}}\,\H_n^\Lambda(q)\biggr)\he(\bi)$ (by the discussion in the paragraph above Definition \ref{MAHA}), it follows that the assumption (O1) is satisfied.

It remains to verify the assumption (O2). Let $1\leq r<k\leq n$ and $\bi\in I^\beta$ with $i_r\neq i_k$. It is clear that $(\hX_r-\hX_{k})\he(\bi)$ commutes with any $\hX_s\he(\bj)$. Thus it suffices to verify the assumption (O2) in Lemma \ref{OreLocal} for $a:=\he(w\bi)\hT_{w,\bi}\he(\bi)$, where $w\in\Sym_n$, $\bi\in I^\beta$.

Let $1\neq z$ be a product of some elements in the following set: $$
\{\hX_r-\hX_s|1\leq r<s\leq n, i_r\neq i_s\}  .
$$
We define $$
\Sym(\bi):=\{\sigma\in\Sym_n|\sigma\bi=\bi\}, \quad\, H(\bi):=\prod_{\sigma\in\Sym(\bi)}(\sigma w^{-1})(z)\he(\bi) .
$$
We claim that $0\neq \he(w\bi)\hT_{w,\bi}\he(\bi)H(\bi)\in z\he(w\bi)\check{\H}_{\beta}\he(\bi)$.

First, using the basis Theorem \ref{basis1}, we can see that $\he(w\bi)\hT_{w,\bi}\he(\bi)H(\bi)\neq 0$. Recall that $$\widehat{T}_{w,\bi}:=\bigl(\he(w\bi)\hT_{j_1}\he(s_{j_1}w\bi)\bigr)
\bigl(\he(s_{j_1}w\bi)\hT_{j_2}\he(s_{j_2}s_{j_1}w\bi)\bigr)\dots \bigl(\he(s_{j_k}\bi)\hT_{j_k}\he(\bi)\bigr),$$ where $s_{j_1}\dots s_{j_k}$ is a prefixed reduced expression of $w$. If $i_{j_k}\neq i_{j_k+1}$, then by the defining relations in (\ref{IJ5}) we can get that \begin{equation}\label{commuExchange}
\he(s_{j_k}\bi)\hT_{j_k}\he(\bi)H(\bi)=\he(s_{j_k}\bi)(s_{j_k}H(\bi))\hT_{j_k}\he(\bi) .
\end{equation}

If $i_{j_k}=i_{j_k+1}$, i.e., $s_{j_k}\bi=\bi$ and hence $s_{j_k}\in\Sym(\bi)$. Then we know that $H(\bi)$ is symmetric in $\hX_{j_k}, \hX_{j_k+1}$. It follows from the defining relations (\ref{IJ5}) again that $$
\he(s_{j_k}\bi)\hT_{j_k}\he(\bi)H(\bi)=\he(\bi)H(\bi)\hT_{j_k}\he(\bi)=
\he(\bi)\prod_{\sigma\in\Sym(\bi)}(\sigma w^{-1})(z)\he(\bi)\hT_{j_k}e(\bi) .
$$
Note that in this case $s_{j_k}\Sym(\bi)=\Sym(\bi)$, so we can write the above equality as $\he(s_{j_k}\bi)\hT_{j_k}\he(\bi)H(\bi)=\he(s_{j_k}\bi)(s_{j_k}H(\bi))\he(s_{j_k}\bi)\hT_{j_k}\he(\bi)$. In other words, the equality (\ref{commuExchange}) still holds in this case.

By an easy induction on $\ell(w)$, we see that $$
\he(w\bi)\hT_{w,\bi}\he(\bi)H(\bi)=\he(w\bi)(wH(\bi))\he(w\bi)\hT_{w,\bi}\he(\bi) .
$$
Note that $\he(w\bi)(wH(\bi))=\he(w\bi)\prod_{\sigma\in\Sym(\bi)}(w\sigma w^{-1})(z)\he(w\bi)$ which has $z\he(w\bi)$ as a left factor. This proves our claim.
Finally, applying the anti-involution $\ast$ of $\widehat{\H}_\beta(q)$, we see that $$0\neq H(\bi)\he(\bi)\hT_{w^{-1},\bi}\he(w\bi)\in \he(\bi)\check{\H}_{\beta}\he(w\bi)z. $$
Thus the assumption (O2) is satisfied. Therefore, this lemma is a direct consequence of Lemma \ref{OreLocal}.
\end{proof}


For any $1\leq k\leq n$, $\bi\in I^\beta$ and $w\in\Sym_n$, we define $w(\hX_k^{\pm 1}\he(\bi)):=\hX_{w(k)}^{\pm 1}\he(w\bi)$, $w(\hx_k\he(\bi)):=\hx_{w(k)}\he(w\bi)$. By Lemma \ref{zeiKeylem}, this is well-defined and extends uniquely to an action of $\Sym_n$ on the set of polynomials in $\{\hX_k^{\pm 1}\he(\bi)|1\leq k\leq n, \bi\in I^{\beta}\}$ and on the set of polynomials in $\{\hx_k\he(\bi)|1\leq k\leq n, \bi\in I^{\beta}\}$ respectively.

\begin{dfn} An element $f$ in $K[\hX_k^{\pm 1}\he(\bi),\he(\bi)|1\leq k\leq n,\bi\in I^\beta]$ (respectively, in $K[\hx_k\he(\bi),\he(\bi)|1\leq k\leq n,\bi\in I^\beta]$) is said to be symmetric in $\{\hX_k^{\pm 1}\he(\bi),\he(\bi)|1\leq k\leq n, \bi\in I^{\beta}\}$ (respectively, in  $\{\hx_k\he(\bi),\he(\bi)|1\leq k\leq n, \bi\in I^{\beta}\}$) if $w(f)=f$ for any $w\in\Sym_n$.
\end{dfn}

\begin{lem} Any symmetric element in $K[\hX_k^{\pm 1}\he(\bi),\he(\bi)|1\leq k\leq n,\bi\in I^\beta]$ is in the center of  $Z(\widehat{\H}_\beta(q))$.
Any symmetric element in $K[\hx_k\he(\bi),\he(\bi)|1\leq k\leq n,\bi\in I^\beta]$ is in the center of  $Z(\widehat{H}_\beta(q))$.
\end{lem}

\begin{proof} We only prove the non-degenerate case as the degenerate case is similar. By the relations (\ref{TREI}) and (\ref{IJ5}), it is easy to see that if $i_r=i_{r+1}$, then $$\begin{aligned}
&\he(\bi)\hT_r\he(\beta)=\hT_r\he(\beta)\he(\bi),\quad \hX_r\hX_{r+1}\he(\bi)\hT_r\he(\beta)=\hT_r\he(\beta)\hX_r\hX_{r+1}\he(\bi),\\
&(\hX_r+\hX_{r+1})\he(\bi)\hT_r\he(\beta)=\hT_r\he(\beta)(\hX_r+\hX_{r+1})\he(\bi) ;
\end{aligned}$$
while if $i_r\neq i_{r+1}$, then $$\begin{aligned}
(\he(\bi)+\he(s_r\bi))\bigr)\hT_r\he(\beta)&=\hT_r\he(\beta)\bigl(\he(\bi)+\he(s_r\bi)\bigr),\\ \bigl(\hX_r\hX_{r+1}(\he(\bi)+\he(s_r\bi))\bigr)\hT_r\he(\beta)&=\hT_r\he(\beta)\bigl((\he(\bi)+\he(s_r\bi))\hX_r\hX_{r+1}\he(\beta)\bigr).\end{aligned}
$$
In view of the relations (\ref{XkEi}), (\ref{FEI}), (\ref{TREI}), (\ref{IJ5}), (\ref{IJ5a}), (\ref{IJ5b}), it suffices to show that for any $1\leq r<n$ and any $\bi\in I^\beta$ with $i_r\neq i_{r+1}$, $$
\bigl(\hX_r\he(\bi)+\hX_{r+1}\he(s_r\bi)\bigr)\hT_r\he(\beta)=\hT_r\he(\beta)\bigl(\hX_r\he(\bi)+\hX_{r+1}\he(s_r\bi)\bigr),\\
$$

In fact, the left-hand side of the above equality is equal to $$\begin{aligned}
&\quad\,\bigl(\hX_r\he(\bi)+\hX_{r+1}\he(s_r\bi)\bigr)\hT_r(\he(\bi)+\he(s_r\bi))\\
&=\he(\bi)\hX_r\hT_r\he(\bi)+\he(s_r\bi)\hX_{r+1}\hT_r\he(s_r\bi)+\he(\bi)\hX_r\hT_r\he(s_r\bi)+\he(s_r\bi)\hX_{r+1}\hT_r\he(\bi)\\
&=\he(\bi)\hT_r\hX_r\he(\bi)+\he(s_r\bi)\hT_r\hX_{r+1}\he(s_r\bi)+\he(\bi)(\hT_r\hX_{r+1}-(q-1)\hX_{r+1})\he(s_r\bi)\\
&\qquad +\he(s_r\bi)(\hT_r\hX_{r}+(q-1)\hX_{r+1})\he(\bi)\\
&=\he(\bi)\hT_r\hX_r\he(\bi)+\he(s_r\bi)\hT_r\hX_{r+1}\he(s_r\bi)+\he(\bi)\hT_r\hX_{r+1}\he(s_r\bi)+\he(s_r\bi)\hT_r\hX_{r}\he(\bi)\\
&=\hT_r\he(\beta)\bigl(\hX_r\he(\bi)+\hX_{r+1}\he(s_r\bi)\bigr),\end{aligned}
$$
as required, where we have used the fact $\he(\bi)\hX_{r+1}\he(s_r\bi)=0=\he(s_r\bi)\hX_{r+1}\he(\bi)$ (because $i_r\neq i_{r+1}$) in the third equality.
\end{proof}

The next result describe the center for the modified affine Hecke algebras $\widehat{\H}_\beta(q)$ and $\widehat{H}_\beta$.

\begin{thm} \label{center0} The center $Z(\widehat{\H}_\beta(q))$ of $\widehat{\H}_\beta(q)$ is equal to $$
\Biggl\{\sum_{\bi\in I^\beta}f(\bi)g(\bi)^{-1}\he(\bi)\Biggm|\begin{matrix}\text{$f(\bi)\in K[\hX_k^{\pm 1}\he(\bj),\he(\bj)|1\leq k\leq n, \bj\in I^\beta]$}\\
\text{$g(\bi)=\prod_{\substack{1\leq r<s\leq n\\ i_r\neq i_s}}(\hX_r-\hX_s)^{d_{r,s,\bi}}, d_{r,s,\bi}\in\N$,}\\
\text{$\sigma(f(\bi))\sigma(g(\bi))^{-1}=f(\sigma\bi)g(\sigma\bi)^{-1}, \forall\,\sigma\in\Sym_n$.}
\end{matrix}\Biggr\} ,
$$
and $Z(\widehat{\H}_\beta(q))\cap K[\hX_k^{\pm 1}\he(\bj),\he(\bj)|1\leq k\leq n, \bj\in I^\beta]$ is the set of symmetric elements in $K[\hX_k^{\pm 1}\he(\bj),\he(\bj)|1\leq k\leq n, \bj\in I^\beta]$. Similarly, the center $Z(\widehat{H}_\beta)$ of $\widehat{H}_\beta$ is equal to $$
\Biggl\{\sum_{\bi\in I^\beta}f(\bi)g(\bi)^{-1}\he(\bi)\Biggm|\begin{matrix}\text{$f(\bi)\in K[\hx_k\he(\bj),\he(\bj)|1\leq k\leq n, \bj\in I^\beta]$}\\
\text{$g(\bi)=\prod_{\substack{1\leq r<s\leq n\\ i_r\neq i_s}}(\hx_r-\hx_s)^{d_{r,s,\bi}}, d_{r,s,\bi}\in\N$,}\\
\text{$\sigma(f(\bi))\sigma(g(\bi))^{-1}=f(\sigma\bi)g(\sigma\bi)^{-1}, \forall\,\sigma\in\Sym_n$.}
\end{matrix}\Biggr\} ,
$$
and $Z(\widehat{H}_\beta(q))\cap K[\hx_k\he(\bj),\he(\bj)|1\leq k\leq n, \bj\in I^\beta]$ is the set of symmetric elements in $K[\hx_k\he(\bj),\he(\bj)|1\leq k\leq n, \bj\in I^\beta]$.
\end{thm}

\begin{proof} We only prove the theorem in the non-degenerate case, while the degenerate case is similar.

Suppose that $z=\sum_{\bi\in I^\beta}f(\bi)g(\bi)^{-1}\he(\bi)$, where $f(\bi)\in K[\hX_k^{\pm 1}\he(\bj),\he(\bj)|1\leq k\leq n, \bj\in I^\beta]$,
$g(\bi)=\prod_{\substack{1\leq r<s\leq n\\ i_r\neq i_s}}(\hX_r-\hX_s)^{d_{r,s,\bi}}, d_{r,s,\bi}\in\N$ for each pair $(r,s)$, and $\sigma(f(\bi))\sigma(g(\bi))^{-1}=f(\sigma\bi)g(\sigma\bi)^{-1}, \forall\,\sigma\in\Sym_n$.

By the relations (\ref{6}), (\ref{7}), it is easy to see that for any $1\leq r<n$, \begin{equation}\label{222}
\he(\bi)f(\bi)\hT_r\he(s_r\bi)-\he(\bi)\hT_rs_r(f(\bi))\he(s_r\bi)=(q-1)e(\bi)\hX_{r+1}\frac{f(\bi)-s_r(f(\bi))}{\hX_{r+1}-\hX_r}\he(s_r\bi),
\end{equation}
and \begin{equation}\label{222right}
\he(\bi)g(\bi)\hT_r\he(s_r\bi)-\he(\bi)\hT_rs_r(g(\bi))\he(s_r\bi)=(q-1)e(\bi)\hX_{r+1}\frac{g(\bi)-s_r(g(\bi))}{\hX_{r+1}-\hX_r}\he(s_r\bi).
\end{equation}

Now, we have that $$\begin{aligned}
&\quad\,\he(\bi)f(\bi)g(\bi)^{-1}\he(\bi)\hT_r\he(s_r\bi)-\he(\bi)\hT_r\he(s_r\bi)f(s_r\bi)g(s_r\bi)^{-1}\he(s_r\bi)\\
&=g(\bi)^{-1}\he(\bi)\Bigl(f(\bi)\he(\bi)\hT_rg(s_r\bi)\he(s_r\bi)-\he(\bi)g(\bi)\hT_r\he(s_r\bi)f(s_r\bi)\Bigr)g(s_r\bi)^{-1}\he(s_r\bi)\\
&=g(\bi)^{-1}\he(\bi)\Bigl(\he(\bi)\hT_rs_r(f(\bi))g(s_r\bi)\he(s_r\bi)-\he(\bi)\hT_rs_r(g(\bi))\he(s_r\bi)f(s_r\bi)\Bigr)g(s_r\bi)^{-1}\he(s_r\bi)\\
&\qquad+(q-1)e(\bi)\hX_{r+1}\frac{(f(\bi)-s_r(f(\bi)))g(s_r\bi)-(g(\bi)-s_r(g(\bi)))f(s_r\bi)}{\hX_{r+1}-\hX_r}\he(s_r\bi)\\
&=(q-1)e(\bi)\hX_{r+1}\Bigl(\frac{(f(\bi)g(s_r\bi)-g(\bi)f(s_r\bi))-(s_r(f(\bi))g(s_r\bi)-s_r(g(\bi)))f(s_r\bi)}{\hX_{r+1}-\hX_r}\Bigr)\he(s_r\bi) ,
\end{aligned}
$$
where we have used the assumption that $s_r(f(\bi))\sigma(g(\bi))^{-1}=f(\sigma\bi)g(\sigma\bi)^{-1}$ ($\forall\,1\leq r<n$) in the last equality.

If $j_r\neq j_{r+1}$, then the last term is zero because $e(\bi)he(s_r\bi)=0$ for any $f\in K[\hX_k^{\pm 1}\he(\bj),\he(\bj)|1\leq k\leq n,\bj\in I^\beta]$, which implies that $\he(\bi)f(\bi)g(\bi)^{-1}\he(\bi)\hT_r\he(s_r\bi)-\he(\bi)\hT_r\he(s_r\bi)f(s_r\bi)g(s_r\bi)^{-1}\he(s_r\bi)=0$ .

If $j_r=j_{r+1}$, then the term in the big bracket is zero by the condition that $f(\bi)g(\bi)^{-1}=f(s_r\bi)g(s_r\bi)^{-1}=s_rf(\bi)s_r(g(\bi))^{-1}$.

Therefore,  $\he(\bi)f(\bi)g(\bi)^{-1}\he(\bi)\hT_r\he(s_r\bi)-\he(\bi)\hT_r\he(s_r\bi)f(s_r\bi)g(s_r\bi)^{-1}\he(s_r\bi)=0$ always holds, which implies that
$z\he(\bi)T_r\he(s_r\bi)=\he(\bi)T_r\he(s_r\bi)z$ and hence $z\in Z(\widehat{\H}_\beta(q))$ as required.

Conversely, suppose that $z=\sum_{\bi,w}\he(w\bi)\widehat{T}_{w,\bi}f_w\in Z(\widehat{\H}_\beta(q))$, where $$
f_w\in K[\hX_1^{\pm 1}\he(\bi),\dots,\hX_n^{\pm 1}\he(\bi),(\hX_r-\hX_{s})^{-1}\he(\bi)|\bi\in I^\beta, 1\leq r<s\leq n, i_r\neq i_{s}].
$$
Since $\he(\bi)z=z\he(\bi)$, we can rewrite $z$ as $$
z=\sum_{\substack{\bi\in I^{\beta}, w\in\Sym_n\\ w(\bi)=\bi}}\he(\bi)\widehat{T}_{w,\bi}f_w\he(\bi)\in Z(\widehat{\H}_\beta(q)) .
$$

Suppose that $z\notin K[\hX_1^{\pm 1}e(\bi),\dots,\hX_n^{\pm 1}e(\bi),(\hX_r-\hX_{s})^{-1}\he(\bi)|\bi\in I^\beta, 1\leq r<s\leq n, , i_r\neq i_{s}]$. Let $u$ be maximal with respect to the Bruhat partial order ``$<$" such that $f_u\neq 0$, $u(\bi)=\bi$ and $u\neq 1$. Then $u(r)\neq r$ for some $1\leq r\leq n$.
By definition of center, we have that $\hX_r\he(\beta)z=z\hX_r\he(\beta)$. However, by an easy induction based on Lemma \ref{affAct1b}, we can get that $$
\hX_r\he(\bi)\widehat{T}_{u,\bi}\he(\bi)=\he(\bi)\hX_r\widehat{T}_{u,\bi}\he(\bi)=e(\bi)\Bigl(\widehat{T}_{u,\bi}\hX_{u^{-1}r}+\sum_{w'<u}\widehat{T}_{w',\bi}g_{w'}\Bigr)\he(\bi),
$$
where $g_{w'}\in K[\hX_1,\dots,\hX_n]\he(\beta)$. It follows that the coefficient of $\he(\bi)\widehat{T}_{u,\bi}\hX_{u^{-1}r}f_u\he(\bi)$ is different in $\hX_rz$ and $z\hX_r$, a contradiction.
Therefore, $$
z\in K[\hX_1^{\pm 1}\he(\bi),\dots,\hX_n^{\pm 1}\he(\bi),(\hX_r-\hX_{s})^{-1}\he(\bi)|\bi\in I^{\beta}, 1\leq r<s\leq n, i_r\neq i_{s}] .
$$
We can write $z=\sum_{\bi\in I^\beta}f(\bi)g(\bi)^{-1}\he(\bi)$, where $f(\bi)\in K[\hX_k^{\pm 1}\he(\bj),\he(\bj)|1\leq k\leq n, \bj\in I^\beta]$,
$g(\bi)=\prod_{\substack{1\leq r<s\leq n\\ i_r\neq i_s}}(\hX_r-\hX_s)^{d_{r,s,\bi}}, d_{r,s,\bi}\in\N$ for each pair $(r,s)$, and $\sigma(f(\bi))\sigma(g(\bi))^{-1}=f(\sigma\bi)g(\sigma\bi)^{-1}, \forall\,\sigma\in\Sym_n$.

For each $1\leq r<n$, $z\he(\bi)\hT_r\he(s_r\bi)=\he(\bi)\hT_r\he(s_r\bi)z$, which implies (by the previous proof) that $$\begin{aligned}
&\quad\,(q-1)e(\bi)\hX_{r+1}\Bigl(\frac{(f(\bi)g(s_r\bi)-g(\bi)f(s_r\bi))-(s_r(f(\bi))g(s_r\bi)-s_r(g(\bi)))f(s_r\bi)}{\hX_{r+1}-\hX_r}\Bigr)\he(s_r\bi)\\
&=\he(\bi)f(\bi)g(\bi)^{-1}\he(\bi)\hT_r\he(s_r\bi)-\he(\bi)\hT_r\he(s_r\bi)f(s_r\bi)g(s_r\bi)^{-1}\he(s_r\bi)=0.
\end{aligned}
$$
If $i_r=i_{r+1}$ then $s_r(\bi)=\bi$ and by the basis Theorem \ref{basis1}, we can deduce that $$
s_r(f(\bi)g(\bi)^{-1})=f(\bi)g(\bi)^{-1}=f(s_r\bi)g(s_r\bi)^{-1} .
$$

if $i_r\neq i_{r+1}$, then by the previous proof and the fact that $e(\bi)he(s_r\bi)=0$ for any $f\in K[\hX_k^{\pm 1}\he(\bj),\he(\bj)|1\leq k\leq n,\bj\in I^\beta]$, we have that $$\begin{aligned}
0&=\he(\bi)f(\bi)g(\bi)^{-1}\he(\bi)\hT_r\he(s_r\bi)-\he(\bi)\hT_r\he(s_r\bi)f(s_r\bi)g(s_r\bi)^{-1}\he(s_r\bi)\\
&=g(\bi)^{-1}\he(\bi)\Bigl(f(\bi)\he(\bi)\hT_rg(s_r\bi)\he(s_r\bi)-\he(\bi)g(\bi)\hT_r\he(s_r\bi)f(s_r\bi)\Bigr)g(s_r\bi)^{-1}\he(s_r\bi)\\
&=g(\bi)^{-1}\he(\bi)\Bigl(\he(\bi)\hT_rs_r(f(\bi))g(s_r\bi)\he(s_r\bi)-\he(\bi)\hT_rs_r(g(\bi))\he(s_r\bi)f(s_r\bi)\Bigr)g(s_r\bi)^{-1}\he(s_r\bi),
\end{aligned}
$$
which (by the basis Theorem \ref{basis1} again) implies that
$$f(s_r\bj)g(s_r\bi)^{-1}=s_r(f(\bi))s_r(g(\bi))^{-1},
$$
as required. This completes the proof of the first statement of the theorem in the non-degenerate case. The second statement follows from the first statement and the basis Theorem \ref{basis1}.
\end{proof}


\begin{cor} \label{center000} The center $Z(\widehat{\HH}^+_\beta)$ of $\widehat{\HH}^+_\beta$ is $$
\Biggl\{\sum_{\bi\in I^\beta}f(\bi)g(\bi)^{-1}\he(\bi)\Biggm|\begin{matrix}\text{$f(\bi)\in K[\hX_k\he(\bj),\he(\bj)|1\leq k\leq n, \bj\in I^\beta]$}\\
\text{$g(\bi)=\prod_{\substack{1\leq r<s\leq n\\ i_r\neq i_s}}(\hX_r-\hX_s)^{d_{r,s,\bi}}, d_{r,s,\bi}\in\N$,}\\
\text{$\sigma(f(\bi))\sigma(g(\bi))^{-1}=f(\sigma\bi)g(\sigma\bi)^{-1}, \forall\,\sigma\in\Sym_n$.}
\end{matrix}\Biggr\} ,
$$
\end{cor}

\begin{proof} This follows from Theorem \ref{333} and a similar argument to that used in the proof Theorem \ref{center0}.
\end{proof}

\begin{cor} \label{center00a} We have that $$\begin{aligned}
&K[\hX_k\he(\bi),\he(\bi)|1\leq k\leq n, \bi\in I^\beta]\cap Z(\widehat{\H}_\beta(q))=\\
&\qquad\qquad K[\hX_k\he(\bi),\he(\bi)|1\leq k\leq n, \bi\in I^\beta]\cap Z(\widehat{\H}^+_\beta(q)),
\end{aligned}$$  which is equal to the set of symmetric elements in $\{\hX_k\he(\bi),\he(\bi)|1\leq k\leq n, \bi\in I^\beta\}$.
\end{cor}

\begin{proof} This follows from Corollary \ref{center000} and (\ref{333}).
\end{proof}

\begin{dfn} Let $\bj\in I^{\beta}$. If $\widehat{\H}_\beta=\widehat{\H}_\beta(q)$, then we define $$
M_n(\bj):=\bigl\{\he(\bj)\bigr\}\cup\biggl\{(\hX_r-q\hX_{s})\he(\bj)\Biggm|1\leq r\neq s\leq n, \bj\in I^\beta, j_r\neq 1+j_{s}\biggr\}.
$$
If $\widehat{\H}_\beta=\widehat{H}_\beta$, then we define $$
M'_n(\bj):=\bigl\{\he(\bj)\bigr\}\cup\biggl\{(\hx_r-\hx_{s}-1)\he(\bj)\biggm|1\leq r\neq s\leq n, \bj\in I^\beta, j_r\neq 1+j_{s}\biggr\}.
$$
Let $\Sigma_n(\bj)$ and $\Sigma'_n(\bj)$ be the multiplicative closures of $M_n(\bj)$ and $M'_n(\bj)$ respectively.
\end{dfn}

\begin{dfn2} \label{Localised1} All the assumptions and conditions in Lemma \ref{OreLocal} of the Appendix are satisfied if we take $$\begin{aligned}
& A=\widehat{\H}_\beta(q),\,\,J:=I^\beta,\,\,e_\bi:=\he(\bi),\,\,S_\bi:=\Sigma_n(\bi),\,\,\text{for $\bi\in I^\beta$},\\
& A_0:=K[\hX_k^{\pm 1}\he(\bj),\he(\bj),(\hX_a-\hX_b)^{-1}\he(\bi)|1\leq k\leq n, 1\leq a<b\leq n, \bi,\bj\in I^\beta, i_a\neq i_b].
\end{aligned}$$
Moreover, the resulting generalized Ore localization $\widetilde{\H}_\beta(q)$ is canonically isomorphic to the subalgebra of $\underset{\Lambda}{\underleftarrow{\lim}}\,\H_n^\Lambda(q)$ generated by $\widehat{\H}_\beta(q)$ and the elements in the following subset \begin{equation}\label{unit}
\biggl\{(\hX_r-q\hX_{s})^{-1}\he(\bj)\biggm|1\leq r\neq s\leq n, \bj\in I^\beta, j_r\neq 1+j_{s}\biggr\} .
\end{equation}
A similar statement holds if we replace $\Sigma_n(\bj), \widehat{\H}_\beta(q)$ and $\widetilde{\H}_\beta(q)$ by $\Sigma'_n(\bj), \widehat{H}_\beta$ and $\widetilde{H}_\beta$ respectively. In particular, the generalized Ore localization $\widetilde{H}_\beta$ is canonically isomorphic to the subalgebra of $\underset{\Lambda}{\underleftarrow{\lim}}\,H_n^\Lambda$ generated by $\widehat{H}_\beta$ and the elements in the following subset \begin{equation}\label{unitb}
\biggl\{(\hx_r-\hx_{s}-1)^{-1}\he(\bj)\biggm|1\leq r\neq s\leq n, \bj\in I^\beta, j_r\neq 1+j_{s}\biggr\}.
\end{equation}
\end{dfn2}

\begin{proof} This follows from a similar argument in the proof of Lemma \ref{Local1}.
\end{proof}

\begin{dfn} Let $\bj\in I^{\beta}$. In the non-degenerate setting, we define $$
\hat{M}_n(\bj):=\Biggl\{\begin{matrix}\bigl((1-y_r)-q^{j_{s}-j_r+1}(1-y_s)\bigr)e(\bj), e(\bj),\\
\bigl((1-y_a)-q^{i_b-i_a}(1-y_b)\bigr)e(\bi)\end{matrix}\Biggm|\begin{matrix}1\leq r,s,a,b\leq n, \bi,\bj\in I^\beta,\\
a\neq b, j_r\neq 1+j_{s},i_a\neq i_b \end{matrix}\Bigr\};
$$
while in the degenerate setting, we define $$
\hat{M}'_n(\bj):=\Biggl\{\begin{matrix}(j_r-j_s-1-y_s+y_r)e(\bj),\\
\bigl(i_a-i_b+y_a-y_b\bigr)e(\bi)\end{matrix}\Biggm|\begin{matrix}1\leq r,s,a,b\leq n, \bi,\bj\in I^\beta, \\ a\neq b, j_r\neq 1+j_{s}, i_a\neq i_b\end{matrix}\Bigr\} .
$$
Let $\hat{\Sigma}_n(\bj)$ and $\hat{\Sigma}'_n(\bj)$ be the multiplicative closures of $\hat{M}_n(\bj)$ and in $\hat{M}'_n(\bj)$
respectively.
\end{dfn}

In a similar way as Theorem \ref{Localised1}, we are going to use Lemma \ref{OreLocal} to construct, in the non-degenerate setting, a bigger ring $\widetilde{\mathscr{R}}_\beta$ which contains ${\mathscr{R}}_\beta$ and the elements in following subset \begin{equation}\label{unit2}
\Biggl\{\begin{matrix}\bigl((1-y_r)-q^{j_{s}-j_r+1}(1-y_s)\bigr)^{-1}e(\bj),\\
\bigl((1-y_a)-q^{i_b-i_a}(1-y_b)\bigr)^{-1}e(\bi)\end{matrix}\Biggm|\begin{matrix}1\leq r,s,a,b\leq n, \bi,\bj\in I^\beta,\\  a\neq b, j_r\neq 1+j_{s},i_a\neq i_b \end{matrix}\Biggr\};
\end{equation}
and in the degenerate setting, a bigger ring $\widetilde{\mathscr{R}}'_\beta$ which contains ${\mathscr{R}}_\beta$ and the elements in following subset \begin{equation}\label{unit3}
\Biggl\{\begin{matrix}(j_r-j_s-1-y_s+y_r)^{-1}e(\bj),\\
\bigl(i_a-i_b+y_a-y_b\bigr)^{-1}e(\bi)\end{matrix}\Biggm|\begin{matrix}1\leq r,s,a,b\leq n, \bi,\bj\in I^\beta,\\ a\neq b, j_r\neq 1+j_{s}, i_a\neq i_b\end{matrix}\Biggr\}.
\end{equation}

For each $w\in\Sym_n$, we fix a reduced expression $s_{j_1}\dots s_{j_k}$ of $w$ and define $$
\psi_w:=\psi_{j_1}\dots\psi_{j_k}.
$$

\begin{lem} \label{basis4} (cf. \cite{KhovLaud:diagI}) The following set $$
\bigl\{\psi_wy_1^{a_1}\dots y_n^{a_n}e(\bi)\bigm|w\in\Sym_n, \bi\in I^{\beta}, a_1,\dots,a_n\in\N\bigr\}
$$
is a $K$-basis of $\mathscr{R}_\beta$.
\end{lem}

\begin{dfn2} \label{Localised2} All the assumptions and conditions in Lemma \ref{OreLocal} are satisfied if we take $$\begin{aligned}
& A=\mathscr{R}_\beta,\,\,A_0:=K[y_ke(\bi),e(\bi)|1\leq k\leq n,\bi\in I^\beta],\\
& J:=I^\beta,\,\,\, e_\bi:=e(\bi),\,\,S_\bi:=\hat{\Sigma}_n(\bi),\,\,\text{for $\bi\in I^\beta$}.
\end{aligned}$$
In particular, we can embedded ${\mathscr{R}}_\beta$ into $\widetilde{\mathscr{R}}_\beta:=\widetilde{A}$ which is generated by elements in ${\mathscr{R}}_\beta$ together with the elements in the subset (\ref{unit2}). A similar statement holds if we replace $\hat{\Sigma}_n(\bj), \mathscr{R}_\beta,
\widetilde{\mathscr{R}}_\beta$ and (\ref{unit2}) by $\hat{\Sigma}'_n(\bj), \mathscr{R}_\beta,
\widetilde{\mathscr{R}}'_\beta$ and (\ref{unit3}) respectively.
\end{dfn2}

\begin{proof} It suffices to show that for any $w\in\Sym_n$, $\bi\in I^\beta$, $e(w\bi)\neq ze(w\bi)\in\hat{\Sigma}_n(w\bi)$, \begin{align}
 e(w\bi)\psi_we(\bi)\hat{\Sigma}_n(\bi)\bigcap ze(w\bi){\mathscr{R}}_{\beta}e(\bi)&\neq\emptyset ,\label{323}
\end{align}
where $z$ is a product of some elements in the following set $$
\Biggl\{\begin{matrix}\bigl((1-y_r)-q^{j_{s}-j_r+1}(1-y_s)\bigr)^{-1},\\
\bigl((1-y_a)-q^{i_b-i_a}(1-y_b)\bigr)^{-1}\end{matrix}\Biggm|\begin{matrix}1\leq r,s,a,b\leq n, a\neq b,\\ i_r\neq 1+i_{s},i_a\neq i_b \end{matrix}\Biggr\}.
$$

We define $$
\Sym(\bi):=\{\sigma\in\Sym_n|\sigma\bi=\bi\}, \quad\, G(\bi):=\prod_{\sigma\in\Sym(\bi)}\sigma(w^{-1}z)e(\bi) .
$$
We claim that $0\neq e(w\bi)\psi_we(\bi)G(\bi)\in ze(w\bi){\mathscr{R}}_{\beta}e(\bi)$.

First, by Lemma \ref{basis4}, it is clear that $e(w\bi)\psi_we(\bi)G(\bi)\neq 0$. Recall that $\psi_w:=\psi_{j_1}\dots\psi_{j_k}$, where $s_{j_1}\dots s_{j_k}$ is a prefixed reduced expression of $w$. If $i_{j_k}\neq i_{j_k}+1$, then by the defining relations for the quiver Hecke algebra ${\mathscr{R}}_\beta$, we can get that \begin{equation}\label{commutatorExchange2}
\psi_{j_k}e(\bi)G(\bi)=\psi_{j_k}\prod_{\sigma\in\Sym(\bi)}\sigma(w^{-1}z)e(\bi)=\Bigl(\prod_{\sigma\in\Sym(\bi)}(s_{j_k}\sigma)(w^{-1}z)e(s_{j_k}\bi)\Bigr)
\psi_{j_k}e(\bi) .
\end{equation}

If $i_{j_k}=i_{j_k}+1$, i.e., $s_{j_k}\bi=\bi$ and hence $s_{j_k}\in\Sym(\bi)$. Then we know that $G(\bi)$ is symmetric in $y_{j_k}, y_{j_k+1}$. It follows from the defining relations for the quiver Hecke algebra ${\mathscr{R}}_\beta$ again that $$
\psi_{j_k}e(\bi)G(\bi)=e(s_{j_k}\bi)G(\bi)e(s_{j_k}\bi)\psi_{j_k}e(\bi)=
e(s_{j_k}\bi)\Bigl(\prod_{\sigma\in\Sym(\bi)}\sigma(w^{-1}z)e(s_{j_k}\bi)\Bigr)\psi_{j_k}e(\bi) .
$$
Note that in this case $s_{j_k}\Sym(\bi)=\Sym(\bi)$, so we can write the above equality as $\psi_{j_k}e(\bi)G(\bi)=e(s_{j_k}\bi)(s_{j_k}G(\bi))e(s_{j_k}\bi)\psi_{j_k}e(\bi)$. In other words, the equality (\ref{commutatorExchange2}) still holds in this case.

By an easy induction on $\ell(w)$, we see that $$
e(w\bi)\psi_we(\bi)G(\bi)=e(w\bi)(wG(\bi))e(w\bi)\psi_we(\bi) .
$$
Note that $e(w\bi)(wG(\bi))=e(w\bi)\prod_{\sigma\in\Sym(\bi)}(w\sigma w^{-1})(z)e(w\bi)$ which has $ze(w\bi)$ as a left factor. This proves (\ref{323}), Hence we prove the first half of the theorem. The second half of the theorem can be proved in a similar way.
\end{proof}

\bigskip

\section{The isomorphisms between $\widetilde{\mathscr{R}}_\beta, \widetilde{\mathscr{R}}'_\beta$ and $\widetilde{\mathscr{H}}_\beta(q), \widetilde{H}_\beta$}

The purpose of this section is to construct the $K$-algebra isomorphisms: $\widetilde{\mathscr{R}}_\beta\cong \widetilde{\mathscr{H}}_\beta(q)$, $\widetilde{\mathscr{R}}'_\beta\cong\widetilde{H}_\beta$ (Theorem \ref{mainthm0a} and \ref{mainthm0b}), which will lift Brundan-Kleshchev's isomorphism.\smallskip

By abuse of notations, we write $e(\beta):=\sum_{\bi\in I^{\beta}}e(\bi)\in\mathscr{H}_\beta^{\Lambda}$. Recall that $e(\beta)\neq 0$ if and only if $\mathscr{H}_\beta^{\Lambda}\neq 0$ and if and only if $\R[\beta]\neq 0$. Assume that $e(\beta)\neq 0$. Let $p^\Lambda: \mathscr{R}_{\beta}\twoheadrightarrow\R[\beta]$ be the canonical surjective algebra homomorphism.

Let $\bi\in I^{\beta}$. It is easy to see that every element in $\hat{M}_n(\bi)$ (respectively, in $\hat{M}'_n(\bi)$) is sent to an invertible element in $\R[\beta]$ because $y_r, y_s$ are commuting nilpotent elements in $\R[\beta]$. It follows that the map ${p}^\Lambda$ naturally induces a surjective algebra homomorphism $p_1(\Lambda): \widetilde{\mathscr{R}}_{\beta}\twoheadrightarrow\R[\beta]$ and a surjective algebra homomorphism $p_2(\Lambda): \widetilde{\mathscr{R}}'_{\beta}\twoheadrightarrow\R[\beta]$.

\begin{thm} \label{mainthm0a} In the non-degenerate case, there is a $K$-algebra isomorphism $\theta: \widetilde{\mathscr{R}}_\beta\cong\widetilde{\mathscr{H}}_\beta(q)$, such that $e(\bi)\mapsto \he(\bi)$, $y_se(\bi)\mapsto \he(\bi)(1-q^{-i_s}\hX_s)\he(\bi)$ and $$
\psi_re(\bi)\mapsto\begin{cases}
q^{i_r}(\hT_r+1)(\hX_r-q\hX_{r+1})^{-1}\he(\bi), &\text{if $i_r=i_{r+1}$;}\\
q^{-i_{r}}\Bigl(\hT_r(\hX_r-\hX_{r+1})+(q-1)\hX_{r+1}\Bigr)\he(\bi), &\text{if $i_{r}=i_{r+1}+1$;}\\
\begin{matrix}\Bigl(\hT_r (\hX_{r+1}-\hX_{r})+(1-q)\hX_{r+1}\Bigr)\\
\times (\hX_r-q\hX_{r+1})^{-1}\he(\bi),\end{matrix} &\text{otherwise.}
\end{cases} ,
$$
for any $\bi\in I^{\beta}$, $1\leq s\leq n$ and $1\leq r<n$.

The inverse map $\eta$ is given by: $$
\eta(\he(\bi))=e(\bi),\quad \eta(\hX_s\he(\bi))=q^{i_s}(1-y_s)e(\bi),\quad
\eta(\hX_s^{-1}\he(\bi))=q^{-i_s}(1-y_s)^{-1}e(\bi), $$
and $\eta(\hT_r\he(\bi))$ is equal to $\psi_r(1-q+qy_{r+1}-y_r)e(\bi)-e(\bi)$ if  $i_r=i_{r+1}$; or $$\Bigl(q\psi_re(\bi)-(q-1)(1-y_{r+1})e(\bi)\Bigr) \Bigl(q(1-y_r)-(1-y_{r+1})\Bigr)^{-1}e(\bi),$$
if $i_{r}=i_{r+1}+1$; or otherwise $$\begin{aligned}
&\psi_r(q^{i_r}-q^{i_{r+1}+1}-q^{i_r}y_r+q^{i_{r+1}+1}y_{r+1})(q^{i_{r+1}}-q^{i_{r}}+q^{i_r}y_r-q^{i_{r+1}}y_{r+1})^{-1}e(\bi)\\
&\qquad -(1-q)q^{i_{r+1}}(1-y_{r+1})(q^{i_{r+1}}-q^{i_{r}}+q^{i_r}y_r-q^{i_{r+1}}y_{r+1})^{-1}e(\bi).
\end{aligned}
$$
\end{thm}

\begin{thm} \label{mainthm0b} In the degenerate case, there is a $K$-algebra isomorphism $\theta': \widetilde{\mathscr{R}}'_\beta\cong\widetilde{{H}}_\beta$, such that $e(\bi)\mapsto \he(\bi)$, $y_se(\bi)\mapsto \he(\bi)(\hx_s-i_s)\he(\bi)$ and $$
\psi_re(\bi)\mapsto\begin{cases}
(\hs_r+1)(1+\hx_{r+1}-\hx_r)^{-1}\he(\bi), &\text{if $i_r=i_{r+1}$;}\\
\Bigl(\hs_r(\hx_r-\hx_{r+1})+1\Bigr)\he(\bi), &\text{if $i_{r}=i_{r+1}+1$;}\\
\begin{matrix}\Bigl(\hs_r (\hx_r-\hx_{r+1})+1\Bigr)\\
\times (1+\hx_{r+1}-\hx_r)^{-1}\he(\bi),\end{matrix} &\text{otherwise.}
\end{cases} ,
$$
for any $\bi\in I^{\beta}$, $1\leq s\leq n$ and $1\leq r<n$.

The inverse map $\eta'$ is given by: $$
\eta'(\he(\bi))=e(\bi),\quad \eta'(\hx_s\he(\bi))=(y_s+i_s)e(\bi), $$
and $\eta'(\hs_r\he(\bi))$ is equal to $\psi_r(1+y_{r+1}-y_r)e(\bi)-e(\bi)$ if  $i_r=i_{r+1}$; or $$\Bigl(\psi_re(\bi)-e(\bi)\Bigr) \Bigl(1-y_{r+1}+y_{r})\Bigr)^{-1}e(\bi),$$
if $i_{r}=i_{r+1}+1$; or otherwise $$\begin{aligned}
&\psi_r\bigl(1-i_r+i_{r+1}+y_{r+1}-y_{r}\bigr)\bigl(i_r-i_{r+1}-y_{r+1}+y_{r}\bigr)^{-1}e(\bi)\\
&\qquad -(i_r-i_{r+1}-y_{r+1}+y_{r})^{-1}e(\bi).
\end{aligned}
$$
\end{thm}

By a natural restriction of the canonical map $\pi_\Lam: \underset{\rm{M}}{\underleftarrow{\lim}}\,\H_{\beta}^{\rm{M}}(q)\rightarrow{\H}_{\beta}^{\Lambda}(q)$ to the subalgebra $\widehat{\H}_{\beta}(q)$, we get a surjective algebra homomorphism $\pi_1^\Lambda: \widehat{\H}_{\beta}(q)\twoheadrightarrow\mathscr{H}_{\beta}^{\Lambda}(q)$ such that for any $\bi\in I^\beta, 1\leq r<n, 1\leq k\leq n$, $$
\pi_1^\Lambda(\he(\bi))=e(\bi),\,\, \pi_1^\Lambda(\hT_r\he(\bi))=T_r e(\bi),\,\,\pi_1^\Lambda(\hX_k\he(\bi))=L_k e(\bi), $$ in the non-degenerate case. Similarly, we have a well-defined surjective algebra homomorphism $\pi_2^\Lambda:\widehat{H}_{\beta}\twoheadrightarrow{H}_{\beta}^{\Lambda}$ such that  for any $\bi\in I^\beta, 1\leq r<n, 1\leq k\leq n$, $$
\pi_2^\Lambda(\he(\bi))=e(\bi),\,\, \pi_2^\Lambda(\hs_r\he(\bi))=s_r e(\bi),\,\,\pi_2^\Lambda(\hx_k\he(\bi))=L_k e(\bi),$$
in the degenerate case.

By a natural restriction of the canonical map $\pi_\Lam: \underset{\rm{M}}{\underleftarrow{\lim}}\,\H_{\beta}^{\rm{M}}(q)\rightarrow{\H}_{\beta}^{\Lambda}(q)$ to the subalgebra $\widetilde{\H}_{\beta}(q)$, we get a surjective algebra homomorphism $\pi_1(\Lambda): \widetilde{\H}_{\beta}(q)\twoheadrightarrow\mathscr{H}_{\beta}^{\Lambda}(q)$ in the non-degenerate case. Similarly, we have a well-defined surjective algebra homomorphism $\pi_2(\Lambda):\widetilde{H}_{\beta}\twoheadrightarrow{H}_{\beta}^{\Lambda}$ in the degenerate case.

Recall the definition of $\widehat{\H}^+_{\beta}(q)$ in Remark \ref{modifiedplus}. By a natural restriction of the canonical map $\pi_\Lam: \underset{\rm{M}}{\underleftarrow{\lim}}\,\H_{\beta}^{\rm{M}}(q)\rightarrow{\H}_{\beta}^{\Lambda}(q)$ to the subalgebra $\widetilde{\H}^+_{\beta}(q)$, we get a surjective homomorphism $\pi_+(\Lambda): \widehat{\H}^+_{\beta}(q)\twoheadrightarrow\H_{\beta}^{\Lambda}$ from $\widehat{\H}^+_{\beta}(q)$ onto $\H_{\beta}^{\Lambda}$. This surjection coincides with the composition of the natural surjective homomorphism $\pi_1^\Lambda$ from $\widehat{\H}_{\beta}(q)$ onto $\H_{\beta}^{\Lambda}$ with the natural injection $\iota$ from $\widehat{\H}^+_{\beta}(q)$ into $\widehat{\H}_{\beta}(q)$.

\begin{cor} \label{smalllem} With the notations as above, we have that $$
\Ker\pi_+(\Lambda)=\Ker\pi_1^\Lambda\bigcap\widehat{\H}^+_{\beta}(q). $$
\end{cor}

Recall that the elements $y_1e(\beta),\dots,y_ne(\beta)\in\mathscr{R}_\beta$ generate a $K$-subalgebra which is isomorphic to the polynomial $K$-algebra $K[t_1,\dots,t_n]$.
Let $$
e_m(y_1,\dots,y_n):=\sum_{1\leq i_1<i_2<\dots<i_m\leq n}y_{i_1}\dots y_{i_m}\in K[y_1,\dots,y_n]^{\Sym_n}$$ be the $m$-th elementary symmetric polynomial. It is well-known that for each $1\leq k\leq n$, \begin{equation}\label{y1n}
y_k^n=\sum_{i=0}^{n-1}(-1)^{n+i-1}y_k^{i}e_{n-i}(y_1,\dots,y_n) .
\end{equation}
Let $\mathfrak{m}_n$ be the maximal ideal of $K[y_1,\dots,y_n]$ generated by $y_1,\dots,y_n$. Let
$\mathfrak{n}_n:=(\mathfrak{m}_n)^{\Sym_n}$. Applying (\ref{y1n}), we get that
\begin{equation}\label{kp1}\begin{matrix}
\text{for any $k\in\N$, there exists some $N(k)\in\N$, such that $y_1^{N(k)}$ lives inside the}\\
\text{two-sided ideal of $K[y_1,\dots,y_n]$ generated by $(\mathfrak{n}_n)^k$.}\end{matrix}
\end{equation}

\begin{lem} \label{keylem1} For each $\Lambda\in P^{+}$, let $I(\Lambda)$ be the two-sided ideal of $\mathscr{R}_\beta$ generated by $\{y_1^{\<\Lambda,\alpha_{i_1}^{\vee}\>}e(\bi)|\bi\in I^\beta\}$. Then $$
\bigcap_{\Lambda}I(\Lambda)=\{0\},
$$
where the subscript runs through all $\Lambda\in P^{+}$. \end{lem}

\begin{proof} Suppose that $\bigcap_{\Lambda}I(\Lambda)\neq 0$. Let $0\neq z\in\bigcap_{\Lambda}I(\Lambda)$. Then there exists an integer $k\in\Z^{>0}$, such that for any $\bj\in I^\beta$, we can write \begin{equation}\label{NNN}
ze(\bj)=\sum_{i=1}^{s}\psi_{w_i}f_ie(\bj),
\end{equation}
where $w_1,\dots,w_s\in\Sym_n$ are pairwise distinct, and $f_i\in K[y_1,\dots,y_n]$ such that $\deg(f_i)<k$ for any $1\leq i\leq s$.

Now we pick an integer $N:=N(k)$ as in (\ref{kp1}). We take a special $\Lambda\in P^{+}$ such that $\<\Lambda,\alpha_{j_1}^\vee\>=N$ for any $\bj\in I^{\beta}$. By assumption, $z\in I(\Lambda)$, which implies that $ze(\beta)$ lives inside the two-sided ideal of $\mathscr{R}_\beta$ generated by $y_1^{N}e(\beta)$. Hence by (\ref{kp1}) $ze(\beta)$ lives inside the two-sided ideal of $\mathscr{R}_\beta$ generated by $(\mathfrak{n}_n)^ke(\beta)$. However, this is a contradiction to
(\ref{NNN}) by Lemma \ref{basis4} and the fact that $\mathfrak{n}_ne(\beta)$ is central in ${\mathscr{R}}_\beta$.
\end{proof}

The following corollary says that all these modified affine Hecke algebras, quiver Hecke algebras as well as their generalized Ore localizations are embedded in the the inverse limits of cyclotomic Hecke algebras or cyclotomic quiver Hecke algebras.

\begin{cor} \label{keycor1} We have the following natural injections: $$
\mathscr{R}_{\beta}\hookrightarrow\widetilde{\mathscr{R}}_{\beta}\hookrightarrow\underset{\Lambda}{\underleftarrow{\lim}}\,\R[\beta],\,\,
\widehat{\H}_{\beta}(q)\hookrightarrow\widetilde{\H}_{\beta}(q)\hookrightarrow\underset{\Lambda}{\underleftarrow{\lim}}\,\H_{\beta}^{\Lambda}(q),\,\,
\widehat{H}_{\beta}\hookrightarrow\widetilde{H}_{\beta}\hookrightarrow\underset{\Lambda}{\underleftarrow{\lim}}\,H_{\beta}^{\Lambda}.$$
\end{cor}

\begin{proof} The first injection follows from Lemma \ref{keylem1}, while the other two injections follows directly from their definitions.
\end{proof}

\medskip
\noindent{\bf Proof of Theorem \ref{mainthm0a} and \ref{mainthm0b}:} By our choices of $Q_r(\bi)$ in (\ref{qri1}) and (\ref{qri2}), it is easy to see that Brundan--Kleshchev's isomorphisms $\theta^{\Lambda}$ induces an isomorphism: $$
\theta_0: \underset{\Lambda}{\underleftarrow{\lim}}\,\R[\beta]\cong \underset{\Lambda}{\underleftarrow{\lim}}\,\H_{\beta}^{\Lambda}(q),$$
in the non-degenerate case. We have the following diagrams: $$\begin{tikzpicture}
     \matrix[matrix of math nodes,row sep=1cm,column sep=16mm]{
       |(O)| \widetilde{\mathscr{R}}_\beta & |(S)| \widetilde{\H}_\beta(q)\\
       |(P)| \underset{\Lambda}{\underleftarrow{\lim}}\,\R[\beta]   &|(H)|\underset{\Lambda}{\underleftarrow{\lim}}\,\H_{\beta}^{\Lambda}(q) \\
     };
      \draw[>->](S) -- node[right]{} (H);
     \draw[>->](O) -- node[left]{} (P);
     \draw[->](P) -- node[above]{$\sim$} node[below]{$\theta_{0}$} (H);
     \draw[->](O) -- node[above]{$\theta$} node[below]{} (S);
    \end{tikzpicture},
\begin{tikzpicture}
     \matrix[matrix of math nodes,row sep=1cm,column sep=16mm]{
       |(O)| \widetilde{\H}_\beta(q) & |(S)| \widetilde{\mathscr{R}}_\beta\\
       |(P)| \underset{\Lambda}{\underleftarrow{\lim}}\,\H_{\beta}^{\Lambda}(q)   &|(H)|\underset{\Lambda}{\underleftarrow{\lim}}\,\R[\beta] \\
     };
      \draw[>->](S) -- node[right]{} (H);
     \draw[>->](O) -- node[left]{} (P);
     \draw[->](P) -- node[above]{$\sim$} node[below]{$\theta_{0}^{-1}$} (H);
     \draw[->](O) -- node[above]{$\eta$} node[below]{} (S);
    \end{tikzpicture}
$$
where the vertical maps are the injections given in Corollary \ref{keycor1}, and for the moment both $\theta$ and $\eta$ are only defined on a set of $K$-algebra generators. Note that the bottom maps are both $K$-algebra isomorphisms. In order to show that
$\theta$ and $\eta$ can be extended to a pair of well-defined $K$-algebra homomorphisms, it is enough to check that the above diagrams commutes on a set of $K$-algebra generators of $\widetilde{\mathscr{R}}_\beta$ and $\widetilde{\H}_\beta(q)$ respectively.

To show the first diagram commutes on a set of $K$-algebra generators of $\widetilde{\mathscr{R}}_\beta$, it is suffices to show that $$\begin{aligned}
\pi_1(\Lambda)\Bigl(\theta\bigl(\psi_re(\bi)\bigr)\Bigr)&=\theta^{\Lambda}\Bigl(p_1(\Lambda)\bigl(\psi_re(\bi)\bigr)\Bigr),\\
\pi_1(\Lambda)\Bigl(\theta\bigl(y_se(\bi)\bigr)\Bigr)&=\theta^{\Lambda}\Bigl(p_1(\Lambda)\bigl(y_se(\bi)\bigr)\Bigr),\\
\pi_1(\Lambda)\Bigl(\theta\bigl(e(\bi)\bigr)\Bigr)&=\theta^{\Lambda}\Bigl(p_1(\Lambda)\bigl(e(\bi)\bigr)\Bigr),\\
\end{aligned}
$$
where $\bi\in I^\beta, 1\leq r<n, 1\leq s\leq n$. The last two equalities are obviously true. It remains to verify the first equality. There are three cases:

\smallskip
{\noindent \it Case 1.} $i_r=i_{r+1}$. In this case, $$\begin{aligned}
&\quad\,\pi_1(\Lambda)\Bigl(\theta\bigl(\psi_re(\bi)\bigr)\Bigr)=q^{i_r}(T_r+1)(L_r-qL_{r+1})^{-1}e(\bi)\\
&=(T_r+1)(q^{-i_r}L_r-q^{1-i_r}L_{r+1})^{-1}e(\bi)\\
&=(T_r+1)(1-y_r-q+qy_{r+1})^{-1}e(\bi)\\
&=(T_r+P_r(\bi))Q_r(\bi)^{-1}e(\bi)=\theta^{\Lambda}\Bigl(p_1(\Lambda)\bigl(\psi_re(\bi)\bigr)\Bigr),
\end{aligned}
$$
as required.

\smallskip
{\noindent \it Case 2.} $i_r=i_{r+1}+1$. In this case, $$\begin{aligned}
&\quad\,\pi_1(\Lambda)\Bigl(\theta\bigl(\psi_re(\bi)\bigr)\Bigr)\\
&=q^{-i_r}\Bigl(T_r(L_r-L_{r+1})+(q-1)L_{r+1}\Bigr)e(\bi)\\
&=\Bigl(T_r(1-q^{-1}-y_r+q^{-1}y_{r+1})+(1-q^{-1})(1-y_{r+1})\Bigr)e(\bi) .
\end{aligned}
$$
By definition, in the non-degenerate case, $$\begin{aligned}
P_r(\bi)&=1+\frac{y_r-y_{r+1}}{1-q^{-1}}\frac{1}{1-\frac{y_{r+1}-qy_r}{1-q}}=1+\frac{q(y_{r+1}-y_r)}{1-q-y_{r+1}+qy_r},\\
Q_r(\bi)&=\frac{1}{1-q^{-1}}\frac{1}{1-\frac{y_{r+1}-qy_r}{1-q}}=\frac{-q}{1-q+qy_r-y_{r+1}} .
\end{aligned}$$
Therefore, $$\begin{aligned}
&\quad\,\theta^{\Lambda}\Bigl(p_1(\Lambda)\bigl(\psi_re(\bi)\bigr)\Bigr)=\bigl(T_r+P_r(\bi)\bigr)Q_r(\bi)^{-1}e(\bi)\\
&=\Bigl(T_r(1-q^{-1}-y_r+q^{-1}y_{r+1})+(1-q^{-1})(1-y_{r+1})\Bigr)e(\bi)\\
&=\pi_1(\Lambda)\Bigl(\theta\bigl(\psi_re(\bi)\bigr)\Bigr).
\end{aligned}$$

\smallskip
{\noindent \it Case 3.} $i_r\notin\{i_{r+1},i_{r+1}+1\}$. In the non-degenerate case, we have that $$\begin{aligned}
P_r(\bi)&=\frac{1-q}{1-q^{i_r-i_{r+1}}}\Bigl\{1+\frac{y_r-y_{r+1}}{1-q^{i_{r+1}-i_r}}\frac{1}{1-\frac{q^{i_{r+1}}y_{r+1}-q^{i_r}y_r}{q^{i_{r+1}}-q^{i_r}}}\Bigr\}\\
&=\frac{1-q}{1-q^{i_r-i_{r+1}}}\Bigl\{1+\frac{q^{i_r}(y_{r+1}-y_r)}{q^{i_{r+1}}-q^{i_r}-q^{i_{r+1}}y_{r+1}+q^{i_r}y_r}\Bigr\},\\
&=\frac{(1-q)q^{i_{r+1}}(1-y_{r+1})}{q^{i_{r+1}}-q^{i_r}-q^{i_{r+1}}y_{r+1}+q^{i_r}y_r},\\
Q_r(\bi)^{-1}&=(P_r(\bi)-1)^{-1}\\
&=\Biggl(\frac{(1-q)q^{i_{r+1}}(1-y_{r+1})}{q^{i_{r+1}}-q^{i_r}-q^{i_{r+1}}y_{r+1}+q^{i_r}y_r}-1\Biggr)^{-1}\\
&=\frac{q^{i_{r+1}}-q^{i_r}-q^{i_{r+1}}y_{r+1}+q^{i_r}y_r}{q^{i_r}-q^{i_r}y_r-q^{i_{r+1}+1}+q^{i_{r+1}+1}y_{r+1}} .
\end{aligned}$$
By definition, $$\begin{aligned}
&\quad\,\pi_1(\Lambda)\Bigl(\theta\bigl(\psi_re(\bi)\bigr)\Bigr)=\Bigl(T_r(L_{r+1}-L_r)+(1-q)L_{r+1}\Bigr)(L_r-qL_{r+1})^{-1}e(\bi)\\
&=(T_r+P_r(\bi))Q_r(\bi)^{-1}e(\bi)\\
&=\theta^{\Lambda}\Bigl(p_1(\Lambda)\bigl(\psi_re(\bi)\bigr)\Bigr) .
\end{aligned}
$$
This proves the claim for the first diagram. In a similar way, we can prove that the second diagram commutes on a set of $K$-algebra generators of $\widetilde{\H}_\beta(q)$. Therefore, $\theta$ and $\eta$ can be extended to a pair of well-defined $K$-algebra homomorphisms.

Finally, they are mutually inverse maps because it is easy to check that $\theta\eta$ and $\eta\theta$ are both equal to the identity map on a set of generators. This completes the proof of Theorem \ref{mainthm0a}, while Theorem \ref{mainthm0b} can be proved in a similar way.
\bigskip

\section{Some applications}

The purpose of this section is to give some applications of Theorem \ref{mainthm0a} and \ref{mainthm0b}.\medskip


Let $\HH_n\in\{\HH_n(q), H_n\}$. Let $\Lam\in P^+$ and $\HH_n^\Lam\in\{\HH_n^\Lam(q), H_n^\Lam\}$. Recall that $\pi^\Lam: \HH_n\twoheadrightarrow\HH_n^\Lam$ is the canonical surjective homomorphism. Note that the Jucys-Murphy operators $L_1,\dots,L_n\in\HH_n^\Lam$ are in general algebraically dependent.

\begin{dfn} An element $z\in\HH_n^\Lam$ is said to be a symmetric polynomial in $L_1,\dots,L_n$ if $z=f(L_1,\dots,L_n)$ for some symmetric polynomial
$f(t_1,\dots,t_n)\in K[t_1,\dots,t_n]$.
\end{dfn}

Since $\pi^\Lam$ maps any central element of $\HH_n$ to a central element of $\HH_n^\Lam$, it is clear that any symmetric polynomial in $L_1,\dots,L_n$ lives in the center of $\HH_n^\lam$. The following center conjecture is well-known.

\begin{conj} \label{conj1} (cf. \cite[3.1]{Mac}) Let $n\in\N$ and $\Lam\in P^+$. Then $\pi^\Lam$ maps the center $Z(\HH_n)$ of $\HH_n$ surjectively onto the center $Z(\HH_n^\Lam)$ of $\H_n^\Lam$. Moreover, the center $Z(\HH_n^\Lam)$ of $\H_n^\Lam$ is the set of symmetric polynomials in $L_1,\dots,L_n$.
\end{conj}

Some special cases of Conjecture \ref{conj1} are known to be true. \begin{enumerate}
\item If $q=1$, then Conjecture \ref{conj1} was proved by Murphy \cite{Mur} in the case when $\ell=1$ and by Brundan \cite{Brundan:degenCentre} for general $\ell$;
\item If $q\neq 1$ and $\ell=1$, then  Conjecture \ref{conj1} was proved by Dipper and James \cite{DJ} in the semisimple case and by Francis and Graham \cite{FG} in general case.
\item If $q\neq 1$, $\ell>1$ and $e=0$, then Conjecture \ref{conj1} was proved by McGerty \cite{Mc}.
\end{enumerate}

Recall that $p^\Lambda: \mathscr{R}_{\beta}\twoheadrightarrow\R[\beta]$ be the canonical surjective algebra homomorphism. Inside the quiver Hecke algebra $\mathscr{R}_\beta$, $e(\bi)\neq 0$ for each $\bi\in I^\beta$; and $e(\bi)=e(\bj)$ if and only if $\bi=\bj$. As a result, we have a natural left action of $\Sym_n$ on the set $\{e(\bi)\in\mathscr{R}_\beta|\bi\in I^\beta\}$ via $w\cdot e(\bi):=e(w\bi)$. On the other hand, we also have a natural left action of $\Sym_n$ on the subalgebra $K[y_1,\dots,y_n]$ generated by $y_1,\dots,y_n$ defined by $$
w(\sum_{i}\lam_iy_1^{c_{1i}}\dots y_n^{c_{ni}}):=\sum_{i}\lam_iy_{w(1)}^{c_{1i}}\dots y_{w(n)}^{c_{ni}}, \,\,\text{where}\,\,w\in\Sym_n, \lam_i\in K, c_{1i},\dots,c_{ni}\in\N .
$$
Therefore, we get a natural action of $\Sym_n$ on the subalgebra of $\mathscr{R}_\beta$ generated by $\{y_1,\dots,y_n,e(\bi)|\bi\in I^\beta\}$. We set \begin{equation}\label{invariant1}
K[y_1,\dots,y_n,e(\bi)|\bi\in I^\beta]^{\Sym_n}:=\bigl\{f\in K[y_1,\dots,y_n,e(\bi)|\bi\in I^\beta]\bigm|\sigma(f)=f,\,\,\forall\,\sigma\in\Sym_n\bigr\}.
\end{equation}
However, in contrast to the quiver Hecke algebra $\mathscr{R}_\beta$ case, inside the cyclotomic quiver Hecke algebra $\R[\beta]$, it is possible that $e(\bi)=0$ for some $\bi\in I^\beta$. Therefore, it is not clear at all whether one can define a natural action of $\Sym_n$ on the set $\{e(\bi)\in\R[\beta]|\bi\in I^\beta\}$ via $w\cdot e(\bi):=e(w\bi)$ or not, because it could happen that $e(\bi)=e(\bj)=0$ while $e(w\bi)\neq e(w\bj)$.

\begin{dfn} \label{SymmetricElement} An element $f\in K[y_1,\dots,y_n,e(\bi)|\bi\in I^\beta]\subseteq\mathscr{R}_\beta$ is said to be symmetric if $f\in K[y_1,\dots,y_n,e(\bi)|\bi\in I^\beta]^{\Sym_n}$. An element $z\in K[y_1,\dots,y_n,e(\bi)|\bi\in I^\beta]\subseteq\R[\beta]$ is said to be symmetric if $z=p^\Lambda(f)$ for some symmetric element $f$ in $\mathscr{R}_\beta$.
\end{dfn}

By \cite[Theorem 2.9]{KhovLaud:diagI}, the center $Z(\mathscr{R}_\beta)$ of $\mathscr{R}_\beta$ is the set of symmetric elements  in $K[y_k,e(\bi)|1\leq k\leq n, \bi\in I^\beta]\subseteq\R[\beta]$. The following is a similar center conjecture for the cyclotomic quiver Hecke algebras $\R[\beta]$ of type $A$, which has been a folklore for some time. In fact, it is even expected to be true for the cyclotomic quiver Hecke algebras associated to any simply laced quiver.

\begin{conj} \label{RbetaLambdaConj1} Then $p^{\Lambda}$ maps the center $Z(\mathscr{R}_\beta)$ of $\mathscr{R}_\beta$ surjectively onto the center $Z(\R[\beta])$ of $\R[\beta]$. In other words, the center $Z(\R[\beta])$ of $\R[\beta]$ is the set of symmetric elements in $K[y_1,\dots,y_n,e(\bi)|\bi\in I^\beta]\subseteq\R[\beta]$.
\end{conj}

The following theorem is the first application of the main result of this paper.

\begin{thm} \label{mainapp1} Let $K$ be any field and $\Lam\in P^+$. Then Conjecture (\ref{conj1}) holds if and only if Conjecture \ref{RbetaLambdaConj1} holds for any $\beta\in Q_n^{+}$.
\end{thm}

Now we are going to use the isomorphism $\theta, \theta'$ obtained in Theorems \ref{mainthm0a}, \ref{mainthm0b} to give a proof of Theorem \ref{mainapp1}. First, we need the following lemma, which express each idempotent $e(\bi)\in\HH_n^\Lam$ as a polynomial (depending only on $\bi=(i_1,\dots,i_n)$) in the Jucys--Murphy operators $L_1,\cdots,L_n$. Recall from (\ref{LrXrN}) and  (\ref{EnLam}) that $$
E_{N}(\bi):=\prod_{r=1}^{n}\prod_{i_r\neq j\in I(\beta)}X_{i_r,j,N}, $$
where $\beta=\sum_{i\in I}k_i\alpha_i$, $I(\beta)=\{j\in I|k_i\neq 0\}$, and
$$X_{i_r,j,N}:=\begin{cases}1-\Bigl(\frac{q^{i_r}-X_r}{q^{i_r}-q^j}\Bigr)^N, &\text{if $q\neq 1$;}\\
1-\Bigl(\frac{i_r-x_r}{i_r-j}\Bigr)^N, &\text{if $q=1$.}\end{cases} $$
By Lemma \ref{keyGeLilemma}, $\pi^\Lambda(E_{N}(\bi))=e(\bi)$.

\begin{lem} \label{eipolynomial} For each $\bi\in I^\beta$, we can associate with a polynomial $f_{\bi}(t_1,\dots,t_n)\in K[t_1,\dots,t_n]$ which depends only on $\bi$, such that \begin{enumerate}
\item $e(\bi)=f_{\bi}(L_1,\dots,L_n)$ holds in $\HH_n^\Lam$; and
\item $f_{s_r\bi}(t_1,\dots,t_n)=s_r(f_{\bi}(t_1,\dots,t_n))$ for any $1\leq r<n$.
\end{enumerate}
In particular, the idempotent $e(\beta)$ in $\HH_n^\Lam$ is a symmetric polynomial in $L_1,\dots,L_n$.
\end{lem}

\begin{proof} We fix an integer $N>\ell^n n!$. Recall that the elements $X_1,\dots, X_n$ (respectively, $x_1,\dots,x_n$) are algebraically independent in $\HH_n(q)$ (respectively, in $H_n$). By Lemma \ref{keyGeLilemma}, we see that if we define $$
f_\bi(t_1,\dots,t_n):=E_{N}(\bi)\downarrow_{X_1:=t_1,\dots,X_n:=t_n}, $$
when $q\neq 1$; or $$
f_\bi(t_1,\dots,t_n):=E_{N}(\bi)\downarrow_{x_1:=t_1,\dots,x_n:=t_n},
$$
when $q=1$; then by the definition of $E_{N}(\bi)$ we can deduce that for any $1\leq r<n$, $f_{s_r\bi}(t_1,\dots,t_n)=s_r(f_{\bi}(t_1,\dots,t_n))$.

As a consequence, since $e(\beta)=\sum_{\bj\in I^\beta}e(\bj)$, it is clear that $e(\beta)\in\HH_n^\Lam$ is a symmetric polynomial in $L_1,\dots,L_n$. This proves the lemma.
\end{proof}

We have the following surjective algebra homomorphisms: $$\begin{aligned}
& \pi_1(\Lambda): \widetilde{\H}_{\beta}(q)\twoheadrightarrow\H_{\beta}^{\Lambda}(q),\quad\,\,\, \pi_2(\Lambda): \widetilde{H}_{\beta}\twoheadrightarrow H_{\beta}^{\Lambda},\\
& p_1(\Lambda): \widetilde{\mathscr{R}}_{\beta}\twoheadrightarrow\R[\beta],\quad\,\,\, p_2(\Lambda): \widetilde{\mathscr{R}}'_{\beta}\twoheadrightarrow \R[\beta] .
\end{aligned}
$$
\bigskip\medskip

\noindent
{\textbf{Proof of the Theorem \ref{mainapp1}}}: We only prove Theorem \ref{mainapp1} in the non-degenerate case (i.e., $q\neq 1$), as the degenerate case (i.e., $q=1$) is the same.

Suppose that Conjecture \ref{RbetaLambdaConj1} holds for any $\beta\in Q_n^{+}$. We define $$
Q(\Lam):=\{\beta\in Q_n^+|\text{$e(\beta)\neq 0$ in $\HH_n^\Lam$}\}, $$ which is a finite subset of $Q_n^+$. We have the following commutative diagram: $$\begin{tikzpicture}
     \matrix[matrix of math nodes,row sep=1cm,column sep=16mm]{
       |(O)| \oplus_{\beta\in Q(\Lam)}\widetilde{\mathscr{R}}_\beta & |(S)| \oplus_{\beta\in Q(\Lam)}\widetilde{\HH}_\beta(q)\\
       |(P)| \mathscr{R}_n^{\Lambda}=\oplus_{\beta\in Q(\Lam)}\R[\beta]   &|(H)|{\HH}_n^{\Lambda}(q)=\oplus_{\beta\in Q(\Lam)}\HH_\beta^\Lam(q) \\
     };
      \draw[->>](S) -- node[right]{$\pi_1(\Lam)$} (H);
     \draw[->>](O) -- node[left]{$p_1(\Lam)$} (P);
     \draw[->](P) -- node[above]{$\sim$} node[below]{$\theta^{\Lambda}$} (H);
     \draw[->](O) -- node[above]{$\theta$} node[below]{$\sim$} (S);
    \end{tikzpicture},
$$
where the two vertical maps are both surjective.

By assumption, the map $p^\Lambda$ sends the center $Z(\mathscr{R}_{\beta})$ of $\mathscr{R}_\beta$ surjectively onto the center $Z(\R[\beta])$ of $\mathscr{R}_\beta^\Lambda$ for each $\beta$. Since $\mathscr{R}_\beta\subseteq\widetilde{\mathscr{R}}_\beta$ and $Z(\mathscr{R}_\beta)\subseteq Z(\widetilde{\mathscr{R}}_\beta)$, we can deduce that $\theta\Bigl(Z(\mathscr{R}_\beta)\Bigr)$ was mapped by the right vertical map in the above commutative diagram surjectively onto the center $Z({\HH}_{\beta}^\Lam(q))$ of ${\HH}_{\beta}^\Lam(q)$. On the other hand, by the definition of $\theta$ in Theorem \ref{mainthm0a} and the explicit description of the center $Z(\mathscr{R}_\beta)$ in \cite[Theorem 2.9]{KhovLaud:diagI}, we know that $$\begin{aligned}
\theta\Bigl(Z(\mathscr{R}_\beta)\Bigr)&\subseteq Z(\widetilde{\HH}_\beta(q))\cap K[\hX_k\he(\bi),\he(\bi)|1\leq k\leq n, \bi\in I^\beta]\\
&\subseteq Z(\widehat{\HH}^+_\beta(q))\cap K[\hX_k\he(\bi),\he(\bi)|1\leq k\leq n, \bi\in I^\beta].
\end{aligned}
$$
It follows that the subspace $Z(\widehat{\HH}^+_\beta(q))\cap K[\hX_k\he(\bi),\he(\bi)|1\leq k\leq n, \bi\in I^\beta]$ of $\widehat{\HH}^+_\beta(q)$ was mapped surjectively onto the center $Z(\HH_\beta^\Lam)(q)$ of $\HH_\beta^\Lam(q)$. On the other hand, by Corollary \ref{center00a}, we know that
$Z(\widehat{\HH}^+_\beta(q))\cap K[\hX_k\he(\bi),\he(\bi)|1\leq k\leq n, \bi\in I^\beta]$ is equal to the set of symmetric elements in $\{\hX_k\he(\bi),\he(\bi)|1\leq k\leq n, \bi\in I^\beta\}$. It follows that the center $Z(\HH_\beta^\Lam(q))$ of $\HH_\beta^\Lam(q)$ is the set of symmetric elements in $L_1e(\bi),\dots,L_ne(\bi),e(\bi),\bi\in I^\beta$. Now combining this with Lemma \ref{eipolynomial} and noting that
$Z(\HH_n^\Lam(q))=\oplus_{\beta\in Q(\Lam)}Z({\HH}_{\beta}^\Lam(q))$, we can deduce that $Z(\HH_n^\Lam(q))$ is the set of symmetric polynomials in $L_1,\dots,L_n$, as required.

Conversely, suppose that Conjecture (\ref{conj1}) holds. That says,  $Z(\HH_n^\Lam(q))$ is the set of symmetric polynomials in $L_1,\dots,L_n$. Let $\beta\in Q_n^+$. Then $Z(\HH_\beta^\Lam(q))$ is the set of symmetric polynomials in $L_1e(\beta),\dots,L_ne(\beta)$. It follows that
$Z(\HH_\beta^\Lam(q))\cap K[\hX_k\he(\bi),\he(\bi)|1\leq k\leq n, \bi\in I^\beta]$ was mapped surjectively onto the center $Z(\HH_\beta^\Lam(q))$ of $\HH_\beta^\Lam(q)$. By the above commutative diagram, we can deduce that the preimage $$
\theta^{-1}\Bigl(Z(\HH_\beta^\Lam(q))\cap K[\hX_k\he(\bi),\he(\bi)|1\leq k\leq n, \bi\in I^\beta]\Bigr)
$$
was mapped surjectively by $p_1(\Lam)$ onto the center $Z(\R[\beta])$. On the other hand, by the definition of $\theta$ in Theorem \ref{mainthm0a} we know that
$$
\theta^{-1}\Bigl(Z(\HH_\beta^\Lam(q))\cap K[\hX_k\he(\bi),\he(\bi)|1\leq k\leq n, \bi\in I^\beta]\Bigr)\subseteq Z(\widetilde{\mathscr{R}}_\beta)\cap{\mathscr{R}}_\beta=Z({\mathscr{R}}_\beta) .
$$
It follows that $p_1(\Lam)$ must map the center $Z({\mathscr{R}}_\beta)$ of ${\mathscr{R}}_\beta$ surjectively onto the center $Z(\R[\beta])$ of $\R[\beta]$ as required. This completes the proof of the theorem.
\bigskip

In the remaining part of this section, we are going to give two more applications of Theorem \ref{mainthm0a} and \ref{mainthm0b}. For simplicity, we assume henceforth that $K$ is an algebraically closed field. For any $K$-algebra $A$, we use $A\mod$ to denote the category of finite dimensional left $A$-modules. Recall that $I=\Z/e\Z$. Let $\H_{n}(q)\mod_I$ be the full subcategory of ${\H}_{n}(q)\mod$ such that all the eigenvalues of $X_1$ are in $q^{I}$, $\widehat{\H}_{\beta}(q)\mod_I$ (resp., $\widetilde{\H}_{\beta}(q)\mod_I$) be the full subcategory of $\widehat{\H}_{\beta}(q)\mod$ (resp., $\widetilde{\H}_{\beta}(q)\mod$) such that all the eigenvalues of $\hX_1\he(\beta)$ are in $q^{I}$. Let $H_{n}\mod_I$ be the full subcategory of ${H}_{n}\mod$ such that all the eigenvalues of $x_1$ are in $I$, and we define $ \widetilde{H}_{\beta}\mod_I$ and  $\widehat{H}_{\beta}\mod_I$ in a similar way. Then we have the following natural inclusions: $$\begin{aligned}
&\H_{\beta}^{\Lambda}(q)\mod\subseteq \widehat{\H}_{\beta}(q)\mod_I=\widetilde{\H}_{\beta}(q)\mod_I,\quad H_{\beta}^{\Lambda}\mod\subseteq\widehat{H}_{\beta}\mod_I=\widetilde{H}_{\beta}\mod_I,\\
&\R[\beta]\mod\subseteq \widetilde{\mathscr{R}}_{\beta}\mod,\quad\,\,\R[\beta]\mod\subseteq \widetilde{\mathscr{R}}'_{\beta}\mod .
\end{aligned}
$$

For any $\bi=(i_1,\dots,i_n), \bj=(j_1,\dots,j_n)\in I^n$, we define $\bi\sim\bj$ whenever there exists some $w\in\Sym_n$ such that $w\bi=\bj$. The central characters of $\H_n(q)$ are in bijection with the elements in the set $I^n/\!\!\sim$ of $\Sym_n$-orbits, and hence are in bijection with the elements $\beta\in Q_n^+$. For any $\beta\in Q_n^+$, let $(\H_n(q)\mod)[\beta]$ be the subcategory of $\H_n\mod$ which is determined by the central character $\chi_\beta$ of $\H_n(q)$ corresponding to $\beta$ (cf. \cite{K}). In a similar way, we have the subcategory $(H_n\mod)[\beta]$ of $H_n\mod$.

\begin{lem} We have that $$\begin{aligned}
&\H_n(q)\mod_I=\oplus_{\beta\in Q_n^+}({\H}_{n}(q)\mod)[\beta],\quad H_n\mod_I=\oplus_{\beta\in Q_n^+}({H}_{n}\mod)[\beta],\\
&\widetilde{\H}_{\beta}(q)\mod_I=\widehat{\H}_{\beta}(q)\mod_I=({\H}_{n}(q)\mod)[\beta],\quad \widetilde{\mathscr{R}}_{\beta}\mod=\mathscr{R}_{\beta}\mod,\\
&\widetilde{H}_{\beta}\mod_I=\widehat{H}_{\beta}\mod_I=({H}_{n}\mod)[\beta],\quad\widetilde{\mathscr{R}}'_{\beta}\mod=\mathscr{R}_{\beta}\mod .
\end{aligned}
$$
\end{lem}

\begin{lem} \label{equalCate} We have that $$\begin{aligned}
&\widetilde{\H}_{\beta}(q)\mod_I=\underset{\Lambda}{\underrightarrow{\lim}}\,\Bigl(\H_{\beta}^{\Lambda}(q)\mod\Bigr),\quad
\widetilde{H}_{\beta}\mod_I=\underset{\Lambda}{\underrightarrow{\lim}}\,\Bigl(H_{\beta}^{\Lambda}\mod\Bigr),\\
&{\mathscr{R}}_{\beta}\mod=\underset{\Lambda}{\underrightarrow{\lim}}\,\Bigl(\R[\beta]\mod\Bigr) .
\end{aligned}$$
\end{lem}

\begin{proof} We only prove the first and the third equalities as the second one can be proved in a similar way. For any  $V\in\widetilde{\H}_{\beta}(q)\mod_I$, we can find $\ell\in\N, \kappa_1,\dots,\kappa_{\ell}\in\Z/e\Z$, such that $$
(\hX_1\he(\beta)-q^{\kappa_1})\dots(\hX_1\he(\beta)-q^{\kappa_{\ell}})(v)=0,\quad\forall\,v\in V,
$$
because $K$ is algebraically closed. Set $\Lambda:=\sum_{i=1}^{\ell}\Lambda_{\kappa_i}$. Note that there is surjective homomorphism $\pi_1(\Lam):\widetilde{\H}_{\beta}(q)\twoheadrightarrow\HH_\beta^\Lam(q)$ such that $$
\pi_1(\Lam)(\he(\bi))=e(\bi),\,\,\pi_1(\Lam)(\hX_k\he(\bi))=L_ke(\bi),\,\,\pi_1(\Lam)(\hT_r\he(\bi))=T_re(\bi),
$$
for any $1\leq k\leq n, 1\leq r<n, \bi\in I^\beta$. It follows that $V\in\H_{\beta}^{\Lambda}(q)\mod$ as required.

For any finite dimensional module $V$ over ${\mathscr{R}}_{\beta}$, since $\deg y_1=2>0$ and $\dim V<\infty$, we can find $N\in\N$, such that $$
y_1^{N}e(\beta)(v)=0,\quad\forall\,v\in V,
$$
because $y_1$ is a homogeneous element of degree $2$. We can take a special $\Lambda\in P^{+}$ such that $\<\Lambda,\alpha_{i_1}^{\vee}\>=N$ for any $\bi\in I^{\beta}$. Then $V\in\R[\beta]\mod$ as required.
\end{proof}

Therefore, we can use Lemma \ref{equalCate} to identify these categories. Let $m,n\in\N$. If we shift the subscripts of each generator of $\H_n(q)$ upward by $m$ positions, then we get an algebra $\H_n^{(m)}(q)$ which is isomorphic to $\H_n(q)$ and with standard generators $T_{m+1},\dots,T_{m+n-1},X_{m+1}^{\pm 1},\dots,X_{m+n}^{\pm 1}$. For each $g\in\H_n(q)$, let ${g}^{(m)}$ be its canonical image in $\H_n^{(m)}(q)$. For any $\alpha\in Q_m^{+}, \beta\in Q_n^{+}$ and $\bi=(i_1,\dots,i_m)\in I^{\alpha}, \bj=(j_1,\dots,j_n)\in I^{\beta}$, we define the concatenation $\bi\vee\bj:=(i_1,\dots,i_m,j_1,\dots,j_n)\in I^{\alpha+\beta}$. Then the map $$
fe(\bi)\otimes ge(\bj)\mapsto fg^{(m)}e(\bi\vee\bj),\quad\forall\,f\in\H_m(q), g\in\H_n(q)
$$ can be naturally extended to a well-defined injective non-unital algebra homomorphism $\widehat{\H}_{\alpha}(q)\boxtimes \widehat{\H}_{\beta}(q)\hookrightarrow
\widehat{\H}_{\alpha+\beta}(q)$. By definition, this injection also induces a natural injection $$
\iota_{\alpha,\beta}: \widetilde{\H}_{\alpha}(q)\boxtimes\widetilde{\H}_{\beta}(q)\hookrightarrow \widetilde{\H}_{\alpha+\beta}(q) .
$$
In a similar way, the well-known non-unital injection ${\mathscr{R}}_{\alpha}\boxtimes {\mathscr{R}}_{\beta}\hookrightarrow
{\mathscr{R}}_{\alpha+\beta}$ naturally induces an injection $$
\widetilde{\mathscr{R}}_{\alpha}\boxtimes \widetilde{\mathscr{R}}_{\beta}\hookrightarrow
\widetilde{\mathscr{R}}_{\alpha+\beta} .
$$
which will still be denoted by $\iota_{\alpha,\beta}$. We have the following commutative diagram of morphisms:
$$\begin{tikzpicture}
     \matrix[matrix of math nodes,row sep=1cm,column sep=16mm]{
       |(O)| \widetilde{\H}_{\alpha}(q)\boxtimes \widetilde{\H}_{\beta}(q) & |(S)| \widetilde{\H}_{\alpha+\beta}(q)\\
       |(P)| \widetilde{\mathscr{R}}_{\alpha}\boxtimes \widetilde{\mathscr{R}}_{\beta}   &|(H)|\widetilde{\mathscr{R}}_{\alpha+\beta} \\
     };
      \draw[->](S) -- node[right]{$\wr$} (H);
     \draw[->](O) -- node[left]{$\wr$} (P);
     \draw[->](P) -- node[above]{} node[below]{$\iota_{\alpha,\beta}$} (H);
     \draw[->](O) -- node[above]{$\iota_{\alpha,\beta}$} node[below]{} (S);
    \end{tikzpicture},
$$
where vertical maps are isomorphisms induced from $\theta$.

For any $V\in\widetilde{\H}_{\alpha}(q)\mod$, let $V^{\theta}\in\widetilde{\mathscr{R}}_{\alpha}\mod$ such that $V^{\theta}=V$ as a $K$-linear space and $\widetilde{\mathscr{R}}_{\alpha}$ acts on $V^\theta$ through the isomorphism $\theta$. For any $V\in\widetilde{\H}_{\alpha}(q)\mod, W\in\widetilde{\H}_{\beta}(q)\mod$, we have the following convolution products: $$\begin{aligned}
V\circ W&:=\ind_{\alpha,\beta}^{\alpha+\beta}V\boxtimes W=\widetilde{\H}_{\alpha+\beta}(q)\otimes_{\widetilde{\H}_{\alpha}(q)\boxtimes \widetilde{\H}_{\beta}(q)}\bigl(V\otimes W\bigr)\in\widetilde{\H}_{\alpha+\beta}(q)\mod ,\\
V^\theta\circ W^\theta &:=\ind_{\alpha,\beta}^{\alpha+\beta}V^\theta\boxtimes W^\theta=\widetilde{\mathscr{R}}_{\alpha+\beta}\otimes_{\widetilde{\mathscr{R}}_{\alpha}\boxtimes \widetilde{\mathscr{R}}_{\beta}}\bigl(V^\theta\otimes W^\theta\bigr)\in\widetilde{\mathscr{R}}_{\alpha+\beta}\mod ,\\
\end{aligned}$$
Then the commutative diagram in the previous paragraph implies that \begin{equation}\label{convol}
\Bigl(V\circ W\Bigr)^{\theta}\cong V^{\theta}\circ W^{\theta} .
\end{equation}
Similar statements apply to the categories $\widetilde{H}_{\alpha}\mod$, $\widetilde{\mathscr{R}}'_{\alpha}\mod$. With these results in mind, one can translate verbatim most of the results  in the representation theory of $\H_n$ (say, in \cite{G}, \cite{Varz}) into the results in the representation theory of $\mathscr{R}_n$ (say, in \cite{LV}) and vice versa.

Let $\H_n\in\{\H_n(q), H_n\}$. For any $(a_1,\dots,a_n)\in I^n$, following \cite{G}, \cite{K} and \cite{LV}, we define
$L(a_1,\dots,a_n):=\widetilde{f}_{a_n}\dots\widetilde{f}_{a_1}1$, where $1$ denotes the trivial irreducible module over $\H_0\cong K$, and $\widetilde{f}_k$ is defined as in \cite{G}. Then $L(a_1,\dots,a_n)$ is an irreducible module over $\H_n$. Two irreducible $\H_n$-modules $L(a_1,\dots,a_n), L(b_1,\dots,b_n)$ lie in the same block if and only if $(a_1,\dots,a_n)\sim (b_1,\dots,b_n)$, i.e., they differ by a permutation.
Note that in general a given irreducible module $L$ may be parameterized by several different tuples $(a_1,\dots,a_n)$.
By a similar procedure \cite{LV}, one can define the irreducible module $\widetilde{L}(a_1,\dots,a_n):=\widetilde{f}_{a_n}\dots\widetilde{f}_{a_1}1$ for the quiver Hecke algebra $\mathscr{R}_\beta$ for each $n$-tuple $(a_1,\dots,a_n)\in I^\beta$, where $\beta\in Q_n^+$.

\begin{dfn} \label{keyfn} Let $\alpha=\sum_{i\in I}l_i\alpha_i,\beta=\sum_{i\in I}k_i\alpha_i\in Q_n^{+}$. We say that $\alpha, \beta$ are {\bf weakly separated} if for any $1\leq i,j\leq n$ with $j-i\in\{1,-1\}$, either $l_i=0$ or $k_i=0$.
We say that $\alpha, \beta$ are {\bf separated} if for any $1\leq i,j\leq n$ with $j-i\in\{0,1,-1\}$, either $l_i=0$ or $k_i=0$.
\end{dfn}

The following result was mentioned in \cite[6.1.3]{K} as a remark in the degenerate setting. The full details of the proof are included in \cite{HZh}, where it gives an alternative approach to the famous Dipper--Mathas' Morita equivalence \cite{DM} for cyclotomic Hecke algebras.

\begin{lem} \label{morita} (cf. \cite[6.1.3]{K}, \cite{HZh}) Let $k\in\N$ and $n_1,\dots,n_k\in\N$ such that $\sum_{i=1}^kn_i=n$. Let $\beta_i\in Q_{n_i}^{+}$ for each $1\leq i\leq k$. Set $\beta:=\sum_{i=1}^{k}\beta_i$. Suppose that $\beta_1,\dots,\beta_k$ are pairwise separated, then there is an equivalence of categories: $$
(\H_n\mod)[\beta]\sim \bigl(\H_{n_1}\boxtimes\dots\boxtimes\H_{n_k}\bigr)\mod\,[\beta_1,\dots,\beta_k] .
$$
\end{lem}

As a second application of Theorem \ref{mainthm0a}, \ref{mainthm0b}, we get that

\begin{cor} \label{app1} Let $k\in\N$ and $n_1,\dots,n_k\in\N$ such that $\sum_{i=1}^kn_i=n$. Let $\beta_i\in Q_{n_i}^{+}$ for each $1\leq i\leq k$. Set $\beta:=\sum_{i=1}^{k}\beta_i$. Suppose that $\beta_1,\dots,\beta_k$ are pairwise separated, then there is an equivalence of categories: $$
\mathscr{R}_\beta\mod\sim\Bigl(\mathscr{R}_{\beta_1}\boxtimes\dots\boxtimes\mathscr{R}_{\beta_k}\Bigr)\mod .
$$
\end{cor}

\begin{proof} This follows from Lemma \ref{morita}, Theorem \ref{mainthm0a}, \ref{mainthm0b} and (\ref{convol}).
\end{proof}

We remark that the proof of Lemma \ref{morita} used certain intertwining elements of affine Hecke algebras introduced in \cite[Sect. 2]{Rog} and \cite[Sect. 5.1]{lusz}. Note that Kang, Kashiwara and Kim have introduced in \cite[(1.3.1)]{KKK} certain intertwiners inside the quiver Hecke algebras. However, it seems that one can not mimic the proof of Lemma \ref{morita} directly to get a proof of Corollary \ref{app1} inside the theory of quiver Hecke algebras because of the equality \cite[Lemma 1.3.1(i)]{KKK} (which only makes a difference for $\nu_a=\nu_{a+1}$ or $\nu_a\neq\nu_{a+1}$).

The degenerate case of the following result follows from \cite[6.1.4]{K} and an inductive argument. The non-degenerate case is similar. In both cases the argument used the categorical equivalence in Lemma \ref{morita}.

\begin{lem} \label{weakly} Let $k\in\N$ and $n_1,\dots,n_k\in\N$ such that $\sum_{i=1}^kn_i=n$. For each $1\leq i\leq k$, let $\beta_i\in Q_{n_i}^{+}$ and
$L(\underline{a}^{(i)})$ be an irreducible module over $\H_{n_i}$, where $\underline{a}^{(i)}=(a_1^{(i)},\dots,a_{n_i}^{(i)})\in I^{\beta_i}$. If for any $1\leq i\neq j\leq k$, $\underline{a}^{(i)}, \underline{a}^{(j)}$ are weakly separated, then $L(\underline{a}^{(1)})\circ\dots\circ L(\underline{a}^{(k)})$ is an irreducible module over $\H_n$.
\end{lem}

The following result is a third application of Theorem \ref{mainthm0a}, \ref{mainthm0b}.

\begin{cor} \label{app2} Let $k\in\N$ and $n_1,\dots,n_k\in\N$ such that $\sum_{i=1}^kn_i=n$. For each $1\leq i\leq k$, let $\beta_i\in Q_{n_i}^{+}$ and
$L(\underline{a}^{(i)})$ be an irreducible module over $\mathscr{R}_{\beta_i}$, where $\underline{a}^{(i)}=(a_1^{(i)},\dots,a_{n_i}^{(i)})\in I^{\beta_i}$. Set $\beta:=\sum_{i=1}^{k}\beta_i$. If for any $1\leq i\neq j\leq k$, $\underline{a}^{(i)}, \underline{a}^{(j)}$ are weakly separated, then  $\widetilde{L}(\underline{a}^{(1)})\circ\dots\circ \widetilde{L}(\underline{a}^{(k)})$ is an irreducible module over $\mathscr{R}_{\beta}$.
\end{cor}

\begin{proof} This follows from Lemma \ref{weakly} and (\ref{convol}).
\end{proof}

In particular, the above corollary gives a partial answer in type $A$ to the question raised in \cite[Problem 7.6(ii)]{KR}. It would be interesting to know whether the sufficient condition given in the above corollary is also necessary or not.\smallskip

\bigskip

\appendix
\def\theequation{\Alph{section}\arabic{equation}}
\def\theProposition{\Alph{section}\arabic{equation}}
\def\theLemma{\Alph{section}\arabic{equation}}
\def\theTheorem{\Alph{section}\arabic{equation}}
\def\theCorollary{\Alph{section}\arabic{equation}}

\section{The generalized Ore localization}

In this appendix, we want to generalize the classical construction of Ore localization with respect to a right denominator set to a more general situation as follows. For the classical theory of Ore localization of non-commutative ring, we refer the readers to \cite[\S10]{Lam} and \cite[\S1]{Stenstrom}).

Let $A$ be a (non-commutative) ring with identity $1$, $A_0$ be a commutative subring of $A$. Let $\{e_i|i\in J\}$ be a set of pairwise orthogonal idempotents of $A$ which sums to $1$, where $J$ is a finite indexing set. That says, $\sum_{i\in J}e_i=1$ and $e_ie_j=\delta_{ij}e_i$ for any $i,j$.
In particular, we have that $$
A=\bigoplus_{i,j\in J}e_jAe_i .
$$
We assume further that $fe_i=e_if$ for any $f\in A_0$ and $i\in J$. For each $i\in J$, let $S_i$ be a multiplicatively closed subset in $e_iA_0e_i$ with $e_i\in S_i$. We want to investigate certain generalized Ore conditions on $S_i$ under which the ring $A$ can be embedded into a larger ring $\widetilde{A}$ such that \begin{enumerate}
\item[(G1)] for any $i\in J$ and any $s\in S_i$, $s$ is an invertible element in the unital ring $e_i\widetilde{A}e_i$ (with identity element $e_i$); and
\item[(G2)] each element in $e_i\widetilde{A}e_j$ has the form $$
\sum_{i,j\in J}e_ia_{i,j}f_{i,j}^{-1}e_j,
$$
where $a_{i,j}\in A$, $f_{i,j}\in S_j$.
\end{enumerate}

\begin{lem} \label{OreLocal} With the notations and assumptions as above, and assume further that the subsets $\{S_i\}_{i\in J}$ satisfy the following two conditions: \begin{enumerate}
\item[(O1)] for any $g\in A$, $s\in S_i$, $$
se_ig=0\Longrightarrow e_ig=0,\quad\,\, ge_is=0\Longrightarrow ge_i=0 .
$$
In particular, $0\notin S_i$; and
\item[(O2)] for any $i,j\in J$ and any $a\in e_jAe_i$, $s\in S_i$ and $t\in S_j$, there exist some $b,c\in e_jAe_i$, $u\in S_j$ and $v\in S_i$, such that $ua=bs$, $av=tc$.
\end{enumerate}
then there exists a ring $\widetilde{A}$ together with an injective ring homomorphism $\varphi: A\hookrightarrow \widetilde{A}$ such that both (G1) and (G2) hold, and for any ring homomorphism $\psi: A\rightarrow B$ such that $\psi(s)$ is invertible in $\psi(e_i)B\psi(e_i)$ for any $s\in S_i$ and $i\in J$, then there is a unique ring homomorphism $\sigma: \widetilde{A}\rightarrow B$ such that $\sigma\varphi=\psi$. Moreover, if $\psi$ is injective then $\sigma$ is injective too.
\end{lem}

\begin{proof} We define $$
\widetilde{A}:=\bigoplus_{i,j\in J}\Bigl(e_jAe_i\times S_i\Bigr)/\!\sim_{ij},
$$
where `$\sim_{ij}$" is an equivalence relation in $e_jAe_i\times S_i$  defined as $(a,s)\sim_{ij} (b,t)$ if $at=bs$, where $a,b\in e_jAe_i$, $s,t\in S_i$. Denote by $[(a,s)]$ the equivalence class containing $(a,s)$.

We define the addition and multiplication in an obvious way: for any $a\in e_jAe_i$, $b\in e_kAe_l$, $s\in S_i$, $t\in S_l$:

1) In the case $i=l$, $[(a,s)]+[(b,t)]:=[(at+bs,st)]$; in the case $i\not=l$, $[(a,s)]+[(b,t)]$ means a formal sum;

2) $$
[(a,s)][(b,t)]:=\begin{cases} [(ac,tu)], &\text{if $i=k$\;\;\; where $bu=sc$, $u\in S_l$, $c\in e_iAe_l$;}\\
0, &\text{if $i\not=k$.}\end{cases}
$$
It is routine to check that the above definition is independent of the choice of the representing couples and $\widetilde{A}$ is a well-defined ring. The universal property of $\widetilde{A}$ follows from a similar (and more easy) argument as in the classical Ore localization (cf. \cite[Corollary 10.11]{Lam},\cite[Proposition 1.4]{Stenstrom}). Finally, assume that $\psi$ is injective. Suppose that $\sigma(z)=0$, where $$
z=\sum_{i,j\in J}a_{i,j}\varphi(f_{i,j})^{-1}\in \widetilde{A},\quad\,a_{i,j}\in e_iAe_j, f_{i,j}\in S_j,\,\,\forall\,i,j.
$$
Then for any $i,j\in J$, $$
\sigma\Bigl(a_{i,j}\varphi(f_{i,j})^{-1}\Bigr)=\sigma(e_ize_j)=\sigma(e_i)\sigma(z)\sigma(e_j)=0.
$$
It follows that $$
\psi(a_{i,j})=\psi(a_{i,j})\psi(f_{i,j})^{-1}\psi(f_{i,j})=\sigma\Bigl(a_{i,j}\varphi(f_{i,j})^{-1}\Bigr)\psi(f_{i,j})=0 .
$$
Since $\psi$ is injective, we can see that $a_{i,j}=0$ and hence $a_{i,j}\varphi(f_{i,j})^{-1}=0$ for each $i,j$, which implies that $z=0$ and hence $\sigma$ is injective. This completes the proof of the lemma.
\end{proof}

\begin{dfn} \label{GOL} With the notations and assumptions as in Lemma \ref{OreLocal}, we shall call the ring $\widetilde{A}$ the generalized Ore localization of $A$ with respect to the datum $(A_0,\{e_i\}_{i\in J}, \{S_j\}_{j\in J})$.
\end{dfn}


\bigskip










\begin{thebibliography}{2}


\bibitem{AK1}
{\sc S.~Ariki and K.~Koike}, {\em A Hecke algebra of $(\Z/r\Z)\wr S_n$ and construction of its irreducible
representations}, Adv. Math., {\bf 106} (1994), 216--243.



\bibitem{Brundan:degenCentre}
{\sc J.~Brundan}, {{\em  Centers of degenerate cyclotomic {H}ecke algebras and parabolic category
  {$\mathcal O$}}}, Represent. Theory, {\bf 12} (2008), 236--259.

\bibitem{BK:GradedKL}
{\sc J.~Brundan and A.~Kleshchev},
  {{\em Blocks of cyclotomic {H}ecke algebras and {K}hovanov-{L}auda algebras}}, Invent. Math., {\bf 178}
  (2009), 451--484.

\bibitem{Ch}
{\sc I.~Cherednik}, {\em A new interpretation of Gelfand-Tzetlin bases}, Duke Math. J., {\bf 54} (1987), 563--577.



\bibitem{DJ}
{\sc R.~Dipper, G.~James}, {\em Blocks and idempotents of Hecke algebras of general linear groups}, Proc. London Math. Soc.,
{\bf 54} (1987), 57--82.

\bibitem{DM}
{\sc R.~Dipper and A.~Mathas}, {\em Morita equivalences of Ariki--Koike algebras}, Math. Z., {\bf 240} (2002), 579--610.

\bibitem{FG}
{\sc A.R.~Francis, J.J.~Graham}, {{\em Centres of Hecke algebras: The Dipper--James conjecture}}, J. Alg., {\bf 306} (2006), 244--267.


\bibitem{G}
{\sc I.~Grojnowski}, {\em Affine $\widehat{\mathfrak{sl}}_p$ controls the modular representation theory of the symmetric group and related Hecke algebras}, preprint, math.RT/9907129, 1999.

\bibitem{HuMathas:GradedCellular}
J. Hu and A.~Mathas, Graded cellular bases for the cyclotomic {K}hovanov-{L}auda-{R}ouquier algebras of type {$A$}, Adv. Math. {\bf 225} (2010) 598--642.


\bibitem{HZh}
{\sc J.~Hu and K.~Zhou}, {\em On Dipper--Mathas's Morita equivalences}, Colloquium Mathematicum, {\bf 149}(1) 2017, 103--123.


\bibitem{Kac}
{\sc V.~G. Kac}, {{\em Infinite-dimensional {L}ie algebras}}, Cambridge University Press, Cambridge,   third~ed., 1990.

\bibitem{KKK}
{\sc S.~J. Kang, M.~Kashiwara and M.~Kim}, {\em Symmetric quiver Hecke algebras and R-matrices of quantum affine algebras}, preprint, arXiv:1304.0323v2, 2014.


\bibitem{KhovLaud:diagI}
{\sc M.~Khovanov and A.~D. Lauda}, {{\em A diagrammatic
  approach to categorification of quantum groups, {I}}}, Represent. Theory,
  {\bf 13} (2009), 309--347.

\bibitem{K}
{\sc A.~Kleshchev}, {\em Linear and projectve representations of symmetric groups}, Cambridge University Press, 2005.

\bibitem{K2010}
{\sc A.~Kleshchev}, {\em Representation theory of symmetric groups and related Hecke algebras}, Bull. Amer. Math. Soc. (N.S.), {\bf 47} (2010), 419--481.

\bibitem{KR}
{\sc A.~Kleshchev and A.~Ram}, {\em Representations of Khovanov-Lauda-Rouquier algebras and combinatorics of Lyndon words}, Math. Ann., {\bf 349} (2011), 943--975.

\bibitem{Lam}
{\sc T.Y.~Lam}, {\em Lectures on modules and rings}, Graduate Texts in Mathematics, {\bf 189} Springer, 1998.

\bibitem{LV}
{\sc A.~Lauda and M.~Vazirani}, {\em Crystals from categorified quantum groups}, Adv. Math., {\bf 228} (2011), 803--861.

\bibitem{Li}
{\sc G.~Li}, {\em The idempotents in cyclotomic Hecke algebras and periodic property of the Jucys-Murphy elements}, preprint,
arXiv:1211.1754, 2012.

\bibitem{LM:AKblocks}
{\sc S.~Lyle and A.~Mathas}, {\em Blocks of cyclotomic {H}ecke algebras}, Adv.
  Math., {\bf 216} (2007), 854--878.

\bibitem{lusz}
{\sc G.~Lusztig}, {\em Affine Hecke algebras and their graded version}. J. Amer. Math. Soc., {\bf 2} (1989), 599--635.

\bibitem{Lusz2}
{\sc G.~Lusztig}, {\em Introduction to quantum groups}, Progress in
  Mathematics, {\bf 110}, Birkh\"auser Boston Inc., Boston, MA, 1993.

\bibitem{Mac}
{\sc I.~G. Macdnoald}, {\em Affine Hecke algebras and orthogonal polynomials}, Cambridge Tracts in Mathematics, {\bf 157}, Cambridge University Press, 2003.


\bibitem{M2015}
{\sc A.~Mathas}, {\em Cyclotomic quiver Hecke algebras of type A}, in: Modular Representation Theory of Finite and $p$-adic Groups
Lecture Notes Series, Institute for Mathematical Sciences, National University of Singapore: Volume {\bf 30}, World Scientific, 2015.

\bibitem{Mc}
{\sc K.~McGerty}, {\em On the center of the cyclotomic Hecke algebra of $G(m,1,2)$}, Proceedings of the Edinburgh Mathematical Society, {\bf 55}  (2012), 497--506.


\bibitem{Mur}
{\sc G.E.~Murphy}, {\em The idempotents of the symmetric group and Nakayama¡¯s conjecture}, J. Algebra, {\bf 81} (1983), 258--265.


\bibitem{Rog}
{\sc J.D.~Rogawski}, {\em On modules over the Hecke algebra of a $p$-adic group}, Invent. Math., {\bf 79} (1985), 443--465.

\bibitem{Rou0}
{\sc R.~Rouquier}, {\em Quiver Hecke algebras and 2-Lie algebras}, Algebra Colloq., {\bf 19} (2012), 359--410.


\bibitem{Rou1}
{\sc R.~Rouquier}, {\em 2-Kac--Moody algebras}, preprint, math.RT/0812.5023v1, 2008.

\bibitem{Stenstrom}
{\sc B.~Stenstr\"{o}m}, {\em Rings of Quotients, an introduction to methods of ring theory}, Springer-Verlag, Berlin, 1975.


\bibitem{StroppelWebster:QuiverSchur}
{\sc C.~Stroppel and B.~Webster}, {\em Quiver {S}chur algebras and $q$-{F}ock
  space}, 2011, preprint, arXiv:1110.1115.

\bibitem{Varz}
{\sc M.~Vazirani}, {\em Irreducible modules over the affine Hecke algebra: a strong multiplicity one result}, PhD thesis, UC
Berkeley, 1999.




\end{thebibliography}
\end{document}